\documentclass{amsart} 
\usepackage{amsthm, amsmath, amssymb, amscd, amsthm, amsfonts, amstext, amsbsy, enumerate, hyperref, mathrsfs, mathtools%, upgreek, stmaryrd
}
\usepackage[shadow]{todonotes}

\renewcommand{\L}{\mathsf{L}}
\newcommand{\W}{\mathsf{W}}
\renewcommand{\c}{\mathsf{c}}
\newcommand{\Lip}{\mathsf{Lip}}
\newcommand{\UCont}{\mathsf{UCont}}

\newcommand{\Bor}{\mathsf{Bor}}
\newcommand{\C}{\mathcal{C}}

\newcommand{\F}{\mathcal{F}}
\newcommand{\RR}{\mathbb{R}}
\newcommand{\ZZ}{\mathbb{Z}}

\newcommand{\pre}[2]{{}^{#1} #2}

\newcommand{\D}{\mathsf{D}}

\newcommand{\id}{\operatorname{id}}

\newcommand{\pow}{\mathscr{P}}

\newcommand{\p}{\operatorname{p}}
\newcommand{\leng}{\operatorname{lh}}
\newcommand{\rg}{\operatorname{rg}}

\newcommand{\Deg}{{\sf Deg}}

\newcommand{\AD}{{\sf AD}}

\newcommand{\ZF}{{\sf ZF}}
\newcommand{\AC}{{\sf AC}}
\newcommand{\DC}{{\sf DC}}
\newcommand{\SLO}{{\sf SLO}}

\newcommand{\Nbhd}{\boldsymbol{\mathrm{N}}}

\theoremstyle{theorem}
\newtheorem{theorem}{Theorem}[section]
\newtheorem{lemma}[theorem]{Lemma}

\newtheorem{claim}{Claim}[theorem]

\newtheorem{main lemma}[theorem]{Main Lemma}
\newtheorem{corollary}[theorem]{Corollary}

\newtheorem{proposition}[theorem]{Proposition}
\newtheorem{question}[theorem]{Question}

\theoremstyle{definition}
\newtheorem{definition}[theorem]{Definition}

\theoremstyle{remark}
\newtheorem{remark}[theorem]{Remark}

\newtheorem{notation}[theorem]{Notation}
\newtheorem{example}[theorem]{Example}
  
\begin{document}

\title[Lipschitz and uniformly continuous reducibilities]{Lipschitz and uniformly continuous reducibilities on ultrametric Polish spaces} 
\date{\today} 
\author{Luca Motto Ros}
\address{Albert-Ludwigs-Universit\"at Freiburg \\
Mathematisches Institut -- Abteilung f\"ur Mathematische Logik\\
Eckerstra{\ss}e, 1 \\
79104 Freiburg im Breisgau\\
Germany}
\email{luca.motto.ros@math.uni-freiburg.de}
\author{Philipp Schlicht} 
\address{Rheinische Friedrich-Wilhelms-Universit\"at Bonn \\
Mathematisches Institut \\
Endenicher Allee, 60 \\ 
53115 Bonn \\ 
Germany}
\email{schlicht@math.uni-bonn.de}
\thanks{The authors would like to congratulate Professor Victor Selivanov on the occasion of
his sixtieth birthday for his wide and important contributions to mathematical
logic and, in particular, to the theory of Wadge-like reducibilities and its connections with theoretical computer science.}

\begin{abstract} 
We analyze the reducibilities induced by, respectively, uniformly continuous, Lipschitz, and nonexpansive functions on arbitrary ultrametric 
Polish spaces, and determine whether under suitable set-theoretical assumptions the induced degree-structures are well-behaved.
\end{abstract} 

\maketitle

% Please add your affiliation here as a comment.
% it will be printed on a separate page in the book

%\author{Luca Motto Ros}
%\address{Albert-Ludwigs-Universit\"at Freiburg \\
%Mathematisches Institut -- Abteilung f\"ur Mathematische Logik\\
%Eckerstra{\ss}e, 1 \\
%79104 Freiburg im Breisgau\\
%Germany}
%\email{luca.motto.ros@math.uni-freiburg.de}
%\author{Philipp Schlicht} 
%\address{Rheinische Friedrich-Wilhelms-Universit\"at Bonn \\
%Mathematisches Institut \\
%Endenicher Allee, 60 \\ 
%53115 Bonn \\ 
%Germany}
%\email{schlicht@math.uni-bonn.de}

%\keywords{Wadge reducibility, continuous reducibility, Lipschitz reducibility, uniformly continuous reducibility, ultrametric Polish space, 
%nonexpansive function, Lipschitz function, uniformly continuous function}
%\subjclass[2010]{03E15, 03E60, 54C10, 54E40} 

%\altmaketitle

%\begin{abstract} 
%We analyze the reducibilities induced by, respectively, uniformly continuous, Lipschitz, and nonexpansive functions on arbitrary ultrametric 
%Polish spaces, and determine whether under suitable set-theoretical assumptions the induced degree-structures are well-behaved.
%\end{abstract} 

\section{Introduction}
Throughout the paper, we work in the usual Zermelo-Fr\ae{}nkel set theory \( \ZF \), plus the Axiom of Dependent Choices over the reals \( \DC(\RR) \). 
Let \( X \) be a Polish space, and let \( \F \) be a \emph{reducibility (on \( X \))}, that is a collection of functions from \( X \) to itself closed under composition and containing the identity  \( \id = \id_X \). Given \( A,B \subseteq X \), we say that \( A \) is \emph{reducible} to \( B \)  if and only if
\[ 
A = f^{-1}(B) \text{ for some } f \colon X \to X,
 \] 
and that \( A \) is \emph{\( \F \)-reducible} to \( B \) (\( A \leq_\F B \) in symbols) if \( A \) is reducible to \( B \) \emph{via a function in \( \F \)}.
Notice that clearly \( A \leq_\F B \iff \neg A \leq_\F \neg B \) (where, to simplify the notation, we set \( \neg A = X \setminus A \) whenever the 
underlying space \( X \) is clear from the context).
Since \( \F \) is a reducibility on \( X \), the relation \( \leq_\F \) is a preorder which can be used to measure the ``complexity'' of subsets 
of \( X \): in fact, if \( \F \) consists of reasonably simple functions, the assertion ``\( A \leq_\F B  \)'' may be understood as ``the set \( A \) is not more
 complicated than the set \( B \)'' --- to test whether a given \( x \in X \) belongs to \( A \) or not, it is enough to pick a witness \( f \in \F \) of 
\( A \leq_\F B \), and then check whether \( f(x) \in B \) or not. This suggests that the reducibility \( \F \) may be used to form a hierarchy of subsets of \( X \) 
 in the following way. Say that  \( A , B \subseteq X \) are \emph{\( \F \)-equivalent}
 (\( A \equiv_\F B \) in symbols) if \( A \leq_\F B \leq_\F A \). Since \( \equiv_\F \) is the equivalence relation canonically induced by \( \leq_\F \), we can
 consider the \emph{\( \F \)-degree} \( [A]_\F = \{ B \subseteq X \mid A \equiv_\F B \} \) of a given \( A \subseteq X \), and then order the 
collection \( \Deg(\F) = \{ [A]_\F \mid A \subseteq X \} \) of such \( \F \)-degrees using
 the quotient of \( \leq_\F \), namely setting \( [A]_\F \leq [B]_\F \iff A \leq_\F B \) for every \( A, B \subseteq X \). The resulting structure \( \Deg(\F) = (\Deg(\F),\leq) \) is then called \emph{\( \F \)-hierarchy on \( X \)}. When considering the restriction 
\( \Deg_{\boldsymbol{\Gamma}}(\F) \) of such structure to the \( \F \)-degrees of sets in a given \( \boldsymbol{\Gamma} \subseteq \pow(X) \), we 
speak of \emph{\( \F \)-hierarchy on \( \boldsymbol{\Gamma} \)-subsets of \( X \)}.

In his Ph.D. thesis~\cite{Wadge:1983}, Wadge considered the case when \( X \) is the Baire space \( \pre{\omega}{\omega} \) (i.e.\ the space 
of all \(\omega\)-sequences of natural numbers endowed with the product of the discrete topology on \(\omega\)) and \( \F \) is either the 
set \( \mathsf{W} = \W(X) \) of all continuous functions, or the set \( \L(\bar{d}) \) of all  functions which are nonexpansive with respect to the usual metric 
\( \bar{d} \) on \( \pre{\omega}{\omega} \) (see Section~\ref{sec:ultrametric} for the definition). Using game-theoretical methods, he was able to show 
that in both cases the \( \F \)-hierarchy on Borel subsets of \( X = \pre{\omega}{\omega} \) is \emph{semi-well-ordered}, that is:
\begin{enumerate}[(1)]
\item
it is \emph{semi-linearly ordered}, i.e.\ either \( A \leq_\F B \) or \( \neg B \leq_\F A \) for all Borel \( A,B \subseteq X \);
\item
it is \emph{well-founded}.
\end{enumerate}

Notice that the \emph{Semi-Linear Ordering principle for \( \F \)} (briefly: \( \SLO^\F \)) defined in (1) implies that antichains have size at 
most \( 2 \), and that they are of the form \( \{ [A]_\F, [\neg A]_\F \} \) for some \( A \subseteq X \) such that \( A \nleq_\F \neg A \) (sets with 
this last property are called \emph{\( \F \)-nonselfdual}, while the other ones are called \emph{\( \F \)-selfdual}: since \( \F \)-selfduality is 
\( \equiv_\F \)-invariant, a similar terminology will be applied to the \( \F \)-degree of \( A \) as well). This in particular means that if we further 
identify each \( \F \)-degree \( [A]_\F \) with its \emph{dual} \( [\neg A]_\F \) we get a linear ordering, which is also well-founded when (2) holds.

A semi-well-ordered hierarchy is practically optimal as a measure of complexity for (Borel subsets of) \( X \): by well-foundness, we can associate 
to each \( A \subseteq X \) an ordinal rank (the \emph{\( \F \)-rank of \( A \)}), and antichains are of minimal size.%
\footnote{Asking for no antichain at all seems unreasonable by the following considerations: let \( A \) be e.g.\ a proper open subset of a given 
Polish space \( X \). On the one hand, checking membership in \( A \) cannot be considered strictly simpler or strictly more difficult than checking
 membership in its complement: this means that the degrees of \( A \) and \( \neg A \) cannot be one strictly  below the other in the hierarchy. On the other
 hand, the fact that open sets and closed sets have in general different (often complementary) combinatorial and topological properties, strongly
 suggests that the degrees of \( A \) and \( \neg A \) should be kept distinct. Therefore such degrees must form an antichain of size \( 2 \).}
In fact, in~\cite{MottoRos:2012b, MottoRos:2012c} it is proposed to classify arbitrary \( \F \)-hierarchies on corresponding topological spaces \( X \)
 according to whether they provide a good measure of complexity for subsets of \( X \). This led to the following definition.

\begin{definition}
Let \( \F \) be a reducibility on a (topological) space \( X \), and let \( \boldsymbol{\Gamma} \subseteq \pow(X) \). The \( \F \)-hierarchy \( \Deg_{\boldsymbol{\Gamma}}(\F) \) on \( \boldsymbol{\Gamma} \)-subsets of \( X \) is called:
\begin{enumerate}[\( \bullet \)]
\item
\emph{very good} if it is semi-well-ordered;
\item
\emph{good} if it is a well-quasi-order, i.e.\ all its antichains and descending chains are finite;
\item
\emph{bad} if it contains infinite antichains;
\item
\emph{very bad} if it contains both infinite antichains and infinite descending chains.
\end{enumerate}
\end{definition}

Since the pioneering work of Wadge, many other \( \F \)-hierarchies on the Baire space \( \pre{\omega}{\omega} \) (or, more generally, on
 \emph{zero-dimensional} Polish space) have been considered in the
 literature~\cite{VanWesep:1978,Andretta:2003d,Andretta:2006,MottoRos:2009,MottoRos:2010,MottoRos:2010b}, including Borel functions, 
\( \boldsymbol{\Delta}^0_\alpha \)-functions,%
\footnote{Given a countable ordinal \( \alpha \geq 1 \) and a Polish space \( X \), a function \( f \colon X \to X \) is called 
\emph{\( \boldsymbol{\Delta}^0_\alpha \)-function} if \( f^{-1}(A) \in \boldsymbol\Sigma^0_\alpha \) for every \( A \in \boldsymbol{\Sigma}^0_\alpha \).}
 Lipschitz functions, uniformly continuous functions, functions of Baire class \( < \alpha \) for a given additively closed countable ordinal \(\alpha\), 
\( \boldsymbol{\Sigma}^1_n \)-measurable functions, and so on. It turned out that all of them are very good when restricted to Borel sets, or even to larger collections of subsets of \( \pre{\omega}{\omega} \) 
if suitable determinacy principles are assumed. 
In contrast, it is shown in~\cite{Hertling:1993,Hertling:1996,Ikegami:2012,Schlicht:2012,MottoRos:2012b} that when considering the continuous
 reducibility  on the real line \( \RR \) or, more generally, on arbitrary Polish spaces with \emph{nonzero dimension}, then one usually gets a (very) 
bad hierarchy (and the same applies to some other classical kind of reducibilities, depending on the space under consideration).% 
\footnote{Of course, one can further extend the class of topological spaces under consideration, and analyze e.g.\ the continuous reducibility on them: for example,~\cite{Selivanov:2005} considers the case of \( \omega \)-algebraic domains (a class of spaces relevant in theoretical computer 
science), while~\cite{MottoRos:2012b} consider the broader class of the so-called quasi-Polish spaces. Moreover, it is possible to generalize 
the notion of reducibility itself by considering e.g.\ reducibilities between finite partitions (see
 e.g.~\cite{vanEngelen:1987,Hertling:1993,Selivanov:2005,Selivanov:2007,Selivanov:2010} and the references contained therein).}

Given all these results, one may be tempted to conjecture that all ``natural'' \( \F \)-hierarchies on (Borel subsets of) a  zero-dimensional Polish 
space \( X \) need to be very good. This conjecture is justified by the fact that every such space is homeomorphic to a closed subset (hence to a
 topological retract) of the Baire space, and a well-known transfer argument (see e.g.\ \cite[Proposition 5.4]{MottoRos:2012b}) shows that this 
already implies the following folklore result.

\begin{proposition} \label{prop:continuous}
Let \( X \) be a zero-dimensional Polish space, and let \( \F \) be an arbitrary reducibility on \( X \) which contains \( \W(X) \), i.e.\ all continuous functions from 
\( X \) to itself. Then the \( \F \)-hierarchy \( \Deg_{\boldsymbol{\Delta}^1_1}(\F) \) on Borel subsets of \( X \) is very good.
\end{proposition}

\noindent
In fact,~\cite[Theorem 3.1]{MottoRos:2009} (essentially) shows that this result can be further strengthened when \( X \) itself is a closed subset 
of \( \pre{\omega}{\omega} \): if \( X \) is equipped with the restriction \( \bar{d}_X \) of the canonical metric \( \bar{d} \) on 
\( \pre{\omega}{\omega} \), then  \( \Deg_{\boldsymbol{\Delta}^1_1}(\F) \) is very good as soon as \( \F \) contains the collection \( \L(\bar{d}_X) \) of all 
\emph{\( \bar{d}_X \)-nonexpansive} functions.

Despite the above mentioned results, in~\cite[Theorem 5.4, Proposition 5.10, and Theorem 5.11]{MottoRos:2012c} it is shown that there are 
various natural reducibilities on \( \pre{\omega}{\omega} \) that actually induce (very) bad hierarchies on its Borel subsets. In particular, it is 
shown that
\( \pre{\omega}{\omega} \) can be equipped with a complete ultrametric \( d' \), still compatible with its usual product topology, such that
the \( \F \)-hierarchy on Borel (in fact, even just clopen) subsets of \( \pre{\omega}{\omega} \) is very bad for \( \F \) the collection of all the 
\( d' \)-nonexpansive (alternatively: \( d' \)-Lipschitz) functions.

%\begin{enumerate}[a)]
%\item
%if \( \F \) is either the collection of all the \( d_0 \)-nonexpansive functions or the collection of all \( d_0 \)-Lipschitz functions, then the \( \F \)-hierarchy %on Borel (in fact, on clopen) subsets of \( X \) is very bad;
%\item
%if \( \F \) is the collection of all \( d_1 \)-nonexpansive functions, then the \( \F \)-hierarchy on Borel (in fact, on clopen) subsets of \( X \) is very bad, while the \( \F \)-hierarchy on Borel subsets of \( X \) is instead very good if \( \F \) is the collection of all \( d_1 \)-Lipschitz functions.
%\end{enumerate}

Motivated by these results, in the present paper we continue this investigation by considering various complete ultrametrics on 
\( \pre{\omega}{\omega} \) (compatible with its product topology) and, more generally, the collection of all \emph{ultrametric Polish spaces \( X = (X,d) \)}, a very natural and interesting class which includes e.g.\ the space \( \mathbb{Q}_p \) of \( p \)-adic numbers (for every prime \( p \in \mathbb{N} \)).%
\footnote{More generally, the completion of any countable valued field \(K \) with valuation \( | \cdot |_K \colon K \to \RR \) and metric \( d(x,y) = | x - y |_K \) (for \( x,y \in K \)) is always an ultrametric Polish space.}
On 
such spaces, we then consider the hierarchies of degrees induced by one of the following reducibilities%
\footnote{Notice that since the metric topology on \( X \) is always zero-dimensional, it does not make much sense to consider reducibilities 
\( \F \supseteq \W(X) \), because  by Proposition~\ref{prop:continuous} they always induce a very good hierarchy on Borel 
subsets of \( X \).}
 on \( X \):
\begin{enumerate}[\( \bullet \)]
\item
the collection \( \L(d) \) of all nonexpansive functions, where \( f \colon X \to X \) is called \emph{nonexpansive} if \( d(f(x),f(y)) \leq d(x,y) \) 
for all \( x,y \in X \);
\item
the collection \( \Lip(d) \) of all Lipschitz functions (with arbitrary constants), where \( f \colon X \to X \) is a \emph{Lipschitz function with constant 
\( L \)} (for a nonnegative real \( L \)) if \( d(f(x),f(y)) \leq L \cdot d(x,y) \) for all  \( x,y \in X \);
\item
the collection \( \UCont(d) \) of all uniformly continuous functions, where \( f \colon X \to X \) is \emph{uniformly continuous} if for every  
\( \varepsilon \in \RR^+ \) there is a  \( \delta \in \RR^+ \) such that \( d(x,y) < \delta \Rightarrow d(f(x),f(y)) < \varepsilon \) for all \( x,y \in X \) 
(here \( \RR^+ \) denotes the set of strictly positive reals).
\end{enumerate}

The main results of the paper are the following:

\begin{enumerate}[(A)]
\item
The \( \UCont(d) \)-hierarchy on Borel subsets of \( X \) is always very good (Theorem~\ref{th:unifcontandlip}). Since by Proposition~\ref{prop:UCont} it is possible to equip the Baire
 space with a compatible complete ultrametric \( d' \) such that \( \L(\bar{d}) \not\subseteq \UCont(d') \) (where $\bar{d}$ is the usual metric on \( \pre{\omega}{\omega} \)), this also
 implies that \( \L(\bar{d}) \subseteq \F \) is a sufficient but not necessary condition for the \( \F \)-hierarchy on Borel subsets of 
\( \pre{\omega}{\omega} \) being very good (for \( \F \) a reducibility on \( \pre{\omega}{\omega} \)).
\item
If \( X \) is perfect, then the \( \Lip(d) \)-hierarchy on the Borel subsets of \( X \) is either very good (if \( X \) has bounded diameter), or else it is very 
bad already when restricted to clopen subsets of \( X \) (if the diameter of \( X \) is unbounded). A technical strengthening of the property of having
 (un)bounded diameter (see Definition~\ref{def:nontriviallyunbounded}) works similarly for arbitrary ultrametric Polish spaces
 (Theorems~\ref{th:unboundeddiam} and~\ref{th:notunbounded}, Corollary~\ref{cor:unboundeddiam}).
\item
If the range of \( d \) contains an honest increasing sequence (see Definition~\ref{def:honestincreasingsequence}), then the \( \L(d) \)-hierarchy on
 clopen subsets of \( X \) is very bad (Theorem~\ref{th:increasingdistances}); in particular, this happens in the special case when \( X \) is perfect and
 has unbounded diameter. If instead the range of \( d \) is either finite or a decreasing \(\omega\)-sequence converging to \( 0 \), then the 
\( \L(d) \)-hierarchy on Borel subsets of \( X \) is always very good (Theorem~\ref{th:descendingdistances}).
\item
It follows from the second part of (C) that if \( X \) is compact, then both%
\footnote{Since on compact metric spaces continuity and uniform continuity coincide, the \( \UCont(d) \)-hierarchy on Borel subsets of a compact 
\( X \) is very good already by Proposition~\ref{prop:continuous}.}
 the \( \Lip(d) \)- and the \( \L(d) \)-hierarchy on Borel subsets of \( X \) are very good (Theorem~\ref{th:compact}).
\item
If we assume the Axiom of Choice \( \AC \), then  the \( \F \)-hierarchy on (arbitrary subsets of) an uncountable \( X \) is very bad for every reducibility 
\( \F \) such that \( \L(d) \subseteq \F \subseteq \Bor (X) \), where \( \Bor(X) \) is the collection of all Borel functions from \( X \) into itself
 (Theorem~\ref{th:illfounded hierarchy}). If we further assume that \( \mathsf{V = L} \), then the \( \F \)-hierarchy on \( X \) is very bad already when
 restricted to
 \( \boldsymbol{\Pi}^1_1 \), i.e.\ coanalytic,%
\footnote{Equivalently, to  \( \boldsymbol{\Sigma}^1_1 \) (i.e.\ analytic) subsets of \( X \).}
 subsets of \( X \) (Theorem~\ref{th:illfounded hierarchy in L}).
\end{enumerate}

In particular, the results in (A)--(D) generalize those from~\cite[Section 5]{MottoRos:2012c} and answer most of the questions 
in~\cite[Section 6]{MottoRos:2012c}. Moreover, they allow us to construct discrete ultrametric Polish spaces \( X = (X,d) \) whose \( \Lip(d) \)- and 
\( \L(d) \)-hierarchies are very bad (Corollaries~\ref{cor:countable} and~\ref{cor:X_1}), a fact which contradicts the conceivable conjecture that 
the \( \Lip(d) \)- and the \( \L(d) \)-hierarchy on them need to be (very) good since all subsets of such spaces are extremely simple (i.e.\ clopen). 
Notice also that the result mentioned in (E) under the assumption \( \mathsf{V = L} \) (which is best possible for most reducibilities \( \F \) by
 Proposition~\ref{prop:continuous} and the comment following it) 
can be viewed as an extension of the well-know classical result that if \( \boldsymbol{\Pi}^1_1 \)-determinacy fails then there are proper \( \boldsymbol{\Pi}^1_1 \) sets which are not (Borel-)complete for coanalytic sets.

We end this introduction with two general remarks concerning the results presented in this paper:

\begin{enumerate}[i)]
\item
to simplify the presentation, we will consider only \( \F \)-hierarchies on \emph{Borel subsets} of a given ultrametric Polish space \( X \) (except in Section \ref{sec: choice}): this is 
because in this way we can avoid to assume any axiom beyond our basic theory \( \ZF + \DC(\RR) \). However, as usual in Wadge theory, all our 
results can be extended to larger pointclasses \( \boldsymbol{\Gamma} \subseteq \pow(X) \) by assuming corresponding determinacy axioms 
(more precisely: the determinacy of subsets of \( \pre{\omega}{\omega} \) which are Boolean combinations of sets in \( \boldsymbol{\Gamma} \)). 
In particular, under the full Axiom of Determinacy \( \AD \) (asserting that all games on \(\omega\) are determined), all these results remain true 
when considering \emph{unrestricted} \( \F \)-hierarchies \( \Deg(\F) \) on \( X \);
\item
when showing that a given \( \F \)-hierarchy on \( X \) (possibly restricted to some \( \boldsymbol{\Gamma} \subseteq \pow(X) \)) is very bad, we will
 actually show that some very complicated partial (quasi-)order on \( \pow(\omega) \), like the inclusion relation \( \subseteq \), or even the more
 complicated relation \( \subseteq^* \) of inclusion modulo finite sets, embeds into such a hierarchy. This gives much stronger results, as it implies 
e.g.\ that the \( \F \)-hierarchy under consideration contains antichains of size the continuum and, in the case of \( \subseteq^* \), that (under 
\( \AC \)) every partial order of size \( \aleph_1 \) embeds into the \( \F \)-hierarchy on (\( \boldsymbol{\Gamma} \)-subsets of) \( X \) (see~\cite{Parovivcenko:1963}).
\end{enumerate}

\section{Basic facts about ultrametric Polish spaces} \label{sec:ultrametric}

Given a metric space \( X = (X,d) \), we denote by \( \tau_d \) the \emph{metric topology (induced by \( d \))}, i.e.\ the topology generated by the 
basic open balls \( B_d(x,\varepsilon) = \{ y \in X \mid d(x,y) < \varepsilon  \} \) (for some \(x \in X \) and \( \varepsilon \in \RR^+ \)). When considered 
as a topological space, the space \( X \) is tacitly endowed with such topology, and therefore we will e.g.\ say that the metric space \( X \) is separable 
if there is a countable \( \tau_d \)-dense subset of \( X \), and similarly for all other topological notions. The diameter of \( X \) is \emph{bounded} if
 there is \( R \in \RR^+ \) such that \( \sup \{ d(x,y) \mid x,y \in X \} \leq R \), and \emph{unbounded} otherwise.

A metric \( d \) on a space \( X \) is called \emph{ultrametric} if it satisfies the following strengthening of the triangle inequality, for all \( x,y,z \in X \):
\[ 
d(x,z) \leq \max \{ d(x,y), d(y,z) \}.
 \] 

\begin{definition}
An \emph{ultrametric Polish space} is a separable metric space \( X = (X,d) \) such that \( d \) is a complete ultrametric. The collection of all 
ultrametric Polish spaces will be denoted by \( \mathscr{X} \).
\end{definition}

Every (\( \tau_d \)-)closed subspace \( C \) of an ultrametric Polish space \( X = (X,d) \) will be tacitly equipped with the metric 
\( d_C = d \restriction C \), which is obviously a complete ultrametric compatible with the relative topology on \( C \) induced by \( \tau_d \). 
When there is no danger of confusion, with a little abuse of notation the metric \( d_C \) will be sometimes denoted by \( d \) again.

\begin{notation}
Given an ultrametric Polish space \( X = (X,d) \), we set \( R(d) = \{ d(x,y) \mid x,y \in X, x \neq y \} \), the set of all nonzero distances realized in \( X \).
\end{notation}%

A typical example of an ultrametric Polish space is obtained by equipping the Baire space with the usual metric \( \bar{d} \) defined by 
\[ 
\bar{d}(x,y) = 
\begin{cases}
0 & \text{if } x = y \\
2^{-n} & \text{if \( n \) is smallest such that } x(n) \neq y(n):
\end{cases}
 \] 
it is straightforward to check that \( \bar{d} \) is actually an ultrametric generating the product topology on \( \pre{\omega}{\omega} \), and 
obviously \( R(\bar{d}) = \{ 2^{-n} \mid n \in \omega \} \). We will keep denoting this ultrametric by \( \bar{d} \) throughout the paper.

We collect here some easy but useful facts about arbitrary ultrametric (Polish) spaces \( X = (X,d) \):
\begin{enumerate}[(1)]
\item
for every \( x,y,z \in X \) two of the distances \( d(x,y) \), \( d(x,z) \), \( d(y,z) \) are equal, and they are greater than or equal to the third 
(the ``isosceles triangle'' rule);
\item
for every \( x,y,z \in X \), if \( d(x,z) \neq d(y,z) \) then \( d(x,y) = \max \{ d(x,z),\) \(d(y,z) \} \). In particular, if \( x,y,z,w \in X \) are such that 
\( d(x,z),d(y,w) < d(x,y) \) then \( d(z,w) = d(x,y) \);
\item
given a (\( \tau_d \)-)dense set \( Q \subseteq X \), all distances are realized by elements of \( Q \), that is: for every \( x,y \in X \) there are 
\( q,p \in Q \) such that \( d(x,y) = d(q,p) \). In particular, if \( X \) is separable then \( R(d) \) is countable;%
\footnote{Vice versa, for every countable \( R \subseteq \RR^+  \) there is an ultrametric Polish space \( X = (X,d) \) such that \( R(d) = R  \), for example \( X=R\cup\{0\} \) with \( d(x,y)=\max\{x,y\} \) for distinct \( x , y \in X \).}
\item
for every \( x \in X \) and \( r \in \RR^+ \) the open ball \( B_d(x,r)  \) is actually clopen, and \( B_d(y,r) = B_d(x,r) \) for every \( y \in B_d(x,r) \). In
 particular,  the topology \( \tau_d \) is always zero-dimensional, and hence if \( X \) is an ultrametric Polish space, then it is homeomorphic to a 
closed subset of the Baire space by~\cite[Theorem 7.8]{Kechris:1995} (see also Lemma~\ref{lemma:unifcontandlip});  
\item
given \( x,y \in X \) and \( r,s \in \RR^+ \), the (cl)open balls \( B_d(x,r) \) and \( B_d(y,s) \) are either disjoint, or else one of them contains the other.
\end{enumerate}

To simplify the terminology, we adapt the definition of family of reducibilities introduced in~\cite[Definition 5.1]{MottoRos:2012b} to the restricted
 context of ultrametric Polish spaces.

\begin{definition}
Let \( \F \) be a collection of functions between any ultrametric Polish spaces. For \( X,Y \in \mathscr{X} \), denote by \( \F(X,Y) \) the collection of all
 functions from \( \F \) with domain \( X \) and range included in \( Y \). The collection \( \F \) is called \emph{family of reducibilities} (on 
\( \mathscr{X} \)) if:
\begin{enumerate}
\item
it contains all the identity functions, i.e.\ \( \id_X \in \F(X,X) \) for every \( X \in \mathscr{X} \);
\item
it is closed under composition, i.e.\ for every \( X,Y,Z \in \mathscr{X} \), \( f \in \F(X,Y) \), and \( g \in \F(Y,Z) \), the function \( g \circ f \) belongs to
\( \F(X,Z) \);
%\item
%it is closed under restrictions, i.e.\ for every \( X,Y,Z \in \mathscr{X} \), if \( Y \subseteq X \) and \( f \in \F(X,Z) \), then \( f \restriction Y \in \F(Y,Z) \).
\end{enumerate}
\end{definition}%

Examples of family of reducibilities are the collections of all continuous functions, of all uniformly continuous functions, of all Lipschitz functions, 
and of all nonexpansive functions. Notice also that  if \( \F \) is a family of reducibilities then \( \F(X) =  \F(X,X) \) is a reducibility on the space \( X \) 
(for every \( X \in \mathscr{X} \)). The next simple lemma is a minor variation of~\cite[Proposition 5.4]{MottoRos:2012b} and can be proved in a
 similar way.

\begin{lemma} \label{lemma:retraction}
Let \( \F \) be a family of reducibilities and \( X,Y \in \mathscr{X} \). Suppose that there is a surjective \( f \in \F(X,Y) \) admitting a right inverse 
\( g \in \F(Y,X) \). Then there is an embedding from \( (\pow(Y), \leq_{\F(Y)}, \neg) \) into \( (\pow(X), \leq_{\F(X)}, \neg ) \).

In particular, if \( \F \) consists of Borel functions
%, \( X \) and \( Y \) are as above, 
and the \( \F(X) \)-hierarchy on Borel subsets of \( X \) is (very) good, then also the \( \F(Y) \)-hierarchy on Borel subsets of \( Y \) is (very) good. 
\end{lemma}

\begin{proof}
The map \( \pow(Y) \to \pow(X) \colon A \mapsto f^{-1}(A) \) is the desired embedding. 
\end{proof}

\section{Uniformly continuous and Lipschitz reducibilities}

In~\cite[Question 6.2]{MottoRos:2012c}, it is asked whether one can equip the Baire space \( \pre{\omega}{\omega} \) with a compatible complete
 ultrametric \( d' \) so that \( \L(\bar{d}) \not\subseteq \UCont(d') \), and whether it is possible to strengthen this last condition to: the 
\( \UCont(d') \)-hierarchy on \( X \) is (very) bad. We start by answering positively the first part of this question.

\begin{notation}
Given a function \( \phi \colon \omega \to \RR^+ \), we denote by \( \rg (\phi) \) the \emph{range of \( \phi \)}, i.e.\ \( \rg(\phi) = \{ r \in \RR^+ \mid \exists n \in \omega \, ( \phi(n) = r) \} \).
\end{notation}

\begin{definition} \label{def: d_phi}
Given a function \( \phi \colon \omega \to  \RR^+  \) with \( \inf \rg (\phi)>0 \), define the metric \( d_\phi \) on \( \pre{\omega}{\omega} \) by setting for every 
\( x,y \in \pre{\omega}{\omega} \)
\[ 
d_\phi(x,y) = \max \{ \phi(x(0)), \phi(y(0)) \} \cdot \bar{d}(x,y).
 \] 
\end{definition}

It is not hard to check that each \( d_\phi \) is a complete ultrametric compatible with the product topology on \( \pre{\omega}{\omega} \) (and that \( \inf \rg(\phi)>0 \) is necessary for completeness). 

\begin{notation}
Given a natural number \( i \in \omega \) and an ordinal \(\alpha\), we denote by \( i^{(\alpha)} \) the constant \(\alpha\)-sequence with value \( i \).
\end{notation}

\begin{proposition} \label{prop:UCont}
Let \( \phi \colon \omega \to \RR^+  \colon n \mapsto 2^{n} \). Then \( \L(\bar{d}) \not\subseteq \UCont(d_\phi) \).
\end{proposition}

\begin{proof}
Consider the map \( f \colon \pre{\omega}{\omega} \to \pre{\omega}{\omega} \colon n {}^\smallfrown{} x \mapsto 3n {}^\smallfrown{} x \). 
We show that for every \( \varepsilon,\delta \in \RR^+ \) there are \(x,y \in \pre{\omega}{\omega} \) such that \( d_\phi(x,y) < \delta \) but 
\( d_\phi(f(x),f(y)) > \varepsilon \). Let \( 0 \neq k \in \omega \) be such that \(2^{-k} < \delta \). Then for every \( n \geq k \) we get that setting 
\( x = n^{(2n)} {}^\smallfrown{} 0^{(\omega)} \) and \( y = n^{(2n)} {}^\smallfrown{} 1^{(\omega)} \), 
\[ d_\phi(x,y) = 2^n \cdot 2^{-2n} = 2^{-n} \leq 2^{-k} < \delta . \]
However, 
\[ 
d_\phi(f(x),f(y)) = 2^{3n} \cdot 2^{-2n} = 2^n,
 \] 
hence letting \( n \) be large enough we get \( d_\phi(f(x),f(y)) > \varepsilon \), as desired.
\end{proof}

In order to answer the second half of~\cite[Question 6.2]{MottoRos:2012c}, we abstractly analyze the behavior of the \( \UCont(d) \)-hierarchy on 
an arbitrary ultrametric Polish space \( X = (X,d) \). 
The following lemma uses standard arguments (see e.g.\ the proof of~\cite[Theorem 7.8]{Kechris:1995}), but we fully reprove it here for the 
reader's convenience.

\begin{lemma} \label{lemma:unifcontandlip} 
Let $X = (X,d)$ be an ultrametric Polish space. Then there is a closed set \( C \subseteq \pre{\omega}{\omega} \) and a bijection 
$f \colon (C, \bar{d}) \to (X,d)$ such that $f$ is uniformly continuous and $f^{-1}$ is nonexpansive.  Moreover, if $X$ has bounded diameter,
 then $f$ is even Lipschitz,
and if \( X \) has diameter \( \leq 1 \) then we can alternatively require \( f \) to be nonexpansive and \( f^{-1} \) to be Lipschitz with constant \( 2 \).
\end{lemma} 

\begin{proof}
Let \( Q \) be a countable dense subset of \( X \).
Define the sets \( A_s  \subseteq X\) for \( s \in \pre{< \omega}{\omega} \) recursively on \( \leng(s) \) as follows: \( A_\emptyset = X \).  
Given \( A_s \subseteq X\), 
let \(  \{ B_{s,i} \mid i < I \} \) (for some \( I \leq \omega \)) be an enumeration without repetitions of the set of open balls 
\( \{ B_d(x,2^{-\leng(s)}) \mid x \in Q \cap A_s \} \), and set \( A_{s {}^\smallfrown{} i} = B_{s,i} \) if \( i < I \) and 
\( A_{s {}^\smallfrown{} i} = \emptyset \) otherwise. Since \( d \) is an ultrametric, one can easily check that the family 
\( (A_s)_{s \in \pre{< \omega}{\omega}} \) is a Luzin scheme with vanishing diameter consisting of clopen sets, and with the further property that 
\( A_s = \bigcup_{n \in \omega} A_{s {}^\smallfrown{} n} \) for every \( s \in \pre{< \omega}{\omega} \). Therefore the set 
\( C = \{ x \in \pre{\omega}{\omega} \mid \bigcap_{n \in \omega} A_{x \restriction n} \neq \emptyset \} \) is a closed subset of 
\( \pre{\omega}{\omega} \), and the map \( f \colon C \to X \) sending \( x \in C \) to the unique element in 
\( \bigcap_{n \in \omega} A_{x \restriction n} \) is a bijection. So it remains only to check that such \( f \) has the desired properties.

Given \( \varepsilon > 0 \), let \( n \in \omega \) be smallest such that \( 2^{-n} \leq \varepsilon \), and set \( \delta = 2^{-n} \). If \( x,y \in C \) are such
 that \( \bar{d}(x,y) < \delta \), then \( x \restriction (n+1) = y \restriction (n+1) \), which implies \( f(x),f(y) \in A_{x \restriction (n+1)} \). By definition of
 the \( A_s \), this implies that \( d(f(x),f(y)) < 2^{-n} \leq \varepsilon \). This shows that \( f \) is uniformly continuous.

Further assuming that \( X \) be of bounded diameter, we get that \( f \) is Lipschitz with constant \( \max \{ 2,k \} \), where \( k \in \omega \) is an
 arbitrary bound to the diameter of \(X \), i.e.\ it is such that \( d(x,y) \leq k \) for every \( x,y \in X \). To see this, fix distinct \( x,y \in C\). If 
\(x(0) \neq y(0) \) then \( d(f(x),f(y)) \leq k \leq k \cdot \bar{d}(x,y) \) by our choice of \( k \in \omega \). Let now \( n \neq 0 \) be smallest such that 
\( x(n) \neq y(n) \), so that \( \bar{d}(x,y) = 2^{-n} \). Since \( x \restriction n = y \restriction n \) we get that \( f(x),f(y) \in A_{x \restriction n} \), 
which implies \( d(f(x),f(y)) < 2^{-(n-1)} \): therefore \( d(f(x),f(y)) < 2 \cdot 2^{-n} = 2 \cdot d(x,y) \).

Now fix \( x,y \in C \), and let \( n \in \omega \) be such that \( \bar{d}(x,y) = 2^{-n} \). Since \( x(n) \neq y(n) \) implies 
\( A_{x \restriction (n+1)} \cap A_{y \restriction (n+1)} = \emptyset \), we get that \( d(f(x),f(y)) \geq 2^{-n} \) (because \( d \) is an ultrametric), 
and hence \( \bar{d}(x,y) \leq d(f(x),f(y)) \). This shows that \( f^{-1} \) is nonexpansive.

Finally, assume that \( X \) has diameter \( \leq 1 \). In the construction above, redefine the collections \( \{ B_{s,i} \mid i < I \} \) as enumerations
 without repetitions of the sets \( \{ B_d(x,2^{-(\leng(s)+1)}) \mid x \in Q \cap A_s \} \), and then use this new sets to define the \( A_s \)'s and the map \( f \).
 Arguing as before, one can easily check that \( f \) is now nonexpansive while \( f^{-1} \) is Lipschitz with constant \( 2 \), as required.
\end{proof}

\begin{remark}
The special case of Lemma~\ref{lemma:unifcontandlip} where \( X \) has diameter \( \leq 1 \) already appeared (with the same proof) 
in~\cite[Theorem 4.1]{MottoRos:2009b}. However, such a result cannot be  literally extended to an arbitrary ultrametric Polish space \( X \), and in fact 
the assumptions in Lemma~\ref{lemma:unifcontandlip} are optimal. To see this, note that if \( X \) has unbounded diameter then we cannot 
require a map \( f \) as in Lemma~\ref{lemma:unifcontandlip} to be Lipschitz because every Lipschitz image of a space with
 bounded diameter (like any set \( C \subseteq \pre{\omega}{\omega} \)) has necessarily bounded diameter too. Similarly, a nonexpansive image of a 
set of diameter 
\( \leq R \) (for some \( R \in \RR^+ \)), has diameter \( \leq R \) too.
\end{remark}
•

\begin{definition}
Let \( X \) be a topological space, \( \F \) be a collection of functions from \( X \) to itself, and \( A \subseteq X \). We call \emph{\( \F \)-retraction} of 
\( X \) onto \( A \) any surjection \( f \in \F \) from \( X \) onto \( A \) such that \( f \restriction A = \id_A \); if such a function exists we also say that \( A \) 
is an \emph{\( \F \)-retract} of \( X \).
\end{definition}

Recall from  \cite[Proposition 2.8]{Kechris:1995} that if \( \emptyset \neq A \subseteq C \) are closed subsets of the Baire space, then there is an 
\( \L(\bar{d}_C) \)-retraction (i.e.\ a nonexpansive retraction) of \( C \) onto \( A \) --- a fact that will be repeatedly used throughout the paper. 
The next corollary generalizes this result to arbitrary ultrametric Polish spaces, provided that we slightly weaken the requirement that the 
retraction be nonexpansive.

\begin{corollary} \label{cor:retraction} 
Let \( X = (X,d) \) be an ultrametric Polish space. For every nonempty closed \( A \subseteq X \), there is a uniformly continuous retraction 
\( r \colon X \twoheadrightarrow A \). If moreover \( A \) has bounded diameter, then the retraction \( r \) can be taken to be Lipschitz.
\end{corollary}

\begin{proof}
Let \( C \) and \( f \) be as in Lemma~\ref{lemma:unifcontandlip}, with \( f \) uniformly continuous and \( f^{-1} \) nonexpansive. Notice that since 
\( f \) is, in particular, a homeomorphism, the set \( A' = f^{-1}(A) \) is a nonempty closed subset of \( C \). Let \( g \colon C \to A' \) be a 
nonexpansive retraction: then \( r = f \circ g \circ f^{-1} \colon (X,d) \twoheadrightarrow (A,d) \) is the desired uniformly continuous retraction.

Assume now that \( A \) has bounded diameter, and let \( C \), \( f \), and \( g \) be as in the previous paragraph. Arguing as in the proof of
 Lemma~\ref{lemma:unifcontandlip}, one can easily check that \( f \restriction A' \colon (A', \bar{d}) \to (X,d) \) is actually Lipschitz (since \( A \) has
 bounded diameter): therefore \( r = (f \restriction A') \circ g \circ f^{-1} \colon (X,d) \twoheadrightarrow (A,d) \) is the desired Lipschitz retraction.
\end{proof}

\begin{remark} 

It is not possible in general to strengthen Corollary~\ref{cor:retraction} by requiring the reduction to be nonexpansive, even if we require the entire 
\( X \) to have small diameter. To see this, let \( X = \{ 0 \} \cup \left\{ \frac{1}{2} + 2^{-(n+1)} \mid n \in \omega \right\} \), and set 
\( d(x,y) = \max \{ x,y \} \) for all distinct \( x,y \in X \). Then \( X = (X,d) \) is a discrete ultrametric Polish space of diameter \( \leq 1 \). Consider the 
clopen set \( A = X \setminus \{ 0 \} \), and let \( f \colon X \twoheadrightarrow A \) be a retraction. Let \( n \in \omega \) be such that 
\( f(0) = \frac{1}{2} + 2^{-(n+1)} \): then setting \( x = 0 \) and \( y =  \frac{1}{2} + 2^{-(n+2)} \) we get that
\[ 
d(f(x),f(y)) = d(f(x),y) =  \frac{1}{2} + 2^{-(n+1)} > \frac{1}{2} + 2^{-(n+2)} = d(x,y),
 \] 
so \( f \) is expansive.
\end{remark}

%\begin{proof} 
%Choose a minimal $2^{-n}$-net for all $n\in\omega$. Choose an enumeration $\{a_{n,k}\mid k<m\}$ without repetitions of the $2^{-n}$-net for some $m\leq\omega$. Then 
%\begin{itemize} 
%\item $\forall k\neq l\ d(a_{n,k}, a_{n,l})\geq 2^{-n}$ and 
%\item for all $x\in X$, $n$ there is a unique $k$ with $d(x,a_{n,k})<2^{-n}$. 
%\end{itemize} 
%Choose $a_{\emptyset}\in X$. Let $(a_{s\smallfrown(i)})_{i<m}$ enumerate without repetitions the $a_{n,k}$ with $d(a_{n,k}, a_s)<2^{-n}$ for $s\in {}^n\omega\cap T$. This defines $T$, $(a_s)_{s\in T}$. Let $f(x)=\lim_n a_{x\upharpoonright n}$ for $x\in [T]$. 
%\end{proof} 

\begin{theorem} \label{th:unifcontandlip}
The \( \UCont(d) \)-hierarchy \( \Deg_{\boldsymbol{\Delta}^1_1}(\UCont(d)) \) on the Borel subsets of an arbitrary ultrametric Polish space \( X = (X,d) \) is always very good. If \( X \) has 
bounded diameter, then the \( \Lip(d) \)-hierarchy \( \Deg_{\boldsymbol{\Delta}^1_1}(\Lip(d)) \) on the Borel subsets of \( X \) is very good as well.
\end{theorem}

\begin{proof}
Let \( C \subseteq \pre{\omega}{\omega} \) and \( f \colon C  \to X \) be as in Lemma~\ref{lemma:unifcontandlip}, and let 
\( g \colon  (\pre{\omega}{\omega}, \bar{d}) \to (C, \bar{d}) \) be a nonexpansive retraction. Then \( f^{-1} \) is a right inverse of \( g \circ f \), 
and hence the result follows from Lemma~\ref{lemma:retraction} and the fact that both the \( \UCont(\bar{d}) \)-hierarchy and the 
\( \Lip(\bar{d}) \)-hierarchy are very good by~\cite{MottoRos:2010}.
\end{proof}

In particular, this fully answers in the negative the second half of~\cite[Question 6.2]{MottoRos:2012c}. Moreover, 
Theorem~\ref{th:unifcontandlip} provides also a negative answer to~\cite[Question 6.1]{MottoRos:2012c}: letting \( \phi \) be as in
 Proposition~\ref{prop:UCont}, we get that the set \( \UCont(d_\phi) \) of uniformly continuous functions is a surjective image of \( \pre{\omega}{\omega} \),\footnote{When working in models of \( \AD \) (as it is often the case when dealing with Wadge-like hierarchies), for technical reasons it is often preferable to express ``cardinality inequality'' using surjections instead of injections. Therefore the stated property should be intended (in any model of \( \ZF \)) as: the cardinality of \( \UCont(d_\phi) \) is not larger than that of the Baire space. Obviously, further assuming the Axiom of Choice \( \mathsf{AC} \) this just means that \( \UCont(d_\phi) \) has cardinality \( \leq 2^{\aleph_0} \).} it does not contain \( \L(\bar{d}) \),
 but it induces a very good hierarchy on the Borel subsets (or, further assuming \( \AD \), on the collection of all subsets) of 
\( \pre{\omega}{\omega} \).

Theorem~\ref{th:unifcontandlip} shows that having a bounded diameter is a sufficient condition for having that the \( \Lip(d) \)-hierarchy on the 
Borel subsets of an ultrametric Polish space \( X = (X,d) \) is very good. In fact, we are now going to show that a technical strengthening of this
 condition is both necessary and sufficient for that.

\begin{definition} \label{def:nontriviallyunbounded}
Let \( X = (X,d) \) be an (ultra)metric Polish space. We say that the diameter of \( X \) is \emph{nontrivially unbounded} if for every 
\( k \in \omega \) and every \( \varepsilon \in \RR^+ \) there are \( x,y \in X \) with \( d(x,y) > k \) such that  both \( x \) and \( y \) are not 
\( \varepsilon \)-isolated.%
\footnote{Recall that a point \( x \) of a metric space is called \emph{\( \varepsilon \)-isolated} (for some \( \varepsilon \in \RR^+ \)) if \( B_d(x,\varepsilon) = \{ x \} \).}
%\[ 
%\forall n\in \omega \, \forall \varepsilon \in \RR^+  \, \exists x,y \in X \, (d(x,y) > n  \wedge x,y \text{ are not \( \varepsilon \)-isolated}).
% \] 
\end{definition}

Notice that if \( X \) is \emph{perfect}, then the diameter of \( X \) is nontrivially unbounded if and only if it is unbounded.

%\begin{example} 
%Let $X=\{x \in 2^\ZZ \mid \exists n_0 \in \ZZ \, \forall n< n_0 (x(n)=0) \}$ and set $d(x,y)=2^{-n}$ for the least $n \in \ZZ$ with $x(n)\neq y(n)$ 
%for distinct  $x, y \in X$, and \( d(x,y) = 0 \) for \( x = y \in X \). Then $X = (X,d)$ is an ultrametric Polish space whose diameter is nontrivially unbounded. 
%\end{example} 

\begin{example} \label{xmp:p-adic}
Let \( p \) be a prime natural number, and let \( \mathbb{Q}_p \) be the ultrametric Polish space of \( p \)-adic numbers equipped with the usual \( p \)-adic metric \( d_p \): then \( \mathbb{Q}_p \) has unbounded diameter and is perfect (hence its diameter is nontrivially unbounded). To see the former,  given \(  k \in \omega \)  let \( n \in \omega \) be such that \( n \geq 2 \) and  \( k < p^n \): setting \( x = p^{-1} \) and \( y = p^{-n} \) we easily get \( d_p(x,y) = p^n > k \). To see that \( \mathbb{Q}_p \) is also perfect, fix an arbitrary \( q \in \mathbb{Q} \), and given \( \varepsilon \in \RR^+ \) let \( l \in \omega \) be such that \( p^{-l} < \varepsilon \): then \( q' = q - p^l \) is distinct from \( q \) and \( d_p(q,q') = p^{-l} < \varepsilon \). This shows that \( q \) is not isolated, and since \( \mathbb{Q} \) is dense in \( \mathbb{Q}_p \) we are done.
\end{example}

\begin{notation}
We let \( \subseteq^* \) denote the relation of inclusion modulo finite sets between subsets of \( \omega \), i.e.\ for every 
\( a,b \subseteq \omega \) we set
\[ 
a \subseteq^* b \iff \exists \bar{k} \in \omega \, \forall k \geq \bar{k} \, (k \in a \Rightarrow k \in b).
 \] 
\end{notation}

\begin{theorem} \label{th:unboundeddiam} 
Let \( X = (X,d) \) be an ultrametric Polish space, and assume that its diameter is nontrivially unbounded. Then there is a map \( \psi \) from 
\( \pow(\omega) \) into the clopen subsets of \( X \) such that for all \(a,b \subseteq \omega \):
\begin{enumerate}
\item
if \( a \subseteq^* b \) then \( \psi(a) \leq_{\L(d)} \psi(b) \);
\item
if \( \psi(a) \leq_{\Lip(d)} \psi(b) \) then \( a \subseteq^* b \).
\end{enumerate}
In particular, \( (\pow(\omega), \subseteq^* ) \) embeds into both \( \Deg_{\boldsymbol{\Delta}^0_1}(\Lip(d)) \) and \( \Deg_{\boldsymbol{\Delta}^0_1}(\L(d)) \).
%  the \( \Lip(d) \)- and the \( \L(d) \)-hierarchy on 
%\( \boldsymbol{\Delta}^0_1 \)-subsets of \( X \).%, and therefore such hierarchies are both very bad (they contain antichains of size \( \pre{\omega}{2} \) and infinite descending chains).
\end{theorem}

\begin{proof}
Let \( (q_n)_{n \in \omega} \) be an enumeration of a countable dense subset \( Q \) of \( X \).
We first recursively construct two sequences \( (r_n)_{n \in \omega} \), \( (s_n)_{n \in \omega} \) of nonnegative reals and two sequences 
\( (x_n)_{n \in \omega} \), \( (y_n)_{n \in \omega} \) of points of \( X \) such that for all distinct \( n,m \in \omega \) the following properties hold:
\begin{enumerate}[(a)]
\item
 \( d(x_n,x_m) = r_{\max \{ n,m \}} \) and
\( d(x_n,y_n) = s_n \);
\item
\( r_{n+1} > \max \{ n+1, r_n^2 \} \) (in particular, \( (r_n)_{n \in \omega} \) is strictly increasing and unbounded in \( \RR^+ \));
\item
\( s_0 < 1 \) and \( s_{n+1} < \frac{s_n}{r_n+1} \) (in particular, \( (s_n)_{n \in \omega} \) is a strictly decreasing sequence).
\end{enumerate}

\begin{claim} \label{claim:nonisolated}
If \( x \in X \) is not \( \varepsilon \)-isolated then there are at least two distinct \( q_i,q_j \in Q \) such that \( q_i,q_j \in B_d(x,\varepsilon) \).
\end{claim}

\begin{proof}[Proof of the Claim]
Since \( x \) is not \( \varepsilon \)-isolated, there is \( y \in B_d(x,\varepsilon) \) such that \( x \neq y \). By density of \( Q \), there are 
\( q_i,q_j \in Q \) such that \( q_i \in B_d(x,d(x,y)) \) and \( q_j \in B_d(y,d(x,y)) \). Then \( q_i \neq q_j \) since 
\( B_d(x,d(x,y)) \cap B_d(y,d(x,y)) = \emptyset \), while \( q_i,q_j \in B_d(x,\varepsilon) \) because 
\( B_d(x,d(x,y)), B_d(y,d(x,y)) \subseteq B_d(x,\varepsilon) \) by \( d(x,y) < \varepsilon \).
\end{proof}

Let \( x\in X \) be not \( 1 \)-isolated (such an \( x \) exists because the diameter of \( X \) is nontrivially unbounded), and let \( q_i,q_j \) be as in
 Claim~\ref{claim:nonisolated} for \( \varepsilon=1 \). Then we set \( x_0 = q_i \), \( y_0 = q_j \), \( r_0 = 0 \), and \( s_0 = d(q_i,q_j) \). Now assume that \( x_n \), \( y_n \), 
\( r_n \), and \( s_n \) have been defined. Let \( x,y \in X \) be such that \( d(x,y) > \max \{ n+1,r^2_n \} \) and \( x,y \) are not 
\( \frac{s_n}{r_n+1} \)-isolated. Then at least one of \( x \) and \( y  \) has distance greater than \(  \max \{ n+1,r^2_n \} \) from \( x_n \) (and hence
 also from all the \( x_m \) for \( m \leq n \)): if not, then we would have \( d(x,y) \leq \max \{ d(x,x_n),d(y,x_n) \} \leq  \max \{ n+1,r^2_n \} \),
 contradicting our choice of \( x,y \). So we may assume without loss of generality that \( d(x,x_n) > \max \{ n+1,r^2_n \} \) and \( x \) is not 
\( \frac{s_n}{r_n+1} \)-isolated. Let \( q_i,q_j \) be as in Claim~\ref{claim:nonisolated} for \( \varepsilon=\frac{s_n}{r_n+1} \), and set \( x_{n+1} = q_i \), \( y_{n+1} = q_j \), 
\( r_{n+1} = d(q_i,x_n) \), and \( s_{n+1} = d(q_i,q_j) \). Since \( d(q_i,x) < \frac{s_n}{r_n+1} \leq 1 \leq \max \{ n+1,r^2_n \} \), we have 
\( r_{n+1} = d(q_i, x_n) = d(x,x_n)  > \max \{ n+1,r^2_n \} \). Moreover, \( s_{n+1} < \frac{s_n}{r_n+1} \) by the fact that 
\( q_i,q_j \in B_d(x,\frac{s_n}{r_n+1}) \). Arguing by induction on \( n \in \omega \), it is then easy to check that the sequences constructed in this 
way have all the desired properties.

Given \( a \subseteq \omega \), let \( \hat{a} = \{ 2i \mid i \in \omega \} \cup \{ 2i+1 \mid i \in a \} \), so that \( \hat{a} \) is always infinite 
and for every \( a,b \subseteq \omega \)
\[ 
a \subseteq^* b \iff \hat{a} \subseteq^* \hat{b}.
 \] 
For \( a \subseteq \omega \), set \( \psi(a) = \bigcup_{i \in \hat{a}} B_d(x_i,s_i) \). Clearly, each \( \psi(a) \) is an open subset of \( X \). To see that it is also
 closed, observe that \( B_d(x_i,s_i) \subseteq B_d(x_i,1) \) for every \( i \in \omega \) by our choice of the \( s_i \)'s, and that for distinct 
\( i,j \in \omega \) the clopen balls \( B_d(x_i,1) \) and \( B_d(x_j,1) \) are disjoint by  our choice of the \( x_i \)'s and of the \( r_i \)'s: therefore, 
since the open balls in \( X \) are automatically closed we get that 
\begin{align*} 
X \setminus \psi(a) =& \bigcup \left\{ B_d(z,1) \mid z \notin \bigcup\nolimits_{i \in \hat{a}} B_d(x_i,1) \right\} \\ 
&\cup \bigcup \{ B_d(x_i,1) \setminus B_d(x_i,s_i) \mid i \in \hat{a} \}
\end{align*} 
%\[ 
%X \setminus \psi(a) = \bigcup \left\{ B_d(z,1) \mid z \notin \bigcup\nolimits_{i \in \hat{a}} B_d(x_i,1) \right\} \cup \bigcup \{ B_d(x_i,1) \setminus B_d(x_i,s_i) \mid i \in \hat{a} \}
% \] 
is open.

Let now \( a ,b \subseteq \omega \) be such that \( a \subseteq^* b \), which in particular implies \( \hat{a} \subseteq^* \hat{b} \), and let 
\( 0 \neq \bar{k} \in \omega \) be such that \( \bar{k} \in \hat{a} \) and \( k \in \hat{a} \Rightarrow k \in \hat{b} \) for every \( k \geq \bar{k} \). Define 
\( f \colon (X,d) \to (X,d) \) as follows:
\[ 
f(x) = 
\begin{cases}
x_{\bar{k}} & \text{if } x \in B_d(x_i,s_i), i < \bar{k}, i \in \hat{a} \\
y_{\bar{k}} & \text{if } x \in B_d(x_0, r_{\bar{k}}) \setminus \bigcup \{ B_d(x_i,s_i) \mid i < \bar{k}, i \in \hat{a} \} \\
y_i & \text{if } x \in B_d(x_i,s_i), i \geq \bar{k}, i \notin \hat{a} \\
x & \text{otherwise.}
\end{cases}%
 \] 
It is straightforward to check that \( f \) reduces \( \psi(a) \) to \( \psi(b) \), so we only need to check that \( f \) is  nonexpansive, and this amounts to
 check that if \( x,y  \) are distinct points of \( X \) which fall in different cases in the definition of \( f \), then \( d(f(x),f(y)) \leq d(x,y) \). A careful
 inspection shows that the unique nontrivial cases are the following:
\begin{enumerate}[{case} A:]
\item
\( x \in B_d(x_0,r_{\bar{k}}) \), while \( y \notin B_d(x_0,r_{\bar{k}})  \cup \bigcup \{ B_d(x_i,s_i) \mid i \geq \bar{k}, i \notin \hat{a} \} \).  Then 
\( d(x,y) \geq r_{\bar{k}} \) (by case assumption) and \( d(x,f(x)) = r_{\bar{k}} \) (because 
either \( f(x) = x_{\bar{k}} \) or \( f(x) = y_{\bar{k}} \), depending on whether \( x \in B_d(x_i,s_i) \) for some \( i \in \hat{a} \) smaller than 
\( \bar{k} \) or not). 
Since in the case under consideration \( f(y) = y \), we get that either \( d(f(x),f(y)) \leq r_{\bar{k}} \), or else \( d(f(x),f(y)) = d(f(x),y) = d(x,y) \) 
by the isosceles triangle rule: in both cases, \( d(f(x),f(y)) \leq d(x,y) \) as required.
\item
\( x \in B_d(x_0, r_{\bar{k}}) \setminus \bigcup \{ B_d(x_i,s_i) \mid i < \bar{k}, i \in \hat{a} \} \), while \( y \in B_d(x_i,s_i) \) for some 
\( i \geq \bar{k} \), \( i \notin \hat{a} \). Then since \( d(x,x_0) < r_{\bar{k}} \) and \( d(x_0,y) = r_i \geq r_{\bar{k}} \), we get \( d(x,y) = r_i \). 
Since by case assumption \( f(x) = y_{\bar{k}} \) and \( f(y) = y_i \), either \( f(x) = f(y) \) (in case \( i = \bar{k} \)) or \( d(f(x),f(y)) = r_i \), and 
hence we again get \( d(f(x),f(y)) \leq d(x,y) \), as required.
\item
\( x \in B_d(x_i,s_i) \) for some \( i \geq \bar{k} \), \( i \notin \hat{a} \), while also \( y \notin B_d(x_0,r_{\bar{k}})  \cup\) \( \bigcup \{ B_d(x_i,s_i) \mid i \geq \bar{k}, i \notin \hat{a} \} \). Then \( d(x,y) \geq s_i \), \( d(x,f(x)) = s_i \) (because \( f(x) = y_i \)), and \( f(y) = y \): this implies that either \( d(f(x),f(y))  \leq s_i \) or \( d(f(x),f(y)) = d(f(x),y) = d(x,y) \), so that in any case \( d(f(x),\) \(f(y)) \leq d(x,y) \).
\end{enumerate}
This concludes the proof of part (1).

\medskip

%, and let \( f \colon (X,d) \to (X,d) \) be defined by
%\[ 
%f = \id_X \restriction \left(X \setminus \bigcup\nolimits_{i \notin a} X_i \right) \cup \bigcup\nolimits_{i \notin a} f_{z_i} \restriction X_i,
% \] 
%where \( f_{z_i} \) is the constant function with value \( z_i \). Then it is straightforward to check that \( f \) is an \( \L(d) \)-reduction of \( \psi(a) \) to \( \psi(b) \).

We now prove part (2) of the theorem. Given \(a,b \subseteq \omega \), assume that \( f \colon (X,d) \to (X,d) \) is a \( \Lip(d) \)-reduction of \( \psi(a) \) to \( \psi(b) \), and let \( 0 \neq n \in \omega \) be such that \( d(f(x),f(y)) \leq r_n \cdot d(x,y) \) for every \( x,y \in X \) (such an \( n \) exists because \( (r_n)_{n \in \omega} \) is unbounded in \( \RR^+ \) by (b) above). 
%Since 
%%%we want to show that \( a \subseteq^* b \), we can assume without loss of generality that 
%\( \hat{a} \neq \emptyset \) we have \( n \geq 1 \). 
Notice that, necessarily, 
\[ 
f \left(\bigcup \{ B_d(x_i,s_i) \mid i \in \hat{a} \}\right) = f(\psi(a)) \subseteq \psi(b) \subseteq \bigcup_{j \in \omega } B_d(x_j,s_j). 
\] 
We now argue as in the proof of~\cite[Theorem 5.4]{MottoRos:2012c}. 

\begin{claim} \label{claim:Lip1}
Fix an arbitrary \( i \in \hat{a} \). If there are \( x \in B_d(x_i,s_i) \) and \( j \geq n \) such that \( f(x) \in B_d(x_j,s_j) \), then \( f(B_d(x_i,s_i)) \subseteq B_d(x_j,s_j) \).
\end{claim}

\begin{proof}[Proof of the Claim]
Suppose not, and let \( y \in B_d(x_i,s_i) \) and \( j' \neq j \) be such that \( f(y) \in B_d(x_{j'},s_{j'}) \). Then
\[ 
d(f(x),f(y)) = \max \{ r_j,r_{j'} \} \geq r_j \geq r_n \cdot 1 > r_n \cdot s_i > r_n \cdot d(x,y),
 \] 
contradicting the choice of \( n \).
\end{proof}

\begin{claim}\label{claim:Lip2}
For every \( i \in \hat{a} \) such that \( i > n \), \( f(B_d(x_i,s_i)) \subseteq B_d(x_j,s_j) \) for some \( j \geq i \).
\end{claim}

\begin{proof}
Suppose towards a contradiction that there are \( x \in B_d(x_i,s_i) \) and \( j < i \) such that \( f(x) \in B_d(x_j,s_j) \), so that, in particular, \( j \in \hat{b} \) 
because \( x \in \psi(a) \) and \( f \) reduces \( \psi(a) \) to \( \psi(b) \). Then since \( d(x,y_i) = s_i \), by our choice of the \( s_i \)'s we get
\[ 
d(f(x),f(y_i)) \leq r_n \cdot s_i \leq r_{i-1} \cdot s_i < s_{i-1} \leq s_j, 
\]
and hence \( f(y_i) \in B_d(f(x),s_j) =  B_d(x_j,s_j) \subseteq \psi(b) \): but this contradicts the fact that \( f \) is a reduction of \( \psi(a) \) to \( \psi(b) \),
because \( y_i \notin \psi(a) \) while \( B_d(x_j,s_j) \subseteq \psi(b) \) since \( j \in \hat{b} \).  Thus, given an arbitrary \( x \in B_d(x_i,s_i) \) there is \( j \geq i > n\) such that \( f(x) \in B_d(x_j,s_j) \): by Claim~\ref{claim:Lip1}, we then get \( f(B_d(x_i,s_i)) \subseteq B_d(x_j,s_j) \), as required.
\end{proof}

Let now \( \bar{\imath} \) be the smallest element of \( \hat{a} \). By Claim~\ref{claim:Lip1}, either \( f(B_d(x_{\bar{\imath}},s_{\bar{\imath}})) \subseteq \bigcup_{j < n } B_d(x_j,s_j) \), or \( f(B_d(x_{\bar{\imath}}s_{\bar{\imath}})) \subseteq B_d(x_j,s_j) \) for some \( j \geq n \). Therefore, in both cases there is \( \bar{k} > \max \{ n, \bar{\imath} \} \) such that \( f(B_d(x_{\bar{\imath}},s_{\bar{\imath}})) \subseteq \bigcup_{j \leq \bar{k}} B_d(x_j,s_j) \): we claim that \( k \in \hat{a} \Rightarrow k \in \hat{b} \) for every \( k \geq \bar{k} \), which also implies \( a \subseteq^* b \).

Fix \( k \geq \bar{k} \) such that \( k \in \hat{a} \). By Claim~\ref{claim:Lip2} and \( \bar{k} > n \), there is \( j \geq k \) such that \( f(B_d(x_k,s_k)) \subseteq B_d(x_j,s_j) \). Assume towards a contradiction that \( j > k \):
%, and fix \( x \in B_d(x_{\bar{\imath}},s_{\bar{\imath}}) \) and \( y \in B_d(x_k,s_k) \). 
then
\[ 
d(f(x_{\bar{\imath}}),f(x_k)) = r_j > r_k \cdot r_k > r_n \cdot r_k = r_n \cdot d(x_{\bar{\imath}},x_k),
 \] 
contradicting the choice of \( n \). Therefore \( f(B_d(x_k,s_k)) \subseteq B_d(x_k,s_k) \), which in particular implies that \( \psi(b) \cap B_d(x_k,s_k) \neq \emptyset \) (since \( x_k \in \psi(a) \) and \( f \) reduces \( \psi(a) \) to \( \psi(b) \)): but this means that \( k \in \hat{b} \), and hence we are done.
\end{proof}

Applying Theorem~\ref{th:unboundeddiam} to the space \( \mathbb{Q}_p \) of \( p \)-adic numbers (which is possible by Example~\ref{xmp:p-adic}) we get the following corollary.

\begin{corollary}\label{cor:p-adic}
Let \( p \) be a prime natural number, and let \( d_p \) be the \( p \)-adic metric on the space \( \mathbb{Q}_p \). Then both the \( \Lip(d_p) \)- and the \( \L(d_p) \)-hierarchies are very bad already when restricted to clopen subsets of \( \mathbb{Q}_p \).
\end{corollary}

The condition on the diameter of \( X = (X,d) \) used to prove Theorem~\ref{th:unboundeddiam} is very weak: this allows us to construct extremely simple (in fact: discrete) ultrametric Polish spaces \( X = (X,d) \) with the property that their \( \Lip(d) \)- and \( \L(d) \)-hierarchies are both very bad, despite the fact that all their subsets are topologically simple (i.e.\ clopen).

\begin{corollary} \label{cor:countable}
There exists a discrete (hence countable) ultrametric Polish space \( X_0 = (X_0,d_0) \) such that \( (\pow(\omega), \subseteq^*) \) embeds into both the \( \Lip(d_0) \)- and the \( \L(d_0) \)-hierarchy on (the clopen subsets of) \( X_0 \). In particular, \( \Deg(\Lip(d_0)) = \Deg_{\boldsymbol{\Delta}^0_1}(\Lip(d_0)) \) and \( \Deg(\L(d_0)) = \Deg_{\boldsymbol{\Delta}^0_1}(\L(d_0)) \) are both very bad.
\end{corollary}

\begin{proof}
Let \( X_0 = \{ x^i_n \mid n \in \omega, i = 0,1 \} \) and set
\[ 
d_0(x^i_n,x^j_m) = 
\begin{cases}
0 & \text{if } n = m \text{ and } i = j \\
2^{-n} & \text{if } n = m \text{ and } i \neq j \\
\max \{ n,m \} & \text{if } n \neq m.
\end{cases}%
 \] 
It is easy to check that \( X_0 = (X_0,d_0) \) is a discrete ultrametric Polish space. Now observe that the diameter of \( X_0 \) is nontrivially unbounded. In fact, given \( n \in \omega \) and \( \varepsilon \in \RR^+ \), let \( k \) be minimal such that \( 2^{-k} < \varepsilon \) and \( l = \max \{ n, k \} \): then \( d_0(x^0_l,x^0_{l+1}) = l+1 > n \), and the points \( x^1_l \) and \( x^1_{l+1} \) witness that \( x^0_l \) and \( x^0_{l+1} \) are not \( \varepsilon \)-isolated.  Therefore \( X_0 \) is as desired by Theorem~\ref{th:unboundeddiam}.
\end{proof}

The next proposition extends Theorem~\ref{th:unifcontandlip} and shows that the condition on \( X \) in Theorem~\ref{th:unboundeddiam} is optimal.

\begin{theorem} \label{th:notunbounded}
Let \( X = (X,d) \) be an ultrametric Polish space whose diameter is not nontrivially unbounded. Then the \( \Lip(d) \)-hierarchy \( \Deg_{\boldsymbol{\Delta}^1_1}(\Lip(d)) \) on Borel subsets of \( X \) is very good.
\end{theorem}

\begin{proof} 
Let \( n \in \omega \) and \( \varepsilon \in \RR^+  \) be such that for every \( x,y \), if \( d(x,y) > n \) then at least one of \( x \) and \( y \) is \( \varepsilon \)-isolated. 

Let us first consider the degenerate case in which all points of \( X \) are \(\varepsilon\)-isolated. Since constant functions are always (trivially) Lipschitz, we get that the sets \( X \) and \( \emptyset \) are \( \Lip(d) \)-incomparable, and that they are both (strictly) \( \leq_{\Lip(d)} \)-below any other set \( \emptyset,X \neq A \subseteq X \). Assume now that \( B \subseteq X \) is another set which is different from both \( \emptyset \) and \( X \): we claim that then \( A \equiv_{\Lip(d)} B \). To see this, fix \( \bar{x} \in B \) and \( \bar{y} \in \neg B \), and for every \( x \in X \) set \( f(x) = \bar{x} \) if \( x \in A \) and \( f(x) = \bar{y} \) if \( x \in \neg A \). Then \( f \colon (X,d) \to (X,d) \) reduces \( A \) to \( B \). Moreover, since for all distinct \( x,y \in X \) we have \( d(x,y) \geq \varepsilon \) (because both \( x \) and \( y \) are \(\varepsilon\)-isolated), we get
\[ 
d(f(x),f(y)) \leq d(\bar{x},\bar{y}) = \frac{d(\bar{x},\bar{y})}{\varepsilon} \cdot \varepsilon \leq \frac{d(\bar{x},\bar{y})}{\varepsilon} \cdot d(x,y),
 \] 
so that \( f \) is Lipschitz with constant \( \frac{d(\bar{x},\bar{y})}{\varepsilon} \). This shows that \( A \leq_{\Lip(d)} B \). Switching the role of \( A \) and \( B \), we get that also \( B \leq_{\Lip(d)} A \), and hence we are done. Therefore we have shown that the \( \Lip(d) \)-hierarchy on \( X \) is constituted by the two \( \Lip(d) \)-incomparable degrees \( [\emptyset]_{\Lip(d)}  = \{ \emptyset \} \) and \( [ X]_{\Lip(d)} = \{ X \} \), plus a unique \( \Lip(d) \)-degree above them containing all other subsets of \( X \), and is thus (trivially) very good.

Assume now that there is a non-\(\varepsilon\)-isolated point \( x_0 \in X \), and set \( X' = B_d(x_0,\) \(n+1) \). By our choice of \( n \) and \( \varepsilon \), we get that \( d(x,y) \geq n+1 \) for every \( x \in X' \) and \( y \in X \setminus X' \), and that each \( y \in X \setminus X' \) is \( \varepsilon \)-isolated (because \( d(x_0,y) > n \) and \( x_0 \) is not \(\varepsilon\)-isolated). We first prove the following useful claim.

\begin{claim} \label{claim:transfer}
Let \( A,B \subseteq X \) be such that \( B \neq \emptyset,X \). If there is a Lipschitz reduction \( f \colon (X',d_{X'}) \to (X',d_{X'}) \) of \( A' = A \cap X' \) to \( B' = B \cap X' \), then \( A \leq_{\Lip(d)} B \).
\end{claim}

\begin{proof}
Let \( f \) be as in the hypothesis of the claim, and let \( 1 \leq k \in \omega \) be such that \( d(f(x),f(y)) \leq k \cdot d(x,y) \) for every \( x,y \in X' \). Fix \( \bar{x} \in B \) and \( \bar{y} \in \neg B \), and extend \( f \) to the map \( \hat{f} \colon (X,d) \to (X,d) \) by letting \( \hat{f}(x) = \bar{x} \) if \( x \in A \setminus X' \) and \( \hat{f}(x) = \bar{y} \) if \( x \in X \setminus (X' \cup A) \). Clearly, \( \hat{f} \) reduces \( A \) to \( B \), and we claim that \( \hat{f} \) is Lipschitz with constant \( c \), where \( c \) is
\[ 
c = \max \left \{ k , \frac{d(\bar{x},\bar{y})}{\varepsilon}, \frac{d(x_0,\bar{x})}{n+1}, \frac{d(x_0,\bar{y})}{n+1} \right\}.
\]
Fix arbitrary \( x,y \in X \). If \( x,y \in X' \), then 
\[ 
d(\hat{f}(x),\hat{f}(y))  = d(f(x),f(y)) \leq k \cdot d(x,y) \leq c \cdot d(x,y) 
\] 
by our choice of \( k \in \omega \). If \( x,y \in X \setminus X' \), then \( d(x,y) \geq \varepsilon \) because both \( x \) and \( y \) are \(\varepsilon\)-isolated, and either \( \hat{f}(x) = \hat{f}(y) \) or \( d(\hat{f}(x),\hat{f}(y)) = d(\bar{x},\bar{y}) \). Therefore in both cases
\[ 
d(\hat{f}(x),\hat{f}(y)) \leq \frac{d(\bar{x},\bar{y})}{\varepsilon} \cdot \varepsilon \leq c \cdot d(x,y).
 \] 
Let now \( x \in X' \) and \( y \in X \setminus X' \), and assume without loss of generality that \( \hat{f}(y) = \bar{x} \) (the case \( \hat{f}(y) = \bar{y} \) is analogous, just systematically replace \( \bar{x} \) with \( \bar{y} \) in the argument below). Then either \( \bar{x} \in X' \), in which case \( d(\hat{f}(x),\hat{f}(y)) < n+1 \leq d(x,y) \leq c \cdot d(x,y) \) (since \( c \geq k \geq 1 \)), or else 
\[ 
d(\hat{f}(x),\hat{f}(y)) = d(x_0,\bar{x}) = \frac{d(x_0,\bar{x})}{n+1} \cdot n+1 \leq c \cdot d(x,y).
\]
The case \( x \in X \setminus X' \) and \( y \in X' \) can be treated similarly, so in all cases we obtained \( d(\hat{f}(x),\hat{f}(y)) \leq c \cdot d(x,y) \), as required. 
\end{proof}

We now want to show that the  \( \SLO^{\Lip(d)} \) principle holds for Borel subsets of \( X \), so let us fix arbitrary Borel \(A,B \subseteq X \). Assume first that \( B = X \). Then either \( A = X \), in which case the identity map on \( X \) witnesses \( A \leq_{\Lip(d)} B \), or else \( \neg A \neq \emptyset \), in which case any constant map with value \( \bar{x} \in \neg A \) witnesses \( B \leq_{\Lip(d)} \neg A \). The symmetric case \( B = \emptyset \) can be dealt with in a similar way, so in what follows we can assume without loss of generality that \( B \neq \emptyset,X \). Moreover, switching the role of \( A \) and \( B \) in the argument above we may further assume that \( A \neq \emptyset , X \). Set \( A' = A \cap X' \) and \( B' = B \cap X' \). Since \( X' \) has bounded diameter, by Theorem~\ref{th:unifcontandlip} there is a Lipschitz function \( f \colon (X',d) \to (X',d) \) such that either \( f^{-1}(B') = A' \) or \( f^{-1}(X' \setminus A') = B' \). Since \( \neg A \cap X' = X' \setminus A' \), applying Claim~\ref{claim:transfer} we get that either \( A \leq_{\Lip(d)} B \) or \( B \leq_{\Lip(d)} \neg A \), as desired.

Finally, let us show that the \( \Lip(d) \)-hierarchy on Borel subsets of \( X \) is also well-founded. Suppose not, and let \( (A_n)_{n \in \omega} \) be a sequence of Borel subsets of \( X \) such that \( A_{n+1} <_{\Lip(d)} A_n \) for every \( n \in \omega \). Notice that this in particular implies that \( A_n \neq \emptyset,X \) for every \( n \in \omega \). By Claim~\ref{claim:transfer} and our choice of the \( A_n \)'s, for all \( i < j \) there is no Lipschitz \( f \colon (X',d_{X'}) \to (X',d_{X'}) \) reducing \( A_i \cap X' \) to \( A_j \cap X' \).
Using Ramsey's theorem, we get that there is an infinite \( I \subseteq \omega \) such that either \( \forall i,j \in I \, (i < j \Rightarrow A_j \cap X' \leq_{\Lip(d_{X'})} A_i \cap X') \), or else \( \forall i,j \in I \, (i < j \Rightarrow A_j \cap X' \nleq_{\Lip(d_{X'})} A_i \cap X') \): in the former case the sequence \( (A_i \cap X' )_{i \in \omega} \) would give an infinite (strictly) descending chain in the \( \Lip(d_{X'}) \)-hierarchy on \( X' \), while in the latter it would give an infinite antichain (in the same hierarchy). Since \( X' \) has bounded diameter and all the sets \( A_i \cap X' \) are clearly Borel in it, both possibilities contradicts Theorem~\ref{th:unifcontandlip}, and hence we are done.
\end{proof}

\begin{corollary}
Let \( X = (X,d) \) be an ultrametric Polish space.Then the following are equivalent:
\begin{enumerate}
\item
the diameter of \( X \) is nontrivially unbounded;
\item
\( (\pow(\omega), \subseteq^*) \) embeds into  \( \Deg_{\boldsymbol{\Delta}^0_1}(\Lip(d)) \);
\item
the \( \Lip(d) \)-hierarchy on Borel (equivalently, clopen) subsets of \( X \) is very bad;
\item
the \( \Lip(d) \)-hierarchy on Borel (equivalently, clopen) subsets of \( X \) is not very good.
\end{enumerate}
Hence \( \Deg_{\boldsymbol{\Delta}^1_1}(\Lip(d)) \) is either very good or very bad.
\end{corollary}

\begin{proof}
By Theorem~\ref{th:unboundeddiam} and Theorem~\ref{th:notunbounded}.
\end{proof}

\begin{corollary} \label{cor:unboundeddiam}
Let \( X = (X,d) \) be a \emph{perfect} ultrametric Polish space. Then
\begin{enumerate}
\item
\( X \) has bounded diameter \( \iff \) the \( \Lip(d) \)-hierarchy on Borel (equivalently, clopen) subsets of \( X \) is very good;
\item
\( X \) has unbounded diameter \( \iff \) the \( \Lip(d) \)-hierarchy on Borel (equivalently, clopen) subsets of \( X \) is very bad, and in fact in this case
the partial order \( ( \pow(\omega), \subseteq^*) \) embeds into \( \Deg_{\boldsymbol{\Delta}^0_1}(\Lip(d)) \).
\end{enumerate}
\end{corollary}

Let us consider again the ultrametrics \( d_{\phi} \) introduced in Definition \ref{def: d_phi}. 

\begin{corollary}
Let \( \phi \colon \omega \to \RR^+  \) have unbounded range and suppose that \( \inf \rg (\phi) >0 \). Then \( (\pow(\omega), \subseteq^*) \) embeds into both the \( \Lip(d_\phi) \)- and \( \L(d_\phi) \)-hierarchy on clopen subsets of \( \pre{\omega}{\omega} \), and therefore both \( \Deg_{\boldsymbol{\Delta}^0_1}(\Lip(d_\phi)) \) and
\( \Deg_{\boldsymbol{\Delta}^0_1}(\L(d_\phi)) \)
are very bad. Conversely, if \( \phi \) has bounded range, then the \( \Lip(d_\phi) \)-hierarchy \( \Deg_{\boldsymbol{\Delta}^1_1}\) \((\Lip(d_\phi)) \) on Borel subsets of \( \pre{\omega}{\omega} \) is very good.
\end{corollary}

\begin{proof}
Observe that \( (\pre{\omega}{\omega},d_\phi) \) is a perfect ultrametric Polish space, and that it has unbounded diameter if and only if the  \( \rg( \phi) \) is unbounded in \( \RR^+ \); then apply Theorems~\ref{th:unboundeddiam} and~\ref{th:unifcontandlip}.
\end{proof}

\section{Nonexpansive reducibilities}

\begin{definition} \label{def:honestincreasingsequence}
Let \( X = (X,d) \) be an ultrametric Polish space. We say that \( R(d) \) contains an \emph{honest increasing sequence} if it contains a strictly increasing sequence  \( (r_n)_{n \in \omega} \) such that for some sequences \( (x_n)_{n \in \omega} \), \( (y_n)_{n \in \omega} \) of points in \( X \) 
%and some auxiliary sequence  \( (s_n)_{n \in \omega} \) of positive reals 
the following conditions holds:
\begin{enumerate}[(i)]
\item
\( d(x_n,x_m) = r_{\max\{ n,m \}} \) for all distinct \( n,m \in \omega \);
\item
\( d(x_0,y_0) < r_0 \) and \( d(x_{n+1},y_{n+1}) < d(x_n,y_n) \) for all \( n \in \omega \).
%\item
%\( x_n \) is not \( t_n \)-isolated.
\end{enumerate}
\end{definition}

The above condition is somewhat technical, but in case \( X = (X,d) \) is a perfect ultrametric Polish space it is immediate to check that \( R(d) \) contains an honest increasing sequence if and only if one of the following equivalent%
\footnote{To see that these two conditions are indeed equivalent, argue as in the first part of the proof of Theorem~\ref{th:unboundeddiam}.}
 conditions are satisfied:
\begin{enumerate}
\item
there is \( X' \subseteq X \) such that \( R(d_{X'}) \) has order type \( \omega \) (with respect to the usual ordering on \( \RR \));
\item
there is a sequence \( (x_n)_{n \in \omega} \) of points in \( X \) and a strictly increasing sequence \( (r_n)_{n \in \omega} \) of distances in \( R(d) \) such that \( d(x_n,x_m) = r_{\max \{n,m \}} \) for all distinct \( n,m \in \omega \).
\end{enumerate}

Notice also that if the diameter of an ultrametric Polish space \( X = (X,d) \) is nontrivially unbounded,  then \( R(d) \) contains an honest increasing sequence by the first part of the proof of Theorem~\ref{th:unboundeddiam}.

\begin{theorem} \label{th:increasingdistances} 
Let \( X = (X,d) \) be a ultrametric Polish space such that \( R(d) \) contains an honest increasing sequence. Then there is a map \( \psi \) from \( \pow(\omega)  \) into the clopen subsets of \( X \) such that for all \( a,b \subseteq \omega \)
\[ 
a \subseteq^* b \iff \psi(a) \leq_{\L(d)} \psi(b).
 \] 
\end{theorem}

\begin{proof}
Argue similarly to Theorem~\ref{th:unboundeddiam}, with the following variations:
\begin{enumerate}[(a)]
\item
let the sequences \( (x_n)_{n \in \omega} \), \( (y_n)_{n \in \omega} \), and \( (r_n)_{n \in \omega} \) constructed at the beginning of the proof of Theorem~\ref{th:unboundeddiam} be witnesses of the fact that \( R(d) \) contains an honest increasing sequence (forgetting about the extra properties required in Theorem~\ref{th:unboundeddiam}), and set \( s_n = d(x_n,y_n) \);%
\footnote{Clearly, the points \( x_n \) and \( y_n \) can again be chosen in any given countable dense set \( Q \subseteq X \).}
\item
given \( a \subseteq \omega \), define \( \psi(a) \) as before, i.e.\ set \( \psi(a) = \bigcup_{i \in \hat{a}}  B_d(x_i,s_i) \), where \( \hat{a} = \{ 2i \mid i \in \omega \} \cup \{ 2i+1 \mid i \in a \} \);
\item
to prove the backward direction, use an argument similar to that of Theorem~\ref{th:unboundeddiam}, but dropping any reference to the integer \( n \) (this simplification can be adopted here because we have to deal only with nonexpansive functions). More precisely:
 let \( f \) be a nonexpansive reduction of \( \psi(a) \) to \( \psi(b) \). Then for every \( i \in \hat{a} \) there is a unique \( j \in \omega \) such that \( f(B_d(x_i,s_i)) \subseteq B_d(x_j,s_j) \) (because of the choice of the  \( x_i \), \( y_i \)'s and the fact that \( f \) is nonexpansive). Arguing as in Claim~\ref{claim:Lip2}, one immediately sees that we cannot have \( j < i \) because in such case \( s_i \leq s_j \). Conclude as in the final part of the proof of Theorem~\ref{th:unboundeddiam}, using the fact that \( r_k < r_j \) for every \( j > k \). \qedhere
\end{enumerate}
\end{proof}

\begin{corollary} \label{cor:X_1}
There is an ultrametric Polish space \( X_1 = (X_1,d_1) \) whose set of nonzero distances \( R(d_1) \) is bounded away from \( 0 \) (hence it is countable and discrete) such that \( (\pow(\omega), \subseteq^*) \) embeds into the \( \L(d_1) \)-hierarchy on (clopen subsets of) \( X_1 \). Therefore \( \Deg(\L(d_1)) = \Deg_{\boldsymbol{\Delta}^0_1}(\L(d_1)) \) is very bad.
\end{corollary}

\begin{proof}
Let \( X_1 = \{ x^i_n \mid n \in \omega, i = 0,1 \} \) and set
\[ 
d_1(x^i_n,x^j_m) = 
\begin{cases}
0 & \text{if } n = m \text{ and } i = j \\
\frac{1}{2} + 2^{-(n+1)} & \text{if } n = m \text{ and } i \neq j \\
2 - 2^{-\max \{ n,m \}} & \text{if } n \neq m.
\end{cases}%
 \] 
It is easy to check that \( X_1 = (X_1,d_1) \) is an ultrametric Polish space. Moreover \( r \geq \frac{1}{2} \) for every \( r \in R(d_1) \), hence \( R(d_1) \) is bounded away from \( 0 \). Moreover, the sequences obtained by setting \( r_n  = 2 - 2^{-n} \), \( x_n = x_n^0 \), and  \( y_n = x_n^1 \) witness that \( R(d_1) \) contains an honest increasing sequence. Hence the result follows from Theorem~\ref{th:increasingdistances}.
\end{proof}

\begin{remark}
Notice that if an ultrametric Polish space \( X = (X,d) \) satisfies the hypothesis of Corollary~\ref{cor:X_1} (i.e.\ it is such that \( R(d) \) is bounded away from \( 0 \)), then its \( \Lip(d) \)-hierarchy is always (trivially) very good by Theorem~\ref{th:notunbounded} and the fact that all its points are \( \varepsilon  \)-isolated for \( \varepsilon = \inf R(d) > 0 \).
\end{remark}

\begin{corollary} \label{cor:specialincreasing}
Given \( \phi \colon \omega \to \RR^+  \) such that \( \inf \rg( \phi)>0 \), if \( \rg(\phi) \) contains an increasing \( \omega \)-sequence then \( ( \pow(\omega), \subseteq^*) \) embeds into the \( \L(d_\phi) \)-hierarchy on clopen subsets of \( \pre{\omega}{\omega} \), and therefore \( \Deg_{\boldsymbol{\Delta}^0_1}(\L(d_\phi)) \) is very bad.
\end{corollary}

\begin{proof}
Notice that \( (\pre{\omega}{\omega}, d_\phi ) \) is always a perfect Polish space, and that \( R(d_\phi) \) has an honest increasing sequence if and only if \( \rg(\phi) \) contains an increasing \( \omega \)-sequence. Then apply Theorem~\ref{th:increasingdistances}.
\end{proof}

\begin{proposition} \label{prop:descendingdistances}
Suppose that \( X = (X,d) \) is an ultrametric Polish space such that \( R(d) \) is either finite or a descending (\( \omega \)-)sequence converging to \( 0 \), let \( I \leq \omega \) be the cardinality of \( R(d) \), and let \( \rho \) be the unique order-preserving map from \( \{ 2^{-i} \mid i < I \} \) and \( R(d) \). Then there is a closed set \( C \subseteq \pre{\omega}{\omega} \) and a bijection \( f \colon C \to X \) such that for all \( x,y \in X \)
\begin{equation} \tag{$*$} \label{eq:isom}
d(x,y) = \rho(\bar{d}(f^{-1}(x),f^{-1}(y))).
\end{equation}
In particular, the structures \( ( \pow(X), \leq_{\L(d)}, \neg ) \) and \( ( \pow(C), \leq_{\L(\bar{d})}, \neg ) \) are isomorphic.
\end{proposition}

\begin{proof}
Let us first assume that \( I = \omega \), i.e.\ that \( R(d) \) is a descending (\(\omega \)-)sequence converging to \( 0 \). Inductively define the family 
\( (A_s)_{s \in \pre{< \omega}{\omega}} \) of subsets of \( X \) by induction on \( \leng(s) \) as follows. Set 
\( A_\emptyset = X \). Then let \( \{ B_{s,j} \mid j < J \} \) (for some  \( J \leq \omega \)) be an enumeration without repetitions 
of the collection \( \{ B_d(x, \rho(2^{-\leng(s)})) \mid x \in A_s \} \), and set 
\( A_{s {}^\smallfrown{} j} = B_{s,j} \) if \( j < J \) and \( A_{s {}^\smallfrown{}  j} = \emptyset \) otherwise. It is easy to check that the family \( (A_s)_{s \in \pre{< \omega}{\omega}} \) is a Luzin scheme with vanishing diameter consisting of clopen sets. Hence letting \( C = \{ x \in \pre{\omega}{\omega} \mid \bigcap_{n \in \omega} A_{x \restriction n} \neq \emptyset \} \) and \( f \colon C \to X \) be defined by letting \( f(x) \) be the unique element of \( \bigcap_{n \in \omega} A_{x \restriction n} \), we get that \( C \) and \( f \) are as required.

Assume now that \( I \) is finite, so that, in particular, \( X \) is a discrete space. Inductively define the sets \( A_s \) as above for all \( s \in \pre{< \omega}{\omega} \) of length \( \leq I \). Then if \( \leng(s) = I \) the set \( A_s \) is either empty or a singleton. Letting \( C = \{ s {}^\smallfrown{}  0^{(\omega)} \mid \leng(s) = I, A_s \neq \emptyset  \} \) and defining \( f \colon C \to X \) by letting \( f(s {}^\smallfrown{} 0^{(\omega)}) \) be the unique element of \( A_s \) we again have that \( C \) and \( f \) are as required.

For the last part, notice that  the map \( \pow(X) \to \pow(C) \colon A \mapsto f^{-1}(A) \) is the desired isomorphism. To see this, simply notice that~\eqref{eq:isom} implies that \( \L(d) = \{ f \circ h \circ f^{-1} \mid h \in \L(\bar{d}_C) \} \).
\end{proof}

\begin{theorem}\label{th:descendingdistances}
Suppose that \( X = (X,d) \) is an ultrametric Polish space such that \( R(d) \) is either finite or a descending  (\( \omega \)-)sequence converging to \( 0 \). Then the \( \L(d) \)-hierarchy \( \Deg_{\boldsymbol{\Delta}^1_1}(\L(d)) \) on Borel subsets of \( X \) is very good.
\end{theorem}

\begin{proof}
By Proposition~\ref{prop:descendingdistances}, it is clearly enough to show that the \( \L(\bar{d}_C) \)-hierarchy on Borel subsets of \( C \) is very good: but this easily follows from the existence of a nonexpansive retraction of \( (\pre{\omega}{\omega}, \bar{d}) \) onto \( (C, \bar{d}_C) \), Lemma~\ref{lemma:retraction}, and the fact that the \( \L(\bar{d}) \)-hierarchy on the Borel subsets of \( \pre{\omega}{\omega} \) is very good.
\end{proof}

\begin{corollary} \label{cor:specialconvergingto0}
Let \( \phi \colon \omega \to \RR^+  \) and suppose that \( \rg(\phi) \) is finite (so that trivially \( \inf \rg(\phi) > 0 \)).
 %or a decreasing (\( \omega \)-)sequence converging to \( 0 \). 
Then the \( \L(d_\phi) \)-hierarchy \( \Deg_{\boldsymbol{\Delta}^1_1}(\L(d_\phi)) \) on Borel subsets of \( \pre{\omega}{\omega} \) is very good.
\end{corollary}

\begin{proof}
Simply observe that under our assumptions the set \( R(d_\phi) \) is always an \( \omega \)-sequence converging to \( 0 \), and then apply Theorem~\ref{th:descendingdistances}.
\end{proof}

Let us now consider the general problem of determining the character of the \( \L(d_\phi) \)-hierarchy on Borel subsets of \( \pre{\omega}{\omega} \) for an arbitrary \( \phi \colon \omega \to \RR^+ \) with \(\inf \rg (\phi) >0\). By Corollary~\ref{cor:specialincreasing}, if \( \rg(\phi) \) contains
an increasing \(\omega\)-sequence, then \( \Deg_{\boldsymbol{\Delta}^1_1}\) \((\L(d_\phi)) \) is very bad, hence we can assume without loss of generality that \( \rg(\phi) \) has order type%
\footnote{Given a linear order \( L = (L, \leq) \), we denote by \( L^* \) the reverse linear oder induced by \( L \), i.e.\ \( L^* = (L,\leq^{-1}) \). Since \( \alpha = \{ \beta \mid \beta < \alpha \} \) (for every ordinal \(\alpha\)), we tacitly identify \(\alpha\) with the linear order \( \alpha =  (\alpha, \leq) \), so that \( \alpha^* = (\alpha, \geq) \).}
 \( \alpha^* \) for some countable ordinal \(\alpha\). Corollary~\ref{cor:specialconvergingto0} considered the subcase where \( \alpha \) is finite:
%or else \( \alpha = \omega \) and \( \inf \rg(\phi) = 0 \). 
 the next proposition considers instead the special (but yet significant) subcase where \( \alpha=\omega \) and \( \phi \) is injective. 
%(and has range bounded away from \( 0 \)). 

\begin{notation}
Given a set \( A \subseteq \pre{\omega}{\omega} \) and a finite sequence \( s \in \pre{< \omega}{\omega} \), let \( s {}^\smallfrown{} A = \{ s {}^\smallfrown{} x \mid x \in A \} \). When \( \leng(s) = 1 \), we simplify the notation by setting \( n {}^\smallfrown{} A = \langle n \rangle {}^\smallfrown{} A \), and
with a little abuse of notation we set \( r {}^\smallfrown{} A = \{ r {}^\smallfrown{} x \mid x \in A \} \subseteq \{ r \} \times \pre{\omega}{\omega} \) also when \( r \) is not a natural number.
Finally, given a family \( (A_n)_{n \in \omega} \) of subsets of \( \pre{\omega}{\omega} \), we set \( \bigoplus_{n \in \omega} A_n = \bigcup_{n \in \omega} n {}^\smallfrown{} A_n \).
\end{notation}

\begin{theorem} \label{prop:specialalpha*} 
Let \( \phi \colon \omega \to \RR^+ \) be such that \( \inf \rg(\phi) > 0 \), and suppose that \(\phi \) is injective and that \( \rg(\phi) \) has order type \( \omega^* \). Then the \( \L(d_\phi) \)-hierarchy \( \Deg_{\boldsymbol{\Delta}^1_1}(\L(d_\phi)) \) on the Borel subsets of \( \pre{\omega}{\omega} \) is very good.
%semi-linearly ordered, and hence not bad. 
\end{theorem} 

\begin{proof}%[Sketch of the proof]
Using the usual game-theoretic arguments (see e.g.\ \cite{Andretta:2007}), it is easy to see that if a Borel \( A \subseteq \pre{\omega}{\omega} \) is \( \L(\bar{d}) \)-selfdual, then its \( \L(\bar{d}) \)-degree \( [A]_{\L(\bar{d})} \) is followed by an \( \omega_1 \)-chain of \( \L(\bar{d}) \)-selfdual degrees \( (\mathscr{L}^{(\alpha)} [A]_{\L(\bar{d})})_{\alpha < \omega_1} \), where the \( \mathscr{L}^{(\alpha)} [A]_{\L(\bar{d})} \) are recursively defined as follows:
\begin{enumerate}[(i)]
\item
\( \mathscr{L}^{(0)} [A]_{\L(\bar{d})} = [A]_{\L(\bar{d})} \);
\item
\( \mathscr{L}^{(\alpha+1)} [A]_{\L(\bar{d})} = [0 {}^\smallfrown{} C]_{\L(\bar{d})} \) for some/any \( C \in \mathscr{L}^{(\alpha)} [A]_{\L(\bar{d})} \);
\item
for limit \( \alpha \)'s, \( \mathscr{L}^{(\alpha)} [A]_{\L(\bar{d})} = [ \bigoplus_{n \in \omega} C_n ]_{\L(\bar{d})} \), where \( C_n \in \mathscr{L}^{(\alpha_n)} [A]_{\L(\bar{d})} \) for each \( n \in \omega \), and \( (\alpha_n)_{n \in \omega} \) is some/any increasing sequence cofinal in \(\alpha\).
\end{enumerate}

We will use the following known facts about the Baire space $(\pre{\omega}{\omega},\bar{d})$.
\begin{enumerate}[\( \bullet \)] 
\item 
A set $A$ is \emph{self-contractible} (i.e.\ reducible to itself via a contraction) if and only if it is $\L(\bar{d})$-nonselfdual; in this case the iterates of the contraction are reductions of $A$ to itself and have a unique common fixed point (see~\cite[Corollary 4.4]{MottoRos:2012c}). 
\item 
The $\L(\bar{d})$-nonselfdual degrees coincide with the $\W(\bar{d})$-nonselfdual degrees (see e.g.~\cite[Theorem 3.1]{VanWesep:1978}). 
\item 
Every $\Lip(\bar{d})$-selfdual degree \( [A]_{\Lip(\bar{d})} \) is of the form $\bigcup \{[0^{(n)} {}^\smallfrown{}  A']_{\L(\bar{d})} \mid {n<\omega}\}$ for some \( \L(\bar{d}) \)-selfdual set $A'$; if instead \( [A]_{\Lip(\bar{d})} \) is \( \Lip(\bar{d}) \)-nonselfdual, then \( [A]_{\Lip(\bar{d})} = [A]_{\L(\bar{d})} \)  (see~\cite{MottoRos:2010}). 
\item 
If $A<_{\Lip(\bar{d})} B$, then for all $\varepsilon \in \RR^+$ there is a Lipschitz reduction of $A$ to $B$ with constant $\varepsilon$ (see the end of Section 4 in~\cite{MottoRos:2012c}). 
\item 
Let \( \W  = \W(\pre{\omega}{\omega}) \) be the set of all continuous functions from \( \pre{\omega}{\omega} \) into itself, which is clearly a reducibility. Then every $\W$-selfdual degree \( [A]_\W \) is of the form $\bigcup \{ \mathscr{L}^{(\alpha)}([A']_{\L(d)}) \mid \alpha < \omega_1 \}$ for some \( \L(\bar{d}) \)-selfdual set $A'$  (see e.g.\ \cite{Andretta:2007}).
\end{enumerate} 

%Suppose $\phi \colon \omega \to \RR^+$ is injective, and that $\gamma^*$ is the order type of $\rg(\phi)$. 
Note that $(\pre{\omega}{\omega}, d_{\phi})$ is isometric to the space $Y=\bigcup_{r \in \rg(\phi)}  r^\smallfrown  \pre{\omega}{\omega} $ equipped with the ultrametric (which with a little abuse of notation will be denoted by \( d_\phi \) again)
\[ 
d_{\phi} (r^\smallfrown x,s^\smallfrown y) = 
\begin{cases}
0 & \text{if } r=s \text{ and }x=y, \\
\max\{r,s\} & \text{if }r \neq s, \\ 
r \cdot 2^{-(n+1)} & \text{if } r=s \text{ and } n \text{ is least such that } x(n) \neq y(n). 
\end{cases} 
 \] 

\begin{claim} \label{claim:normalforms}
Every Borel subset \( \bar{C} \) of $Y = (Y, d_{\phi})$ is $\L(d_\phi)$-equivalent to one of the following \emph{(\(\L(d_\phi) \)-)normal forms \( \bar{A} \)} (where in what follows $A,A_n \subseteq \pre{\omega}{\omega}$ and \( < \) is the usual order on the reals): 
\begin{enumerate} 
\item 
$\bar{A}=\bigcup_{n\in\omega} r_n^\smallfrown A_n$, where the sequence of the \( A_n \)'s  is $<_{\Lip(\bar{d})}$-increasing and the sequence $(r_n)_{n \in \omega}$ in $\rg(\phi)$ is strictly \( < \)-decreasing. 
\item 
$\bar{A}=\bigcup_{n\in\omega} r_n^\smallfrown A_n$, where the sequence of the \( A_n \)'s  is $<_{\L(\bar{d})}$-increasing, $A_m\equiv_{\Lip(\bar{d})} A_n$ for all $m,n \in \omega$, and the sequence $(r_n)_{n \in \omega}$ in $\rg(\phi)$ is strictly \( < \)-decreasing.  
\item 
$A$ is $\L(\bar{d})$-nonselfdual and 
\begin{enumerate} 
\item 
$\bar{A}=r^\smallfrown A $ for some $r \in \rg(\phi)$, or 
\item 
$\bar{A}=(r_0^\smallfrown A)\cup (r_1^\smallfrown (\neg A))$ for some $r_0, r_1 \in \rg(\phi)$ with $r_0>r_1$, or 
%\item 
%\( \bar{A}=\bigcup_{i \in \omega} (r_{i} {}^\smallfrown{}  A) \cup (r {}^\smallfrown{}  (\neg A)) \) for some  strictly decreasing sequence $(r_n)_{n \in \omega}$ in $\rg(\phi)$ and some $r\in \rg(\phi)$ with $r< r_i$ for all $i\in\omega$, or  
\item 
\( \bar{A}=\bigcup_{i \in \omega} r_{2i} {}^\smallfrown{}  A \cup \bigcup_{i \in \omega} r_{2i+1} {}^\smallfrown{}  (\neg A) \) for some  strictly \( < \)-decreasing sequence $(r_n)_{n \in \omega}$ in $\rg(\phi)$. 
\end{enumerate} 
\item 
$A$ is $\mathsf{L}(\bar{d})$-selfdual and 
\begin{enumerate} 
\item $\bar{A}=r {}^\smallfrown{}  A $ for some $r\in \rg(\phi)$, or 
\item \( \bar{A}=\bigcup_{n\in\omega} r_n^\smallfrown A \)  for some  strictly \( < \)-decreasing sequence $(r_n)_{n \in \omega}$ in $\rg(\phi)$. 
\end{enumerate} 
\end{enumerate} 
\end{claim}

\begin{proof}[Proof of the Claim]
Let us sketch how to obtain these normal forms. We will often use the following easy fact. Let \( D \subseteq \rg(\phi) \), \( \rho \colon D \to \rg(\phi) \) be a non-\( < \)-increasing map, \( \{ f_r \colon \pre{\omega}{\omega} \to \pre{\omega}{\omega} \mid r \in D \} \subseteq \L(\bar{d}) \), and \( f' \colon \bigcup_{r \in \rg(\phi) \setminus D} (r {}^\smallfrown{} \pre{\omega}{\omega}) \to Y \) be a nonexpansive map (with respect to \( d_\phi \)): then the map \( f \colon Y \to Y \) defined by
\[ 
f(r {}^\smallfrown{} x) = 
\begin{cases}
\rho(r) {}^\smallfrown{}  f_r(x) & \text{if } r \in D \\
f'(r {}^\smallfrown{} x) & \text{otherwise}
\end{cases}%
 \] 
is in \( \L(d_\phi) \).

Now let $\bar{C}=\bigcup_{r\in \rg(\phi)} (r^\smallfrown C_r)$ be an arbitrary Borel subset of $(Y, d_{\phi})$, and set $\mathcal{C}=\{C_r\mid r\in \rg(\phi)\}$, so that each \( C_r \) is a Borel subset of \( \pre{\omega}{\omega} \). 
If $\mathcal{C}$ has no  $\mathsf{Lip}(\bar{d})$-maximal element, choose a strictly \( < \)-decreasing sequence $(r_n)_{n\in\omega}$ in $\rg(\phi)$ such that $(C_{r_n})_{n\in\omega}$ is strictly $<_{\mathsf{Lip}(d_{\phi})}$-increasing and $<_{\mathsf{Lip}(d_{\phi})}$-cofinal in $\mathcal{C}$. Then $\bar{A}=\bigcup_{n\in\omega} C_{r_n}$ is in the normal form (1), and moreover it is easy to see that $\bar{A}\equiv_{\mathsf{L}(d_{\phi})}\bar{C}$. 
Otherwise, if $\mathcal{C}$ has a $\mathsf{Lip}(\bar{d})$-maximal element but no $\mathsf{L}(\bar{d})$-maximal element, then we can similarly find a set $\bar{A}$ in the normal form (2) which is $\mathsf{L}(d_{\phi})$-equivalent to $\bar{C}$. 

Now suppose that there is an  $\mathsf{L}(\bar{d})$-maximal element \( B \) among the sets in $\mathcal{C}$. Suppose first that \( B \) is $\mathsf{L}(\bar{d})$-nonselfdual. 

If there is no $C\in \mathcal{C}$ with $C\equiv_{\mathsf{L}(\bar{d})} \neg B$, then we choose some $r\in \rg(\phi)$ with $C_r\equiv_{\mathsf{L}(\bar{d})} B$. Using the assumption $\inf \rg(\phi)>0$ and the fact mentioned at the beginning of the proof that $\mathsf{L}(\bar{d})$-nonselfdual sets are self-contractible with arbitrarily small Lipschitz constant, it follows that $\bar{A}=r^\smallfrown C_r\equiv_{\mathsf{L}(d_{\phi})} \bar{C}$, and $\bar{A}$ is in the normal form (3a). 
(For the nontrivial reduction, for each \( t \in \rg(\phi) \) choose a \( \L(\bar{d}) \)-reduction \( f_t \colon \pre{\omega}{\omega} \to \pre{\omega}{\omega} \) of \( C_t \) to \( C_r \), let \( \varepsilon \in \RR^+ \) be such that \( \max \rg(\phi) \cdot \varepsilon \leq \inf \rg(\phi) \), and let \( g \colon \pre{\omega}{\omega} \to \pre{\omega}{\omega} \) be a Lipschitz map with constant \( \varepsilon \) reducing \( C_r \) to itself. Define \( f \colon Y \to Y \) by setting \( f( t {}^\smallfrown{} x) = r {}^\smallfrown{}  g(f_t(x)) \) for every \( t \in \rg(\phi) \) and \( x \in \pre{\omega}{\omega} \): it is easy to check that \( f \in \L(d_\phi) \) reduces \( \bar{C} \) to \( \bar{A} \).)

If there is a \( < \)-minimal $s\in\rg(\phi)$ with $C_s\equiv_{\mathsf{L}(\bar{d})} \neg B$, let $r$ be either the \( < \)-minimal element of \( \rg(\phi) \) with $C_r\equiv_{\mathsf{L}(\bar{d})} B$, or the \( < \)-largest element of \( \rg(\phi) \) satisfying both $C_r\equiv_{\mathsf{L}(\bar{d})} B$ and \( r < s \).  
Then $\bar{A}=r^\smallfrown C_r\cup s^\smallfrown C_s$ is in the normal form (3b), and arguing as above one can check that $\bar{C}\equiv_{\mathsf{L}_{d(\phi)}} \bar{A}$ using the assumption $\inf \rg(\phi)>0$ and the previously mentioned fact about self-contractions. (For the nontrivial reduction, notice that we can assume without loss of generality that \( r < s \) (otherwise we simply switch the role of \( C_r \) and \( C_s \)). Let \( D = \{ t \in \rg(\phi) \mid C_t \equiv_{\L(\bar{d})} \neg B \} \), so that \( s = \min D \). For \( t \in \rg(\phi) \), set \( \rho(t)  =  s \) if \( t \in D \)  and \( \rho(t) = r \) otherwise. Let \( f_t \) be a \( \L(\bar{d}) \)-reduction of \( C_t \) to \( C_s \) if \( t \in D \) and of \( C_t \) to \( C_r \) otherwise. Let \( \varepsilon \) and \( g \) be as above. Then the map \( f \colon Y \to Y \) defined by \( f(t {}^\smallfrown{} x) = s {}^\smallfrown{}  f_t(x) \) if \( t \in D \) and \( f(t {}^\smallfrown{} x) = r {}^\smallfrown{} g(f_t(x)) \) otherwise is an \( \L(d_\phi) \)-reduction of \( \bar{C} \) to \( \bar{A} \).)

If there are unboundedly many $s\in\rg(\phi)$ with $C_s\equiv_{\mathsf{L}(\bar{d})} \neg B$ and an \( < \)-minimal \( r \in \rg(\phi) \) with \( C_r \equiv_{\L(\bar{d})} B \), argue as in the previous paragraph switching the role of \( B \) and \( r \) with, respectively, \( \neg B \) and \( s \).

In the remaining case there are unboundedly many $r\in\rg(\phi)$ with $C_r\equiv_{\mathsf{L}(\bar{d})} B$ and unboundedly many $s\in\rg(\phi)$ with $C_s\equiv_{\mathsf{L}(\bar{d})} \neg B$. In this situation it is easy to see that $\bar{C}$ is $\mathsf{L}(d_{\phi})$-equivalent to a set $\bar{A}$ in the normal form (3c). 

Finally, suppose that $\mathcal{C}$ has a $\mathsf{L}(\bar{d})$-maximal element $B$ and that $B$ is \( \L(\bar{d}) \)-selfdual. It follows from the remarks at the beginning of the proof that there is an \( \L(\bar{d}) \)-nonselfdual set $A$ with $B\in\mathscr{L}^{(\lambda+n)}[A\oplus(\neg A)]_{\mathsf{L}(\bar{d})}$ for some $n\in\omega$ and $\lambda=0$ or $\lambda$ a countable limit ordinal. Set \( D = \{ r \in \rg(\phi) \mid C_r \in \bigcup_{j \in \omega} \mathscr{L}^{(\lambda+j)}[A\oplus(\neg A)]_{\mathsf{L}(\bar{d})}\} \), and
define the \emph{index} of any $r\in D$ as $i(r)=r\cdot 2^{-(j+1)}$, where \( j \) is the unique natural number such that $C_r\in \mathscr{L}^{(\lambda+j)}[A\oplus(\neg A)]_{\mathsf{L}(\bar{d})}$. Then for any $r,s\in D$ for which $i(r)\leq i(s)$ there is an \( \L(d_\phi) \)-map \( f \) such that $f(s^\smallfrown {}^\omega \omega) \subseteq r{}^\smallfrown {}^\omega \omega$ and \( f \) reduces $s^\smallfrown C_s$ to $r^\smallfrown C_r$.

Suppose first that there is \( j \leq n \) such that $C_{r_m}\in \mathscr{L}^{(\lambda+j)}[A\oplus(\neg A)]_{\mathsf{L}(\bar{d})}$ for some strictly \( < \)-descending sequence \( (r_m)_{m \in \omega} \) of distances in \( \rg(\phi) \), and let \( k \) be the largest of such \( j \)'s. If \( n = k \), then $\bar{A}=\bigcup_{m\in\omega} r_m^\smallfrown C_{r_m}$ is in the normal form (4b) and $\bar{A}\equiv_{\L(d_{\phi})}\bar{C}$. If \( n>k \), let \( r \) be \( < \)-smallest in \( \rg(\phi) \) such that \( C_r \equiv_{\L(\bar{d})} B \). If
\( \inf \rg(\phi) < r \cdot 2^{n-k} \), then using the fact that \( C_r \) is reducible to each of the \( C_{r_m} \)'s with some Lipschitz function with constant \( 2^{n-k} \) we get that \( \bar{C} \equiv_{\L(d_\phi)} \bar{C}' \), where \( \bar{C}' = \bar{C} \setminus \left ( \bigcup_{t \geq r} t {}^\smallfrown{} \pre{\omega}{\omega} \right ) \). Applying recursively this same procedure, after finitely many steps we will end up with a set \( \bar{C}^* \equiv_{\L(d_\phi)} \bar{C} \) such that either the \( C_{r_m} \) are \( \L(\bar{d}) \)-maximal in \( \mathcal{C}^* \), or else there is an \( < \)-smallest \( r \) such that \( C_r \) is \( \L(\bar{d}) \)-maximal in \( \mathcal{C}^* \), \( C_r \in \mathscr{L}^{(\lambda+n^*)}[A \oplus (\neg A)]_{\L(\bar{d})}\) for some \( k < n^* \leq n \), and \( r \cdot 2^{n^*-k} \leq \inf \rg(\phi) \). In the former case we
again easily get that $\bar{A}=\bigcup_{m\in\omega} r_m^\smallfrown C_{r_m}$ is in the normal form (4b) and $\bar{A}\equiv_{\mathsf{L}(d_{\phi})}\bar{C}^*  \equiv_{\L(d_\phi)} \bar{C}$. In the latter case, we get that $\bar{A} = r {}^\smallfrown{} C_r$ is in normal form (4a) and \( \bar{A} \equiv_{\mathsf{L}(d_\phi)} \bar{C}^* \equiv_{\L(d_\phi)} \bar{C} \). (To see that \( \bar{C}^* \leq_{\L(d_\phi)} \bar{A} \), which is the only nontrivial reduction, notice that we 
may assume without loss of generality that all the \( C_{r_m} \)'s equal a fixed set \( C \neq \pre{\omega}{\omega} \), that
 \( C_r = 0^{(n^*-k)} {}^\smallfrown{} C \), and that for \( t \notin \{ r \} \cup \{ r_m \mid m \in \omega \} \) either \( C_t = \emptyset \) or \( C_t = 0^{(i_t+1)} {}^\smallfrown{}  C \) for some \( i_t < n^*-k \). 
%(this last possibility can occur only finitely many times by our choice of \( k \)). 
Fix \( t \in \rg(\phi) \). If \( t \geq r \) then
  let \( f_t \colon \pre{\omega}{\omega} \to \pre{\omega}{\omega} \) be a \( \L(\bar{d}) \)-reduction of \( C_t \) to \( C_r \). If  \( t = r_m \) for some \( m \in \omega \), define \( f_t \) by setting \( f_t(x) = 0^{(n^*-k)} {}^\smallfrown{} x \) for all \( x \in \pre{\omega}{\omega} \). Finally, if \( t < r \) and \( t \neq r_m \), then let \( f_t \) be a constant
 map with value \( 0^{(n^*-k)} {}^\smallfrown{}  y \) for some fixed \( y \notin C \) if \( C_t =\emptyset \), and otherwise set \( f_t(x) = 0^{(n^*-k -i_t - 1)} {}^\smallfrown{} x \) for all \( x \in \pre{\omega}{\omega} \). Then the map \( f \colon Y \to Y \) defined by setting
\( f(t {}^\smallfrown{} x) = r {}^\smallfrown{} f_t(x) \) for all \( t \in \rg(\phi) \) and \( x \in \pre{\omega}{\omega} \) is a \( \L(d_\phi) \)-reduction of \( \bar{C} \) to \( \bar{A} \).)

Therefore we may assume without loss of generality that \( D \) is finite. Actually, applying the standard arguments used above it is not difficult to see that we may also assume that there are \( m \in \omega \), a strictly \( < \)-decreasing sequence \( r_0, \dotsc, r_m \in \rg(\phi) \), and a strictly decreasing sequence \( n_0, \dotsc, n_m \in \omega \) such that:
\begin{itemize}
\item
\( C_{r_k} \in \mathscr{L}^{(\lambda+n_k)}[A \oplus(\neg A)]_{\L(\bar{d})} \) for all \( k \leq m \);
\item
\( i(r_k) < i(r_{k+1}) \) for all \( k < m \);
\item
\( C_t = \emptyset \) for all \( t \geq r_m \) which are not of the form \( r_k \) for some \( k \leq m \);
\item
\( C_t <_{\Lip(\bar{d})} C_{r_m} \) for all \( t < r_m \).
\end{itemize}

Assume first that \( \lambda > 0 \). Then without loss of generality we may assume that 
\( C_{r_k} = 0^{(n_k)} {}^\smallfrown{} \bigoplus_{l \in \omega} (0^{(l)} {}^\smallfrown{} C'_l) \) for all \( k \leq m \), where the \( C'_l \)'s are strictly \( \L(\bar{d}) \)-increasing
subsets of \( \pre{\omega}{\omega} \) such that their \( \L(\bar{d}) \)-degrees are cofinal below \( \mathscr{L}^{(\lambda)}[A \oplus (\neg A)]_{\L(\bar{d})} \). Notice that in this case \( i(r_k) \) measures the 
\( d_\phi \)-distance between each pair of subsets of \( C_{r_k} \) of the form \( 0^{(n_k)} {}^\smallfrown{} l {}^\smallfrown{}0^{(l)} {}^\smallfrown{}  C'_l \). Assume first
 that there is \( l \in \omega \) such that \( C_t \leq_{\L(\bar{d})} C'_l \) for all \( t < r_m \). Then it is not hard to see that 
\( \bar{A} = r_0 {}^\smallfrown{} C_{r_0} \) is in normal form (4a) and \( \bar{A} \equiv_{\L(d_\phi)} \bar{C} \). (An 
\( \L(d_\phi) \)-reduction \( f \) of \( \bar{C} \) to \( \bar{A} \) may be defined on sets of the form \( t {}^\smallfrown{}  \pre{\omega}{\omega} \) for \( t < r_m \) by
 fixing \( l'\geq l \) such that \( 2^{-l'} \leq \inf \rg(\phi) \) and an \( \L(\bar{d}) \)-reduction $f_t$ of \( C_t \) to \( C'_{l'} \), and then setting 
%\( f(t {}^\smallfrown{}  x) = 0^{(n_0)} {}^\smallfrown{}  (l + l') {}^\smallfrown{}  0^{(l')} {}^\smallfrown{} f_t(x) \);
\( f(t {}^\smallfrown{}  x) = r_0 {}^\smallfrown 0^{(n_0)} {}^\smallfrown{}  l' {}^\smallfrown{} 0^{(l')} {}^\smallfrown{}  f_t(x) \);
for \( t \geq r_m \), the map \( f \) may be defined on 
\( t {}^\smallfrown{} \pre{\omega}{\omega} \) in the obvious way using the property of the \( i(r_k) \)'s mentioned above.) Now assume instead that the
 family \( \{ C_t \mid t < r_m \} \) is \( \L(\bar{d}) \)-cofinal below \( \bigoplus_{l \in \omega} C'_l \equiv_{\L(\bar{d})} \bigoplus_{l \in \omega} 0^{(l)} {}^\smallfrown C'_l \in \mathscr{L}^{(\lambda)}[A \oplus (\neg A)]_{\L(\bar{d})} \). Then using arguments similar to the one already
 applied, one gets that if \( i(r_0) \leq \inf \rg(\phi) \) then we can again set \( \bar{A} = r_0 {}^\smallfrown{} C_0 \), so that \( \bar{A} \) is in normal
 form (4a), and prove that \( \bar{A} \equiv_{\L(d_\phi)} \bar{C} \), while if \( i(r_0) > \inf \rg(\phi) \) then we may choose a strictly decreasing sequence
 \( (t_h)_{h \in \omega} \) so that \( t_0 < \min \{r_m, i(r_0) \} \) and the \( C_{t_h} \)'s are \(\leq_{\L(\bar{d})} \)-increasing, all in the same \( \Lip(\bar{d}) \)-degree,
 and cofinal below \( \bigoplus_{l \in \omega} 0^{(l)} {}^\smallfrown C'_l \), and then prove that \( \bar{A} = \bigcup_{h \in \omega} t_h {}^\smallfrown{} C_{t_h} \) is in
 normal form (2) and \( \L(d_\phi) \)-equivalent to \( \bar{C} \).

Finally, let $\lambda=0$. In this case we may assume without loss of generality that 
\( C_{r_k} = 0^{(n_k)} {}^\smallfrown{} (A \oplus \neg A) \) for all \( k \leq m \), and \( i(r_k) \) measures the distance between the copies of $A$ and $\neg A$ in $C_{r_k}$. Let us first suppose that there are arbitrarily small $r,s>\inf \rg(\phi)$ with $C_r\equiv_{\mathsf{L}(d_{\phi})}A$ and  $C_s\equiv_{\mathsf{L}(d_{\phi})} \neg A$. If $i(r_0) \leq \inf \rg(\phi)$, we let $\bar{A}=r_0^\smallfrown C_{r_0}$; then $\bar{A}$ is in the normal form 
(4a) and arguing as above we get $\bar{A}\equiv_{\mathsf{L}(d_{\phi})}\bar{C}$. If $i(r_0)>\inf \rg(\phi)$, we choose a strictly decreasing sequence $(t_h)_{h\in\omega}$ 
in $\rg(\phi)$ with \( t_0 < \min \{ r_m, i(r_0) \} \), $C_{t_{2p}}\equiv_{\mathsf{L}(\bar{d})} A$ and $C_{t_{2p+1}}\equiv_{\mathsf{L}(\bar{d})} \neg A$, and let $\bar{A}=\bigcup_{h\in\omega} t_h {}^\smallfrown{}  C_{t_h}$. Then $\bar{A}$ is in the normal form (3c) and, arguing as in the case \( \lambda > 0 \), we get $\bar{A}\equiv_{\mathsf{L}(d_{\phi})}\bar{C}$. 
Next, let us suppose that there are no $r,s<i(r_0)$ in $\rg(\phi)$ with $C_r \equiv_{\mathsf{L}(\bar{d})} A$ and $C_s \equiv_{\mathsf{L}(\bar{d})} \neg A$. Let $\bar{A}=r_0^\smallfrown C_{r_0}$. Then $\bar{A}$ is in the normal form (4a), and using the self-contractibility of \( A \) and \( \inf \rg(\phi) > 0 \) we again obtain $\bar{A}\equiv_{\mathsf{L}(d_{\phi})}\bar{C}$. 
Finally, suppose that there are $r,s<i(r_0)$ in $\rg(\phi)$ with $C_r \equiv_{\mathsf{L}(\bar{d})} A$ and $C_s \equiv_{\mathsf{L}(\bar{d})} \neg A$ and that there is an \( < \)-minimal $r\in\rg(\phi)$ with $C_r \equiv_{\mathsf{L}(\bar{d})} A$ (the analogous situation in which there is a minimal $r\in\rg(\phi)$ with $C_r\equiv_{\mathsf{L}(\bar{d})} \neg A$ can be treated similarly). We consider the \( < \)-smallest $s\in\rg(\phi)$ with $C_s\equiv_{\mathsf{L}(\bar{d})} \neg A$ if this exists, and any $s\in\rg(\phi)$ with $s<r$ and $C_s\equiv_{\mathsf{L}(\bar{d})} \neg A$ otherwise. Then $\bar{A}=r^\smallfrown A\cup s^\smallfrown(\neg A)$ is in the normal form (3b) and $\bar{A}\equiv_{\mathsf{L}(d_{\phi})}\bar{C}$. 
%
%Finally, suppose that $\mathcal{C}$ has a $\mathsf{L}(\bar{d})$-maximum $B$ and that $B$ is \( \L(\bar{d}) \)-selfdual. If there is a strictly decreasing sequence $(r_n)_{n\in\omega}$ with $C_r\equiv_{\L(\bar{d})} B$, then it is easy to see that $\bar{A}:=\bigcup_{n\in\omega} C_{r_n}$ is $\mathsf{L}(d_{\phi})$-equivalent to $\bar{C}$ and in the normal form (4b). Otherwise, using the assumption $\inf \rg(\phi)>0$, there is some $r\in\rg(\phi)$ such that $C_s\leq_{\mathsf{L}(\bar{d})} 0^{(n)}{}^{\smallfrown}C_r$ for all $s\in \rg(\phi)$ with $r\cdot 2^{-n}\leq s<r$. 
%We choose the maximal such $r\in\rg(\phi)$ and define $\bar{A}:=r^{\smallfrown} C_r$. Then $\bar{A}$ is of the form (4a) and it is easty to check that it follows from the choice of $r$ that $\bar{A}\equiv_{\mathsf{L}(d_{\phi})} \bar{C}$. 
\end{proof}

By Claim~\ref{claim:normalforms}, to show that \( \SLO^{\L(d_\phi)} \) holds for Borel subsets of \( Y \) it is enough to show that for every pair of Borel sets $\bar{A}$ and $\bar{B}$ in $\L(d_\phi)$-normal form, either \( \bar{A} \leq_{\L(d_\phi)} \bar{B} \) or \( \bar{B} \leq_{\L(d_\phi)} \neg \bar{A} \): we are now going to sketch the proof of this fact, by considering all the possible combinations of normal forms. 
%Given an arbitrary set \( \bar{C} \subseteq Y \) in normal form, set $r(\bar{C})=\inf \{ t\in \rg(\phi) \mid \exists x \in \pre{\omega}{\omega}\, (r {}^\smallfrown{} x \in C) \}$.

If $\bar{A}$ is in case (1) of the normal form, then it is \( \L(d_\phi) \)-selfdual, and hence semi-linearity is equivalent to showing that 
\( \bar{A} \leq_{\L(d_\phi)} \bar{B} \) or \( \bar{B} \leq_{\L(d_\phi)} \bar{A} \). Let \( A' = \oplus_{n \in \omega} A_n \), so that 
\( [ A']_{\L(\bar{d})}  = \sup_{n \in \omega} [A_n]_{\L(\bar{d})} \). First assume that \( \bar{B} \) is either in normal form (1) or (2), and let 
\( B' = \bigoplus_{n \in \omega} B_n \). 
%Notice that \( A' <_{\L(\bar{d})} B' \iff A' <_{\Lip(\bar{d})} B' \) and, conversely, 
%\( B' <_{\L(\bar{d})} A' \iff B' <_{\Lip(\bar{d})} A' \).
If \( A' <_{\L(\bar{d})} B' \) (equivalently, \( A' <_{\Lip(\bar{d})} B' \)), then 
%(using the assumption $\inf \rg(\phi)>0$ if \( r(B) > r(A) \)) 
we get \( \bar{A} \leq_{\L(d_\phi)} \bar{B} \), and similarly switching the role of \( A \) and \( B \). 
If instead \( A' \equiv_{\L(\bar{d})} B' \), then we get \( \bar{A} \equiv_{\L(d_\phi)} \bar{B} \). 
%If instead \( A' \equiv_{\L(\bar{d})} B' \), then we get \( \bar{A} \leq_{\L(d_\phi)} \bar{B} \) if \( r(B) \leq r(A) \), and \( \bar{B} \leq_{\L(d_\phi)} \bar{A} \) otherwise. 
 Assume now that \( \bar{B} \) is either in normal form (3) or (4). Then using \( B \) in place of \( B' \) in the argument above (and noticing 
that either \( A' \leq_{\Lip(\bar{d})} B \) or else \( B \leq_{\L(\bar{d})} A_n \) for all sufficiently large \( n \in \omega \)) we get again that \( \bar{A} \) 
is \( \L(d_\phi) \)-comparable with \( \bar{B} \), as required.

Let now \( \bar{A} \) be in normal form (2). If $\bar{B}$ is in normal form (2) too, arguing as in the previous case we compare $A' = \bigoplus_{n \in \omega} A_n$ and  $B' = \bigoplus_{n \in \omega} B_n$ with respect to \( \L(\bar{d}) \). 
%and if they are equal, $r(\bar{A})$ with $r(\bar{B})$. 
Similarly, if $\bar{B}$ is in case (3), we compare $A'$ with $B$ with respect to \( \L(\bar{d}) \), and then argue as above again. 
Now let us suppose that $\bar{B}$ is in case (4). If $A_n <_{\mathsf{Lip}(\bar{d})} B$ for all $n\in\omega$, then $\bar{A}\leq_{\mathsf{L}(d_{\phi})} \bar{B}$. 
Otherwise $B\leq_{\mathsf{L}(\bar{d})} A_n$ for some $n\in\omega$ and thus $\bar{B}\leq_{\mathsf{L}(\bar{d})} \bar{A}$. 

%The case when $\bar{B}$ is in case (4) is slightly more involved. If $A_n <_{\mathsf{Lip}(\bar{d})} B$ for all $n\in\omega$, then $\bar{A}\leq_{\mathsf{L}(d_{\phi})} \bar{B}$. 
%Otherwise there are \( n,m \in \omega \) such that \( B \equiv_{\L(\bar{d})} 0^{(n)}{}^\smallfrown A_m \), and in this case  we can prove that \( \bar{A} \leq_{\L(d_\phi)} \bar{B} \) or \( \bar{B} \leq_{\L(d_\phi)} \neg \bar{A} \) by comparing $r(\bar{A})$ and the distances between the various copies of $B$ in $\bar{B}$. 

We now assume that \( \bar{A} \) is in normal form (3). If $\bar{B}$ is in normal form (3) too, we can prove \( \bar{A} \leq_{\L(d_\phi)} \bar{B} \) or \( \bar{B} \leq_{\L(d_\phi)} \neg \bar{A} \)  by first comparing \( A \) and \( B \) with respect to \( \Lip(\bar{d}) \)-reducibility (equivalently, \( \L(\bar{d}) \)-reducibility), using the assumption $\inf \rg(\phi)>0$. 
If $\bar{A}\equiv_{\mathsf{Lip}(\bar{d})} \bar{B}$ and both $\bar{A}$ and $\bar{B}$ are in case (3b), then we simply compare the minimum of the values $r_1$ appearing in their normal forms. The comparison is straightforward in all other cases for $\bar{A}$ and $\bar{B}$ in the normal form (3) with $\bar{A}\equiv_{\mathsf{Lip}(\bar{d})} \bar{B}$. 
If instead $\bar{B}$ is in case (4), using the assumption $\inf \rg(\phi)>0$, 
we simply need to compare the $\mathsf{L}(\bar{d})$-degrees of $A$ and $B$; all possible relationships between these degrees with respect to \( \leq_{\L(\bar{d})} \) can be transferred back to analogous relationships between \( \bar{A} \) and \( \bar{B} \) with respect to \( \leq_{\L(d_\phi)} \). 

Let us finally suppose that $\bar{A}$ and $\bar{B}$ are both in case (4); this is the more delicate case. Since $\bar{A}$ and $\bar{B}$ are clearly $\mathsf{L}(d_\phi)$-selfdual, it is again sufficient to show that they are $\mathsf{L}(d_{\phi})$-comparable. 
%We may without loss of generality assume that $r(\bar{B}) \leq r(\bar{A}) $ (the case \( r(\bar{A}) < r(\bar{B}) \) is symmetric and can be treated similarly, switching the role of \( A \) and \( B \) in the argument below).
%First assume that $\bar{A}$ and $\bar{B}$ are  both in case (4a), where $\bar{A}=r^\smallfrown A$ and $\bar{B}=s^{\smallfrown} B$ with $r<s$. 
First assume that $\bar{A}$ and $\bar{B}$ are  both in case (4a), where $\bar{A}=r^\smallfrown A$ and $\bar{B}=s^{\smallfrown} B$. 
% and that $r(\bar{B}) \leq r(\bar{A}) $. 
If $A$ and $B$ are not in the same ${\mathsf{Lip}(\bar{d})}$-degree, then it is easy to compare \( \bar{A} \) and \( \bar{B} \) with respect to $\leq_{\mathsf{L}(d_{\phi})}$, and if $A\leq_{\mathsf{L}(\bar{d})} B$ then $\bar{A} \leq_{\mathsf{L}(d_{\phi})} \bar{B}$.
Hence we can assume that \( A \equiv_{\Lip(\bar{d})} B \) and $B\leq_{\mathsf{L}(\bar{d})} A$, so that $A\equiv_{\mathsf{L}(\bar{d})} 0^{(n)}{}^\smallfrown B$ for some $n\in\omega$. 
Using the assumption $\inf \rg(\phi)>0$, it is now easy to check that $\bar{A}\leq_{\mathsf{L}(d_{\phi})}\bar{B}$ holds if $s\cdot 2^n \leq r$, while $\bar{B}\leq_{\mathsf{L}(d_{\phi})}\bar{A}$ holds if $r\leq s\cdot 2^n$.  
 
%Then if $r(\bar{B}) \leq 2^{-n} \cdot r(\bar{A})$ we get $\bar{A}\leq_{\mathsf{L}(d_{\phi})}\bar{B}$, while if $r(\bar{B}) \geq 2^{-n} \cdot r(\bar{A})$ then $\bar{B}\leq_{\mathsf{L}(d_{\phi})}\bar{A}$. The case \( r(\bar{B}) > r(\bar{A}) \) can be treated similarly, switching the role of \( A \) and \( B \).

Suppose now that $\bar{A}=r^\smallfrown A$ is in case (4a) and $\bar{B}$ is in case (4b). 
%, and assume first that \( r(\bar{B}) \leq r(\bar{A}) \).
%$\bar{A}=r^\smallfrown A $ for some $r\in \rg(\phi)$ and $\bar{B}=\bigcup_{n\in\omega} s_n^\smallfrown B$ with $(s_n)_{n \in \omega}$ strictly decreasing and $\inf_{n\in\omega} s_n=s$. 
We have $\bar{A}\leq_{\mathsf{L}(d_{\phi})} \bar{B}$ if $A \leq_{\mathsf{L}(\bar{d})} B$ holds, 
and moreover  
$B<_{\mathsf{Lip}(\bar{d})} A$ implies that $0^{(n)}{}^\smallfrown B\) \(\leq_{\mathsf{L}(\bar{d})} A$ for all $n\in\omega$ (which in turn implies $\bar{B}\leq_{\mathsf{L}({d_{\phi}})}\bar{A}$). 
Thus we can assume that $A\equiv_{\mathsf{Lip}(\bar{d})} B$ and $B\leq_{\mathsf{L}(\bar{d})} A$, so that again $A\equiv_{\mathsf{L}(\bar{d})} 0^{(n)}{}^\smallfrown B$ for some $n\in\omega$, and let $s=\inf \rg(\phi)$. 
Arguing similarly to the previous case, it is easy to check that $\bar{A}\leq_{\mathsf{L}(d_{\phi})}\bar{B}$ holds if $s\cdot 2^n < r$, while $\bar{B}\leq_{\mathsf{L}(d_{\phi})}\bar{A}$ holds if $s \cdot 2^n \geq r$. 
%The case \( r(\bar{A}) < r(\bar{B}) \) can be treated similarly, switching the role of \( r(\bar{A}) \) and \( r(\bar{B}) \) to distinguish the various cases. 

The last case that needs to be considered is when both $\bar{A}$ and $\bar{B}$ are in case (4b). 
%We may further assume \( r(\bar{B}) \leq r(\bar{A}) \), for otherwise it is enough to switch the role of \( A \) and \( B \). 
%$\bar{A}=\bigcup_{n\in\omega} r_n^\smallfrown A$ and $\bar{B}=\bigcup_{n\in\omega} s_n^\smallfrown B$ with $(r_n)_{n \in \omega}$ and $(s_n)_{n \in \omega}$ strictly decreasing and $\inf_{n\in\omega} r_n=r\leq \inf_{n\in\omega} s_n=s$. 
We may assume that $A\leq_{\mathsf{L}(\bar{d})} B$ and hence $\bar{A}\leq_{\mathsf{L}(d_{\phi})}\bar{B}$. 
%Now, $A\leq_{\mathsf{L}(\bar{d})} B$ again implies $\bar{A}\leq_{\mathsf{L}(d_{\phi})}\bar{B}$, and 
%$B<_{\mathsf{Lip}(\bar{d})} A$ implies that $0^{(n)}{}^\smallfrown B\leq_{\mathsf{L}(\bar{d})} A$ for all $n\in\omega$, 
%whence $\bar{B}\leq_{\mathsf{L}({d_{\phi}})}\bar{A}$. 
%In the remaining case,  $A\equiv_{\mathsf{Lip}(\bar{d})} B$ and $B\leq_{\mathsf{L}(\bar{d})} A$, so that $A\equiv_{\mathsf{L}(\bar{d})} 0^{(n)}{}^\smallfrown B$ for some $n\in\omega$ again. As before,  $r(\bar{B}) \leq 2^{-n} \cdot r(\bar{A})$ implies $\bar{A}\leq_{\mathsf{L}(d_{\phi})} \bar{B}$, while  $r(\bar{B})\geq 2^{-n} \cdot r(\bar{A})$ implies $\bar{B}\leq_{\mathsf{L}(d_{\phi})} \bar{A}$. 
This concludes the proof that \( \SLO^{\L(d_\phi)} \) holds for Borel subsets (in normal form) of \( Y \).

\medskip

It remains to show that the \( \L(d_\phi) \)-hierarchy on Borel subsets of \( Y \) is well-founded, and for this we may again concentrate only on sets in normal form. Assume towards a contradiction that there is a family \( (\bar{A}^{(i)})_{i \in \omega} \) of Borel subsets of \( Y \) in normal form such that \( \bar{A}^{(i+1)} <_{\L(d_\phi)} \bar{A}^{(i)} \) for all \( i \in \omega \). Since there are only finitely many types of  normal form, passing to a subsequence if necessary we may further assume that all the \( \bar{A}^{(i)} \)'s share the same type of normal form. We now consider the various possibilities. %Let \( r(\bar{A}^{(i)}) \) be defined as before.

First assume that the \( \bar{A}^{(i)} \)'s are all in normal form (1), and set \( (A')^{(i)} = \bigoplus_{n \in \omega} A^{(i)}_n \), where the sets 
\( A^{(i)}_n \subseteq \pre{\omega}{\omega} \) are those appearing in the normal form of \( \bar{A}^{(i)} \). Notice that all the \( (A')^{(i)} \) are necessarily \( \L(\bar{d}) \)-selfdual. Then \( (A')^{(i+1)} <_{\L(\bar{d})} (A')^{(i)} \), because otherwise \( (A')^{(i)} \leq_{\L(\bar{d})} (A')^{(i+1)} \), whence one would easily get \( \bar{A}^{(i)} \leq_{\L(d_\phi)} \bar{A}^{(i+1)} \), contradicting the choice of the \( \bar{A}^{(i)} \)'s. Therefore the \( (A')^{(i)} \) are strictly \( \L(\bar{d}) \)-decreasing, 
 contradicting the fact that the \( \L(\bar{d}) \)-hierarchy on Borel subsets of \( \pre{\omega}{\omega} \) is well-founded.

The case where all the \( \bar{A}^{(i)} \)'s are in normal form (2) can be dealt with in the same way, and a similar argument works also for the other cases with the following minor modifications:
\begin{enumerate}[\( \bullet \)]
\item
When considering normal forms as in (3a), set \( (A')^{(i)} = A^{(i)} \), where \( A^{(i)} \subseteq \pre{\omega}{\omega} \) is the set appearing in the normal form of \( \bar{A}^{(i)} \), and pass to a subsequence if necessary to avoid the situations in which \( A^{(i+1)} \equiv_{\L(\bar{d})} \neg A^{(i)} \);
\item
When considering normal forms as in (3b) or (3c), set \( (A')^{(i)} = (0 {}^\smallfrown{}  A^{(i)}) \cup (1 {}^\smallfrown{} (\neg A^{(i))}) \), where \( A^{(i)}, \neg A^{(i)} \subseteq \pre{\omega}{\omega} \) are the sets appearing in the normal form of \( \bar{A}^{(i)} \). In case (3b), we may need to pass to a subsequence \( ((A')^{(i_l)})_{l \in \omega} \) to guarantee that \( A^{(i_{l+1})} <_{\L(\bar{d})} A^{(i_l)} \). 
\item
When considering normal forms as in (4), set \( (A')^{(i)} = A^{(i)} \), where \( A^{(i)} \subseteq \pre{\omega}{\omega} \) is the set appearing in the normal form of \( \bar{A}^{(i)} \). In case (4a) it may be necessary to first pass to a subsequence \( ((A')^{(i_l)})_{l \in \omega} \) to guarantee that the sequence of the \( r^{(i_l)} \)'s appearing in the canonical form of \( \bar{A}^{(i_l)} \) is not \( < \)-increasing.
\end{enumerate}
This concludes the proof of the well-foundness of \( \leq_{\L(d_\phi)} \) on Borel subsets of \( Y \), and hence of the entire proposition.
\end{proof}

Corollaries~\ref{cor:specialincreasing},~\ref{cor:specialconvergingto0}, and Theorem~\ref{prop:specialalpha*} already cover many interesting cases, and using the methods developed in the proof of Theorem~\ref{prop:specialalpha*} it seems plausible to conjecture that if the range of \( \phi \) does not contain increasing \(\omega\)-sequences, then the \( \L(d_\phi) \)-hierarchy on Borel subsets of \( \pre{\omega}{\omega} \) is well-founded. 
%very good (so that for every \( \phi \colon \omega \to \RR^+ \),  \( \Deg_{\boldsymbol{\Delta}^1_1}(\L(d_\phi)) \) is either very good or very bad). 
However, the general problem of determining the character of the \( \L(d) \)-hierarchy on an arbitrary ultrametric Polish space \( X = (X,d) \) remains open:
 
\begin{question} \label{quesn:notincreasing}
Let $X = (X,d)$ be an ultrametric (perfect) Polish space such that \( R(d) \) does not contain an honest increasing sequence, and assume that \( R(d) \) is neither finite nor a (\( \omega \)-)sequence converging to \( 0 \). Is the $\mathsf{L}(d)$-hierarchy \( \Deg_{\boldsymbol{\Delta}^1_1}(\L(d)) \) on the Borel subsets of $X$ (very) good? 
\end{question} 

\begin{remark}
In order to answer Question~\ref{quesn:notincreasing}, it may be useful to note the following.
It is proved in \cite[Theorem 4.1]{Gao:2011} that every ultrametric Polish space $X = (X,d)$ is isometric to a closed subspace of the ultrametric Urysohn space $U_{R(d)}=\{(x_n)_{n \in \omega} \in {}^{\omega}(R(d) \cup \{ 0 \})\mid x_n \geq x_{n+1}$ for all $n$ and $\lim_{n\to \infty} x_n=0\}$ equipped with the complete ultrametric 
\[ 
d_{U_{R(d)}} ((x_n)_{n \in \omega},(y_n)_{n \in \omega}) = 
\begin{cases}
0 & \text{if } x_n = y_n \text{ for all }n, \\
\max({x_n, y_n}) & \text{if \( n \) is least such that } x(n) \neq y(n). 
\end{cases}
 \] 
Suppose that $X = (X,d)$ is a perfect ultrametric Polish space and choose a closed subspace $Y$ of $(U_{R(d)},d_{U_{R(d)}})$ such that $Y =(Y, d_{U_{R(d)}})$ is isometric to $X$. 
Let $S(Y)=\{y \restriction n\mid y\in Y,\ n\in\omega\}$, and set  $D(s)=\{r\in\mathbb{R}^+ \mid \exists x\in {}^\omega(R(d) \cup \{ 0 \}) \, (s {}^\smallfrown{}  r {}^\smallfrown{}  x\in Y) \}$ for each $s\in S(Y)$. Notice that $S(Y)$ and $D(s)$ are countable since $R(d)$ is countable. 
If there is a strictly increasing sequence $(r_n)_{n\in\omega}$ in $D(s)$ for some $s\in S(Y)$, then we obtain an honest increasing sequence in $R(d)$ from the assumption that $(X,d)$ is perfect. If there is no honest increasing sequence in $R(d)$, it follows that the order type of each $D(s)$ is $\alpha_s^*$ for some countable ordinal $\alpha_s$. 
\end{remark}

%We saw in Theorem~\ref{th:notunbounded} that if the diameter of an ultrametric Polish space \( X = (X,d) \) is not (nontrivially) unbounded, then the \( \Lip(d) \)-hierarchy on the Borel subsets of \( X \) is very good, while 

Finally, we want to show that, even if 
by Theorem~\ref{th:increasingdistances} it is possible that  the \( \L(d) \)-hierarchy \( \Deg_{\boldsymbol{\Delta}^1_1}(\L(d)) \) on Borel subsets of a given ultrametric Polish space \( X = (X,d) \) with bounded diameter is very bad, a natural (modest) strengthening of the preorder \( \leq_{\L(d)} \) already yields to a semi-linearly ordered hierarchy.

\begin{definition} \label{def:almost}
Suppose $X=(X,d)$ is an ultrametric Polish space, and let $A,B\subseteq X$. Let us write $A\leq_{\Lip(d,L)} B$ if there is a Lipschitz function $f\colon (X,d) \rightarrow (X,d)$ with constant $L \in \RR^+$ such that $A=f^{-1}(B)$. 
We say that $A$ is \emph{almost nonexpansive reducible} to $B$ ($A\leq_{\mathsf{a}\mathsf{L}(d)}B$ in symbols) if  \( A \leq_{\Lip(d,L)} B \) for every $1 < L \in \RR^+ $. 
\end{definition} 

Notice that the relation $\leq_{\mathsf{a}\mathsf{L}(d)}$ is a preorder (for the transitivity use the fact that if \( f,g \colon X \to X \) are Lipschitz functions with constant \( L, L' \), respectively, then \( g \circ f \) is Lipschitz with constant \( L \cdot L' \)). Moreover, \( \leq_{\mathsf{a} \L (d)} \) is  strictly between \( \leq_{\L(d)} \) and \( \leq_{\Lip(d)} \). Even if literally \( \leq_{\mathsf{a}\L(d)} \) is not of the form \( \leq_\F \) for some reducibility \( \F \) on \( X \), with a little abuse of notation and terminology we can nevertheless  consider the \( \mathsf{a}\L(d) \)-hierarchy on (Borel subsets of) \( X \), the Semi-Linear Ordering principle \( \SLO^{\mathsf{a}\L(d)} \), and so on (with the obvious definitions).

\begin{proposition} 
Let $X=(X,d)$ be an ultrametric Polish space with bounded diameter. Then the $\mathsf{a}\mathsf{L}(d)$-hierarchy on the Borel subsets of $X$ is semi-linearly ordered, and hence not bad.
\end{proposition} 

\begin{proof} 
Given \( L > 1 \), let $d_L \colon X\times X \to \RR^+$ be defined by $d_L(x,y)=\min (\{L^{n}\mid d(x,y)\leq L^n$ and $n\in \ZZ\})$ if \(x,y \in X \) are distinct, and by \( d_L(x,y) = 0 \) if \( x = y \in X \). Then $d_L$ is a complete ultrametric on $X$ compatible with the metric topology \( \tau_d \), and since we assumed that \( X \) has bounded (\( d \)-)diameter we also have that \( R(d_L) \subseteq \{ L^n \mid n \in \ZZ \} \) is either finite, or a decreasing sequence converging to \( 0 \). By  Theorem~\ref{th:descendingdistances}, this means that the \( \L(d_L) \)-hierarchy on Borel subsets of \( X \) is very good, and hence, in particular, semi-linearly ordered.
%with $R(\bar{d})\subseteq \{L^{-n}\mid n\in \omega\}$ 
Moreover, $\id\colon (X,d)\rightarrow (X,d_L)$ is Lipschitz with constant \( L \), while $\id \colon (X,d_L)\rightarrow (X,d)$ 
is nonexpansive.  
Hence for all subsets $A,B$ of $X$: 

\begin{itemize} 
\item if $A\leq_{\mathsf{Lip}(d,L')} B$, then $A\leq_{\mathsf{Lip}(d_L, L \cdot L')} B$;
\item if $A\leq_{\mathsf{Lip}(d_L,L')} B$, then $A\leq_{\mathsf{Lip}(d, L\cdot L')} B$. 
\end{itemize} 
In particular, $A\leq_{\mathsf{L}(d_L)} B$ implies that $A\leq_{\mathsf{Lip}(d,L)} B$. 

We claim $\SLO^{\mathsf{a}\mathsf{L}(d)}$ holds for Borel subsets of $X$. By the observation above and \( \SLO^{\Lip(d_L)} \), for every fixed \( L > 1 \) we have that either $A\leq_{\mathsf{Lip}(d,L)} B$ or $B\leq_{\mathsf{Lip}(d,L)} \neg A$. If for every \( n \in \omega \) there is \( 1< L \leq 1+2^{-n} \) such that $A\leq_{\mathsf{Lip}(d,L)} B$, then \( A \leq_{\mathsf{a} \L(d)} B \). Similarly, if for every \( n \in \omega \) there is \( 1< L \leq 1+2^{-n} \) such that $B\leq_{\mathsf{Lip}(d,L)} \neg A$, then \( B \leq_{\mathsf{a} \L(d)} \neg A \). Since one of the two possibilities necessarily occurs, we get that either \( A \leq_{\mathsf{a}\L(d)} B \) or \( B \leq_{\mathsf{a} \L(d)} \neg A \), as required. 
%Now suppose that $(A_n)_{n\in\omega}$ is strictly decreasing in $<_{a\mathsf{L}(d)}$. Then $A_{n_i}\leq_{\mathsf{Lip}(d_l)} A_{n_{j}}$ for constants $r>l$ arbitrarily close to $l$ for all $l>1$ and all $i<j<\omega$. Suppose first that there is an infinite subsequence $(A_{n_i})_{i}$ which is strictly increasing in $<_{\mathsf{L}(d_l)}$ for some $l>1$. 
%Then $A_{n_2}$ is not reducible to $A_{n_0}$ by a Lipschitz function with constant $l^2>l$, contradicting the properties of $(A_{n_i})$. 
%
%Hence for any $l>1$ there is a strictly increasing sequence $(n_i)_i$ such that $A_{n_i}\equiv_{\mathsf{L}(d_l)}A_{n_j}$ for all $i,j\in\omega$. 
\end{proof}

\section{Compact ultrametric Polish spaces}

It is well-known that any continuous function between metric spaces is automatically uniformly continuous as soon as its domain is compact (see e.g.\ \cite[Proposition 4.5]{Kechris:1995}). In particular, this means that it does not make much sense to consider the \( \UCont(d) \)-hierarchy on a compact ultrametric Polish space \( X = (X,d) \): since it coincide%
\footnote{In fact in the specific case of the Cantor space \( \mathcal{C} = (\pre{\omega}{2}, \bar{d}_{\mathcal{C}}) \) one can check that, although \( \Lip(\bar{d}_{\mathcal{C}}) \subsetneq \UCont(\bar{d}_{\mathcal{C}})\), the \( \Lip(\bar{d}_{\mathcal{C}}) \)- and the \( \UCont(\bar{d}_{\mathcal{C}}) \)-hierarchies coincide.}
 with the \( \W(X) \)-hierarchy on \( X \), its restriction to the Borel sets is always very good by Proposition~\ref{prop:continuous}. However, one may wonder about the character of the \( \Lip(d) \)- and the \( \L(d) \)-hierarchy on (Borel subsets of) such an \( X \): the next results show that they  must always be very good as well.

\begin{proposition} \label{prop:R(d)}
Let \( X = (X,d) \) be a compact ultrametric Polish space. Then either \( X \) (and hence also \( R(d) \)) is finite, or else \( R(d) \) is a strictly decreasing (\( \omega \)-)sequence converging to \( 0 \). In particular, \( X \) has bounded diameter.
\end{proposition}

\begin{proof}
It is clearly enough to show that for every \(  \bar{r} \in \RR^+  \), the set \( R(d)_{\geq \bar{r}} = \{ r \in R(d) \mid r \geq \bar{r} \} \) is finite. To see this, observe that the family \( \mathcal{B} =  \{ B_d(x,\bar{r}) \mid x \in X \} \) is a finite covering of \( X \) because \( X \) is compact. Assume towards a contradiction that \( R(d)_{\geq \bar{r}} \) is infinite, let \( (r_n)_{n \in \omega} \) be an enumeration without repetitions of it, and let \( (x_n)_{n \in \omega} \) and \( (y_n)_{n \in \omega } \) be such that \( d(x_n,y_n) = r_n \) for every \( n \in \omega \). Since \( \mathcal{B} \) is finite, there are distinct \( n,m \in \omega \) such that \( d(x_n,x_m), d(y_n,y_m) < \bar{r} \). Since \( r_m \geq \bar{r} \), we get that \( r_n = d(x_n,y_n) = d(x_m,y_m) = r_m \), contradicting the choice of the \( r_n \)'s.
\end{proof}

\begin{theorem} \label{th:compact}
Let \( X = (X,d) \) be a compact ultrametric Polish space. Then both the \( \L(d) \)- and the \( \Lip(d) \)-hierarchy on Borel subsets of \( X \) are very good. 
\end{theorem}

\begin{proof}
Use Proposition~\ref{prop:R(d)} together with Theorems~\ref{th:descendingdistances} and~\ref{th:unifcontandlip}. 
\end{proof}

In particular, we cannot change the ultrametric on the Cantor space \( \pre{\omega}{2} \) to make its nonexpansive or its Lipschitz hierarchy (very) bad: if \( d' \) is any complete ultrametric compatible with the product topology on \( \pre{\omega}{2} \), then both the \( \Lip(d') \)- and the \( \L(d') \)-hierarchy on Borel subsets of \( \pre{\omega}{2} \) are very good.%
\footnote{However, analogously to~\cite[Section 5]{MottoRos:2012c} it is still possible to define compatible complete ultrametrics \( d_i \), \( i = 0,1 \), on \( \mathcal{C} =  \pre{\omega}{2} \) so that \( \L(\bar{d}_{\mathcal{C}}) \not\subseteq \Lip(d_0) \) (hence also  \( \L(\bar{d}_{\mathcal{C}}) \not\subseteq \L(d_0) \)), while \( \L(\bar{d}_{\mathcal{C}}) \not\subseteq \L(d_1) \) but \( \Lip(\bar{d}_{\mathcal{C}}) = \Lip(d_1) \).}

\begin{remark}
Albeit Theorem~\ref{th:compact} shows that there is no compact ultrametric Polish space \( X = (X,d) \) with a (very) bad \( \Lip(d) \)- or \( \L(d) \)-hierarchy, Corollaries~\ref{cor:countable} and~\ref{cor:X_1} shows that there are \( \boldsymbol{K}_\sigma \)-spaces%
\footnote{A topological space is \( \boldsymbol{K}_\sigma \) if it is the union of countably many compact subsets.} 
\( X_i = (X_i,d_i) \), \( i = 0,1 \), such that:
\begin{enumerate}[\( \bullet \)]
\item
both \( \Deg_{\boldsymbol{\Delta}^0_1}(\L(d_0)) \) and \( \Deg_{\boldsymbol{\Delta}^0_1}(\Lip(d_0)) \) are very bad;
\item
\( \Deg_{\boldsymbol{\Delta}^0_1}(\L(d_1)) \) is very bad, while \( \Deg_{\boldsymbol{\Delta}^1_1}(\Lip(d_1)) \) is very good.
\end{enumerate}
\end{remark}

Let us now concentrate on the Cantor space \( \mathcal{C} =  \pre{\omega}{2} \), and let us briefly consider another kind of reducibility that was analyzed in~\cite{MottoRos:2012c} for the case of the Baire space, namely the collection of all contraction mappings.

\begin{notation}
Let \( \bar{d} = \bar{d}_{\mathcal{C}} \) be the usual metric on the Cantor space. We denote by \( \c(\bar{d}) \) the collection of all \emph{contractions} from \( \mathcal{C} \) into itself, i.e.\ of all Lipschitz functions \( f \colon \mathcal{C} \to \mathcal{C} \) with constant strictly smaller than \( 1 \).

Given two sets \( A,B \subseteq \mathcal{C} \), set 
\[ 
A \leq_{\c(\bar{d})} B \iff A = B \vee \exists f \in \c(\bar{d}) \, (A = f^{-1}(B)). 
\] 
In fact, \( {\leq_{\c(\bar{d})}} = {\leq_\F} \), where \( \F \) is the reducibility on \( \C \) obtained by adding the identity  \( \id = \id_{\mathcal{C}} \) to the set \( \c(\bar{d}) \).
\end{notation}

Using the methods developed in~\cite[Section 4]{MottoRos:2012c}, it is easy to check that the following hold:

\begin{theorem} \label{th:contractions}
Let \( A,B \) be Borel subsets of \( \mathcal{C} \). 
\begin{enumerate}
\item
If \( A \not\equiv_{\L(\bar{d})} B \), then \( A \leq_{\c(\bar{d})} B \iff A \leq_{\L(\bar{d})} B \), while if \( A \equiv_{\L(\bar{d})} B \), then \( A \leq_{\c(\bar{d})} B \iff A \nleq_{\L(\bar{d})} \neg A \).
\item
\( A \) is \emph{selfcontractible} (i.e.\ \( A = f^{-1}(A) \) for some \( f \in \c(\bar{d}) \)) if and only if \( A \nleq_{\L(\bar{d})}  \neg A \).
\item
If \( A \nleq_{\L(\bar{d})}  \neg A \), then \( [A]_{\c(\bar{d})} = [A]_{\L(\bar{d})} \), while if 
\( A \leq_{\L(\bar{d})}  \neg A \), then \( [A]_{\c(\bar{d})} = \{ A \} \).
\item
\( A <_{\c(\bar{d})} B \iff A <_{\L(\bar{d})} B \).
\end{enumerate}
\end{theorem}

Therefore, to describe the \( \c(\bar{d}) \)-hierarchy on Borel subsets of \( \mathcal{C} \) it is enough to  determine how many sets are 
contained in each \( \L(\bar{d}) \)-degree of an \( \L(\bar{d}) \)-selfdual Borel subset of \( \mathcal{C} \), and to combine this information 
with the well-known description of the \( \L(\bar{d}) \)-hierarchy on Borel subsets of \( \mathcal{C} \) (see~\cite{Andretta:2007}). Let us first briefly describe this last 
hierarchy. First of all, the hierarchy is semi-well-ordered. At the bottom we found the \( \L(\bar{d}) \)-nonselfdual pair constituted by 
\( [\mathcal{C}]_{\L(\bar{d})} = \{ \mathcal{C} \} \) and \( [\emptyset]_{\L(\bar{d})} = \{ \emptyset \} \). Immediately after each
 \( \L(\bar{d}) \)-nonselfdual pair \( \{ [A]_{\L(\bar{d})} , [ \neg A ]_{\L(\bar{d})} \} \) there is the \( \L(\bar{d}) \)-degree 
of the \( \L(\bar{d}) \)-selfdual set 
\( A \oplus \neg A = (0 {}^\smallfrown{} A) \cup (1 {}^\smallfrown{} (\neg A)) =  \{ 0 {}^\smallfrown{} x \mid x \in A \} \cup \{ 1 {}^\smallfrown{} x \mid x \in \neg A \} \).
%,  where \( {}^\smallfrown{}  \) denotes the usual concatenation operation. 
On the other hand, if \( A \) is \( \L(\bar{d}) \)-selfdual, then immediately
 after \( [A]_{\L(\bar{d})} \) there is the \( \L(\bar{d})\)-degree of the selfdual set \( 0 {}^\smallfrown{} A = \{ 0 {}^\smallfrown{} x \mid x \in A \} \).
 Finally, at all limit level there is always an \( \L(\bar{d}) \)-nonselfdual pair. Therefore we get the structure represented in Figure~\ref{fig:Lhierarchy}, where bullets
 represent \( \L(\bar{d}) \)-degrees and each \( \L(\bar{d}) \)-degree is \( \L(\bar{d}) \)-reducible to another one if and only if it is (strictly) to the left 
of it.

\begin{figure}[!htbp]
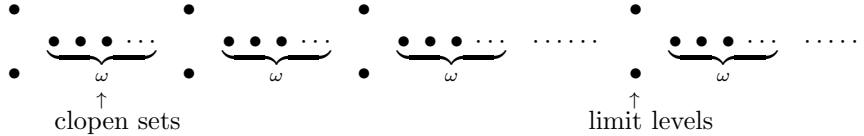

 \centering
%\begin{small}
\[
\begin{array}{llllllllll}
\bullet & & \bullet & & \bullet & & &  \bullet & &
\\
& \smash[b]{\underbrace{\bullet \; \bullet \;
\bullet \; \cdots}_{\omega}} & &
\smash[b]{\underbrace{\bullet \; \bullet \; \bullet \;
\cdots}_{\omega}} &  & \smash[b]{\underbrace{\bullet \; \bullet \;
\bullet \; \cdots}_{\omega}} & \cdots\cdots
 &   & \smash[b]{\underbrace{\bullet \; \bullet \;
\bullet \; \cdots}_{\omega}} &  \cdots\cdots
\\
\bullet & & \bullet & & \bullet & & &  \bullet & &
\\
& \quad \; \; \; \stackrel{\uparrow}{\makebox[0pt][l]{\text{clopen sets }}} & & & & 
 & &
 \stackrel{\uparrow}{\makebox[0pt][l]{\text{limit levels }}} &
\end{array}
\]
%\end{small}
 \caption{The \( \L(\bar{d}) \)-hierarchy on Borel subsets of \( \mathcal{C} \).
%, and a degree precedes another one in the ordering of \( \Deg(\L) \) if and only if it is more to the left in the picture. 
%The symbol \( \cof \) refers to the cofinality of the corresponding level (more precisely: of the rank of the \( \L \)-degrees in such level) in the well-ordering of \( \L \)-degrees.
}
 \label{fig:Lhierarchy}
\end{figure}

Notice that the first \(\omega\)-chain of consecutive \( \L(\bar{d}) \)-selfdual degrees contains all nontrivial clopen sets, while the first non-trivial 
\( \L(\bar{d}) \)-nonselfdual pair is formed by all proper open and proper closed subsets of \( \mathcal{C} \).

To compute the cardinality of a given \( [A]_{\L(\bar{d})} \) (for \( A \subseteq \mathcal{C} \)), recall first that if \(  \emptyset, \mathcal{C} \neq A \) 
is clopen, then there is \( 0 \neq n \in \omega \), called the \emph{level of \( A \)}  such that \( A \equiv_{\L(\bar{d})} \Nbhd_{0^{(n)}}\),
 where for an arbitrary \( s \in \pre{< \omega}{2} \) we set \( \Nbhd_s = \{ x \in \mathcal{C} \mid s \subset x \} \) --- in fact \( A \) is in the 
\( n \)-th \( \L(\bar{d}) \)-selfdual degree of the first \(\omega\)-chain of consecutive \( \L(\bar{d}) \)-selfdual degrees if and only if it is of level \( n \).

\begin{proposition} \label{prop:contractions}
Let \( \emptyset,\mathcal{C} \neq A \subseteq \mathcal{C} \).
\begin{enumerate}
\item
if \( A \) is clopen, then \( [A]_{\L(\bar{d})} \) contains exactly \( 2^{2^{n}} - 2^{2^{n-1}} \)-many sets, where \( n \) is the level of \( A \);
\item
if \( A \) is not clopen, then there is an injection \( j \colon \mathcal{C} \to [A]_{\L(\bar{d})} \).
\end{enumerate}
\end{proposition}

\begin{proof}
For each \( 0 \neq n \in \omega \), the collection of all clopen sets \( \L(\bar{d}) \)-reducible to \( \Nbhd_{0^{(n)}} \) consists of all the sets of the 
form \( \bigcup_{s \in S} \Nbhd_s \) for \( S \) a subset of \( \{ s \in \pre{<\omega}{2} \mid \leng(s) = n \} \): therefore there are \( 2^{2^n} \)-many
 such sets. So if \( A \) is a clopen set of level \( n \), then to compute the cardinality of \( [A]_{\L(\bar{d})} \) we have to subtract to \( 2^{2^n} \)  
the number of sets which are not \( \L(\bar{d}) \)-equivalent to \( \Nbhd_{0^{(n)}} \), i.e.\ \( \emptyset \), \( \mathcal{C} \), and all sets 
\( \L(\bar{d}) \)-reducible to \( \Nbhd_{0^{(n-1)}} \): since there are \( 2^{2^{n-1}} \)-many such sets, we get that \( [A]_{\L(\bar{d})} \) 
contains exactly \( 2^{2^{n}} - 2^{2^{n-1}} \)-many sets.

For the second part, let us first assume that \( A \nleq_{\L(\bar{d})} \neg A \). If \( A \) is a proper open set, then the map 
\( j \colon \mathcal{C} \to [A]_{\L(\bar{d})} \colon x \mapsto \mathcal{C} \setminus \{ x \} \) is as required. Therefore we can assume without 
loss of generality that \( B \leq_{\L(\bar{d})} A \) for every proper closed set \( B \). By Theorem~\ref{th:contractions}(2), there is 
\( f \in \c(\bar{d}) \) such that \( f^{-1}(A) = A \). Let \( i =0,1 \) be such that \( f(\mathcal{C}) \subseteq \Nbhd_{\langle i \rangle} \), and consider the map
 \[ 
j \colon \mathcal{C} \to [A]_{\L(\bar{d})} \colon x \mapsto A_x = (A \cap \Nbhd_{\langle i \rangle}) \cup \{ (1-i) {}^\smallfrown{} x \}.
\]
Clearly \( j \) is an injection, so it remains only to show that \( A \equiv_{\L(\bar{d})} A_x \) for every \( x \in \mathcal{C} \). For one direction, \( f \) 
witnesses \( A \leq_{\L(\bar{d})} A_x \). For the other direction, let \( g_x \in \L(\bar{d}) \) be a reduction of \( \{ (1-i) {}^\smallfrown{} x \} \) 
to \( A \): then \( (\id_{\mathcal{C}} \restriction \Nbhd_{\langle i \rangle}) \cup (g_x \restriction \Nbhd_{\langle (1-i) \rangle}) \) witnesses 
\( A_x \leq_{\L(\bar{d})} A \).

Finally, let \( A \) be \( \L(\bar{d}) \)-selfdual. Since by case assumption \( A \) is not clopen, there is an \( \L(\bar{d}) \)-nonselfdual 
\( B \neq \emptyset, \mathcal{C} \) and \( n \in \omega \) such that
 \( A \equiv_{\L(\bar{d})} 0^{(n)} {}^\smallfrown{}  (B \oplus \neg B) \). Let \( j' \colon \mathcal{C} \to [B]_{\L(\bar{d})} \) 
be an injective map: then
\[ 
j \colon \mathcal{C} \to [A]_{\L(\bar{d})} \colon x \mapsto 0^{(n)} {}^\smallfrown{} (j'(x) \oplus \neg j'(x)) 
 \] 
is clearly as required.
\end{proof}

Since by Theorem~\ref{th:contractions} the \( \c(\bar{d}) \)-hierarchy on Borel subsets of \( \mathcal{C} \) is the refinement of the 
\( \L(\bar{d}) \)-hierarchy obtained by splitting each \( \L(\bar{d}) \)-selfdual degree into the singletons of its elements, 
using Proposition~\ref{prop:contractions} we can represent such hierarchy as in Figure~\ref{fig:chierarchy}, where the 
bullets represent the \( \c(\bar{d}) \)-degrees and the boxes around them represent the \( \L(\bar{d}) \)-degrees they come from (notice that by Proposition~\ref{prop:contractions}(1) in the second column there are \( 12 \) different \( \c(\bar{d}) \)-degrees, while in the third column we already find \( 240 \) distinct \( \c(\bar{d}) \)-degrees!).

\begin{figure}[!htbp]
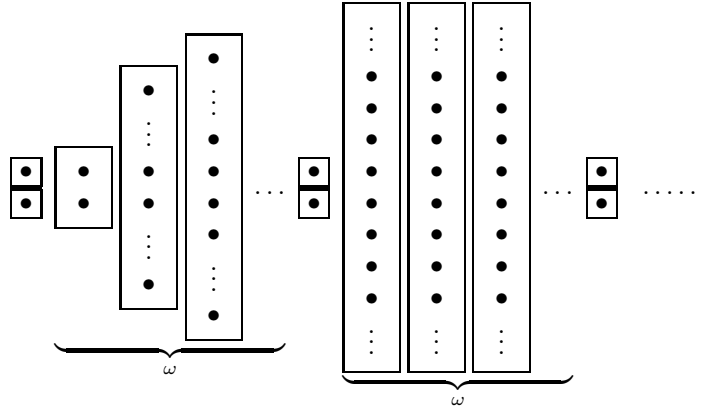

 \centering
\[
\begin{array}{c}
\phantom{\vdots}\\
\phantom{\bullet} \\
\framebox{$\bullet$} \\
\framebox{$\bullet$} \\
\phantom{\bullet} \\
\phantom{\vdots}
\end{array}%
\smash[b]{
\underbrace{
\framebox{$
\begin{array}{c}
\bullet \\
\bullet 
\end{array}
$} \;
\framebox{$
\begin{array}{c}
\bullet\\
\vdots \\
\bullet \\
\bullet\\
\vdots\\
\bullet
\end{array}
$} \;
\framebox{$
\begin{array}{c}
\bullet\\
\vdots \\
\bullet\\
\bullet \\
\bullet\\
\bullet\\
\vdots\\
\bullet
\end{array}
$} \;
\dotsc
}_{\omega}}
\begin{array}{c}
\phantom{\vdots}\\
\phantom{\bullet} \\
\framebox{$\bullet$} \\
\framebox{$\bullet$} \\
\phantom{\bullet} \\
\phantom{\vdots}
\end{array}
\smash[b]{
\underbrace{
\framebox{$
\begin{array}{c}
\vdots \\
\bullet\\
\bullet\\
\bullet \\
\bullet\\
\bullet \\
\bullet \\
\bullet \\
\bullet \\
\vdots
\end{array}
$} \;
\framebox{$
\begin{array}{c}
\vdots \\
\bullet\\
\bullet \\
\bullet\\
\bullet \\
\bullet\\
\bullet \\
\bullet\\
\bullet\\
\vdots
\end{array}
$} \;
\framebox{$
\begin{array}{c}
\vdots \\
\bullet\\
\bullet\\
\bullet \\
\bullet\\
\bullet \\
\bullet \\
\bullet\\
\bullet\\
\vdots
\end{array}
$} \;
\dotsc
}_{\omega}}
\begin{array}{c}
\phantom{\vdots}\\
\phantom{\bullet} \\
\framebox{$\bullet$} \\
\framebox{$\bullet$} \\
\phantom{\bullet} \\
\phantom{\vdots}
\end{array}
\;
\dotsc \dotsc 
\]
\vspace{1cm}
 \caption{The \( \c(\bar{d}) \)-hierarchy on Borel subsets of \( \mathcal{C} \).}
 \label{fig:chierarchy}
\end{figure}

Proposition~\ref{prop:contractions} and Theorem~\ref{th:contractions} also imply the following corollary.

\begin{corollary}\label{cor:contractions}
\begin{enumerate}[(1)]
\item
The \( \c(\bar{d}) \)-hierarchy \( \Deg_{\boldsymbol{\Delta}^1_1}(\c(\bar{d})) \) on Borel subsets of \( \mathcal{C} \) is bad but not very bad. In fact it contains antichains of size the continuum.
\item
The \( \c(\bar{d}) \)-hierarchy \( \Deg_{\boldsymbol{\Sigma}^0_1 \cup \boldsymbol{\Pi}^0_1}(\c(\bar{d})) \) on open or closed subsets of \( \mathcal{C} \) is good but not very good.
\end{enumerate}
\end{corollary}

Notice that Corollary~\ref{cor:contractions}(2) gives a partial answer to~\cite[Question 6.3]{MottoRos:2012c}. However, such solution is not completely satisfactory, as we needed to restrict our hierarchy to a very small class of subsets of \( \mathcal{C} \) --- of course it would be more interesting to find a reducibility \( \F \) (on some Polish space \( X \)) inducing a good but not very good hierarchy on the entire collection of Borel subsets of \( X \) (or, under \( \AD \), even on the entire \( \pow(X) \)). This last problem seems to be completely open, but the next example shows that if the requirement that the preorder inducing the hierarchy be of the form \( \leq_\F \) (for some reducibility \( \F \) on \( X \)) is dropped, then one can obtain a ``natural'' hierarchy on the collection of all Borel subsets of \( \pre{\omega}{\omega} \) which is good but not very good .

\begin{example}
Given a set $R\subseteq\mathbb{R}^+$ and $A,B \subseteq \pre{\omega}{\omega}$ such that \( A \leq_{\Lip(\bar{d})} B \), let 
\( L_{A,B} = \inf \{  0 < L \in \RR^+ \mid A \leq_{\Lip(\bar{d},L)} B \}\), where \( \leq_{\Lip(\bar{d},L)} \) is as in Definition~\ref{def:almost}. Then set 
\[
A \leq_R B \iff  A \leq_{\Lip(\bar{d})} B \wedge L_{A,B}\in R\cup\{0,1\}.
\] 
Notice that \( \leq_R \) is always reflexive: in fact, either \( A \nleq_{\Lip(\bar{d},L)} A \) for all \( L < 1 \) (in which case the identity function witnesses \( L_{A,A} = 1 \)), or else by considering arbitrarily large powers of any witness of \( A \leq_{\Lip(\bar{d}),L} A \) (for some \( L < 1 \)) we see that \( L_{A,A}  =  0\). In contrast, notice that in general $\leq_{R}$ need not to be transitive. However, when \( \leq_R \) actually happens to be a preorder (as in all the relevant cases considered below), then with a little abuse of terminology we can consider the \( \leq_R \)-hierarchy on Borel subsets of \( \pre{\omega}{\omega} \) (with the obvious definition).

Using the methods introduced at the end of~\cite[Section 4]{MottoRos:2012c}, it is easy to see that if \( A,B \subseteq \pre{\omega}{\omega} \) are Borel sets such that \( A <_{\Lip(\bar{d})} B \), then also \( A <_R B \), because in this case \( L_{A,B} = 0 \). Moreover, since by~\cite[Corollary 4.4]{MottoRos:2012c} if \( A \subseteq \pre{\omega}{\omega} \) is \( \Lip(\bar{d}) \)-nonselfdual (equivalently: \( \L(\bar{d}) \)-nonselfdual), then \( A \leq_{\Lip(\bar{d},L)} A \) for every \( L > 0 \), we get that for such an \( A \), \( A \leq_{\L(\bar{d})} B \Rightarrow A \leq_R B \) for every \( B \subseteq \pre{\omega}{\omega} \), and if \( R \subseteq (0,1] \) we in fact have that \( A \leq_{\L(\bar{d})} B \iff A \leq_R B \). 
Finally, if \( A \subseteq \pre{\omega}{\omega} \) is \( \L(\bar{d}) \)-selfdual and \( B \in [A]_{\L(\bar{d})} \), then \( A \equiv_R B \) because~\cite[Proposition 4.2]{MottoRos:2012c} implies that all the witnesses of \( A \leq_{\L(\bar{d})} B \) and \( B \leq_{\L(\bar{d})} A \) cannot have Lipschitz constant \( < 1 \). 
Summing up, we get  that if \( R \subseteq (0,1] \) (and \( \leq_R \) is transitive), then the $\leq_R$-hierarchy refines the $\mathsf{L}(\bar{d})$-hierarchy, and may differ from it  only within the $\mathsf{Lip}(\bar{d})$-selfdual degrees.

Let us now concentrate on the canonical examples given by $R_n=(0, 2^{-n}]$ (for \( n \in \omega \)). It is easy to check that if \( n \leq 1 \), 
then the \( \leq_R \)-hierarchy coincides with the \( \L(\bar{d}) \)-hierarchy. However, if \( n > 1 \) and  $A$ is an $\L(\bar{d})$-selfdual set,  then 
\begin{equation} \tag{$\dagger$} \label{eq:dagger}
A\leq_{R_n} B \iff A \leq_{\L(\bar{d})} B \wedge (B \equiv_{\L(\bar{d})} A  \vee 0^{(n)} {}^\smallfrown{}  A \leq_{\mathsf{L}(\bar{d})} B).
\end{equation}
Therefore the restriction of \( \leq_{R_n} \) to the Borel subsets of 
\( \pre{\omega}{\omega} \) is always transitive (hence a preorder), and it is also well-founded. Moreover, \eqref{eq:dagger} also implies that the 
antichains in $\leq_{R}$ have always size \( \leq n \). Since e.g.\  \( \{ 0^{(i+1)} {}^\smallfrown{} \pre{\omega}{\omega} \mid i < n \} \) is an \( \leq_{R_n} \)-antichain of size precisely \( n \) consisting of clopen sets, we get that for all $n\geq 3$ the $\leq_{R_n}$-hierarchy on Borel subsets of \( \pre{\omega}{\omega} \) is good
 but not very good. 
\end{example}

\section{Wadge-like reducibilities and the Axiom of Choice} \label{sec: choice}

By (the comment following) Proposition~\ref{prop:continuous}, the \( \L(\bar{d}) \)-hierarchy \( \Deg_{\boldsymbol{\Delta}^1_1}(\L(\bar{d})) \) on the Borel subsets of 
\( \pre{\omega}{\omega} \) is very good, and as already recalled the same is true for larger classes of subsets of 
\( \pre{\omega}{\omega} \) if we further assume corresponding determinacy axioms. It is therefore natural to ask what happens if, 
instead of assuming such determinacy principles, we assume the Axiom of Choice \( \AC \) or other strong choice principles. 

Similar considerations apply to arbitrary Polish spaces as well. It is shown in~\cite{Schlicht:2012} that for every non-zero-dimensional Polish space 
\( X \) the \( \W(X) \)-hierarchy  on Borel subsets of \( X \) already contains antichains of size the continuum, and in fact~\cite{Ikegami:2012} shows that 
if e.g.\ \( X = \RR \) then we can also embed \( (\pow(\omega), \subseteq^*) \) into \( \Deg_{\boldsymbol{\Delta}^1_1}(\W(X)) \) (but this last result 
cannot be extended to arbitrary \( X \): as explained in~ \cite[Section 5.1]{MottoRos:2012b}, all continuous functions on the Cook continuum \( X \) 
are either constant or the identity, and therefore all chains of subsets of \( X \) with respect to continuous reducibility have length \( \leq 2 \)).
 However,~\cite{MottoRos:2012b} shows that for every Polish space \( X \), the \( \D_\alpha(X) \)-hierarchy on Borel subsets of \( X \) (where 
\( \D_\alpha(X) \) denotes the collection of all \( \boldsymbol{\Delta}^0_\alpha \)-functions from \( X \) to itself) is always very good for 
\( \alpha \geq \omega \), and that the same is true for \( \alpha \geq 3 \) if \( X \) is of dimension \( \neq \infty \). Also these last results extend to larger
 classes of subsets of \( X \) under suitable determinacy assumptions, and therefore it is meaningful to ask what happens if instead we assume \( \AC \).

Not surprisingly, it turns out that under choice all the above mentioned hierarchies of degrees (on arbitrary subsets of \( X \)) become very bad. 
Clearly, Borel determinacy forces us to consider non-Borel subsets of \( X \) to get such results: therefore in what follows we will concentrate only on
 uncountable (ultrametric) Polish spaces.
%It is proved in \cite[Theorem 3]{Andretta:2005} under the assumption that the real line is well-ordered that there are antichains of size \( %2^{\aleph_0} \) in the Wadge order on \( [0,1] \). 
%Moreover, the existence of infinite chains and antichains in the Wadge order of $\mathbb{R}$ and $[0,1]$ is proved in \cite{Ikegami:2012} without %the axiom of choice. 
%However, the construction of infinite chains cannot be generalized to arbitrary uncountable Polish spaces even with the axiom of choice, since for some such spaces $X$, there is no chain of length $3$ in the Wadge order on $X$ \cite[Section 5.1]{MottoRos:2012b}. 
%
%Therefore let us consider uncountable ultrametric Polish spaces. For these we prove that the $\mathsf{L}(d)$-hierarchy is very bad under the axiom of choice. 
%We further show that the $\mathsf{Bor}(X)$-hierarchy and many other hierarchies induced by natural reducibilities on arbitrary uncountable Polish spaces is very bad under the axiom of choice. 
%

\begin{notation}
If $X$ is a set and $A\subseteq X^2$, we denote by \( A_x \) the ``vertical section'' determined by \( x \in X \), i.e.\ we set  $A_x=\{y\in X\mid (x,y)\in A\}$. Moreover, for every cardinal \( \mu \) we set $[X]^{\mu}=\{Y\subseteq X \mid |Y|=\mu\}$. 
\end{notation}

\begin{lemma} [\(\AC \)] \label{lem:choice}
Let \(\mu\) be an infinite cardinal and $X$ be a set of size $\mu$. Moreover, let $\mathcal{C} \subseteq [X]^{\mu}$, $\F$ be a collection of functions from \( X \) to itself, and suppose that \( | \mathcal{C} | = | \F | = \mu \). Then there is a set $A\subseteq X^2$ such that $A_{x}\cap C \nleq_{\F} A_{y}$ for all distinct $x,y \in X$ and all $C\in\mathcal{C}$. 
\end{lemma} 

\begin {proof} 
We first recursively construct a sequence  $ (\{ A_{x,\alpha}, B_{x,\alpha} \mid \ x\in X \})_{\alpha<\mu}$ such that $A_{x,\alpha}\cap B_{x,\alpha}=\emptyset$, 
 $A_{x,\alpha}\subseteq A_{x,\beta}$, 
 $B_{x,\alpha}\subseteq B_{x,\beta}$, and 
$|A_{x,\alpha}\cup B_{x,\alpha}|\leq |2 \cdot\alpha|$ 
for all $\alpha \leq \beta<\mu$ and $x\in X$. %The construction is obtained by simultaneously defining all sets in the subfamily \( \{ A_{x,\alpha}, B_{x,\alpha} \mid x \in X \} \) recursively on \(\alpha < \mu \).

Fix a surjection $h\colon \mu\rightarrow \mathcal{C}\times \mathcal{F}\times X^2$, and
set $A_{x,0}=B_{x,0}=\emptyset$ for all $x\in X$. Let now \( 0 < \alpha < \mu \), and assume that all sets of the form \( A_{x,\beta},B_{x,\beta} \) for \( x \in X \) and \( \beta < \alpha \) have already been defined, so that we can set $A_{x,<\alpha}=\bigcup_{\beta<\alpha} A_{x,\beta}$ and $B_{x,<\alpha}=\bigcup_{\beta<\alpha} B_{x,\beta}$. Let \( (C,f,x,y) \in \mathcal{C} \times \F \times X^2 \) be such that $h(\alpha)=(C,f,x,y)$, and let $C_0=C\setminus (A_{x,<\alpha}\cup B_{x,<\alpha})$. Notice that $| C_0| = \mu$ because $|A_{x,<\alpha}\cup B_{x,<\alpha} | < \mu \) and $|C| = \mu$. 
We distinguish two cases: if $|f(C_0)|<\mu$, we choose distinct $a,b\in C_0$ such that $f(a)=f(b)$ (this is possible because $| C_0| = \mu > |f(C_0)|$), and then we set $A_{x,\alpha}= A_{x,<\alpha}\cup \{a\}$, $B_{x,\alpha}= B_{x,<\alpha}\cup \{b\}$, and \( A_{z,\alpha} = A_{z, <\alpha} \), \( B_{z, \alpha} = B_{z,<\alpha} \) for all \( z \in X \) distinct from \( x \). 
If instead $|f(C_0)|=\mu$, we pick some $a\in C_0$ with $f(a)\notin A_{y,<\alpha}\cup B_{y,<\alpha}$ (which exists because $|A_{y,<\alpha}\cup B_{y,<\alpha}| < \mu$, and hence $f(C_0) \setminus (A_{y,<\alpha}\cup B_{y,<\alpha}) \neq \emptyset$), and then we set $A_{x,\alpha}=A_{x,<\alpha}\cup\{a\}$, $B_{x,\alpha} = B_{x,<\alpha} $, \( A_{y,\alpha} = A_{y,<\alpha} \), $B_{y,\alpha}=B_{y,<\alpha}\cup\{f(a)\}$, and \( A_{z,\alpha} = A_{z,<\alpha} \), \( B_{z,\alpha} = B_{z,<\alpha} \) for all \( z \in X \) distinct from \( x \) and \( y \). This completes the recursive step of our construction, and it is easy to check by induction on \( \alpha < \mu \) that the sets \( A_{x,\alpha} \), \( B_{x,\alpha} \) are as required.

Finally, we set $A_x=\bigcup_{\alpha<\mu} A_{x,\alpha}$, \( B_x = \bigcup_{\alpha < \mu} B_{x,\alpha} \), and \( A = \{ (x,y) \in X^2 \mid y \in A_x \} \), so that, in particular, \( A_x \cap B_x = \emptyset \) for every \( x \in X \). It is straightforward to check that  the $\alpha$-th step in the recursive construction above ensures that \( f \) is not a reduction of $A_x\cap C$ to $A_y$, because either there are \( a \in A_x\cap C \) and \( b \in B_x \subseteq X \setminus A_x \) such that \( f(a) = f(b) \), or else there is \( a \in A_x\cap C \) such that \( f(a) \in B_y \subseteq X \setminus A_y \). 
\end{proof} 

\begin{theorem}[\(\AC \)] \label{th:illfounded hierarchy}
Let $X = (X,d)$ be an uncountable ultrametric Polish space. Then there is a 
map \( \psi \colon \pow(\omega) \to \pow(X) \) such that for all \( a,b \subseteq \omega \)
\begin{enumerate}
\item
if \( a \subseteq b \), then \( \psi(a) \leq_{\L(d)} \psi(b) \);
\item
if \( \psi(a) \leq_{\Bor(X)} \psi(b) \), then \( a \subseteq b \).
\end{enumerate}
In particular, \( (\pow(\omega), \subseteq) \) embeds into the \( \F \)-hierarchy  on \( X \) for every reducibility \( \L(d) \subseteq \F \subseteq \Bor(X) \), hence \( \Deg(\F) \) is very bad.
%If moreover \( \mathsf{V=L} \), then the map \( \psi \) can be assumed to range into the \( \boldsymbol{\Sigma}^1_2 \) subsets of \( X \).
% $(P(\omega),\subseteq, \not\subseteq)\rightarrow (P(X),\leq_{\mathsf{L}}, \not\leq_{\mathsf{B}})$.  
\end{theorem} 

\begin{proof} 
We apply the Lemma~\ref{lem:choice} letting \( \mu = |X| = 2^{\aleph_0} \), \( \mathcal{C} \) be the set of all \emph{uncountable} Borel subsets of $X$, and \( \F  = \Bor(X) \) be the collection of all Borel functions from \( X \) to itself. Thus we obtain a sequence of $\leq_{\Bor(X)}$-incomparable sets $A_n\subseteq X$ (the lemma gives more, but an $\omega$-sequence is sufficient here). Notice that each \( A_n \) is necessarily uncountable and that \( A_n \neq X \), as otherwise in both cases we would easily have \( A_n \leq_{\Bor(X)} A_m \) for every \( m \in \omega \). Now choose a sequence $(X_n)_{n\in\omega}$ of pairwise disjoint uncountable clopen balls in $X$, and fix a Borel isomorphism $h_n\colon X \to X_n$ for every \( n \in \omega \). Given \( a \subseteq \omega \), set $\psi(a)=\bigcup_{n \in a} h_n(A_n)$.

To see that $\psi$ is as required, first suppose that \( a,b \subseteq \omega \) are such that $a \subseteq b$, and for every $n\in b \setminus a$ pick a point \( y_n \in X_n \setminus h_n(A_n) \) (which exists because \( A_n\neq X \)). Then we define \( f \colon X \to X \) by setting
\[ 
f(x) = 
\begin{cases}
y_n & \text{if }x \in X_n \text{ for some } n \in b \setminus a, \\
x & \text{otherwise}.
\end{cases}%
 \] 
Clearly \( f \) reduces \( \psi(a) \) to \( \psi(b) \), and it is easy to check that since \(d \) is an ultrametric and the \( X_n \) are (cl)open balls, then \( f \in \L(d) \): therefore $\psi(a)\leq_{\L(d)} \psi(b)$, as required.

Now let \(a,b \subseteq \omega \) be such that $\psi(a) \leq_{\Bor(X)} \psi(b)$, let \( f \in \Bor(X) \) be a witness of this, and fix an arbitrary $n\in a$. Notice that \( f( \psi(a)) \subseteq \psi(b) \subseteq \bigcup_{m \in b} X_m \). Since \( A_n \) is uncountable, this means that there is $m \in b$ such that $f^{-1}(X_m)\cap X_n$ is uncountable. Fix \( \bar{y} \in X \setminus A_m \): setting $C=h_n^{-1}( f^{-1}(X_m)\cap X_n)$, we get that \( C \) is an uncountable Borel set, and that the map \( g \colon X \to X \) defined by
\[ 
g(x) = 
\begin{cases}
(h^{-1}_m \circ f \circ h_n) (x) & \text{if } x \in C, \\
\bar{y} & \text{otherwise}
\end{cases}%
\]
witnesses \( A_n \cap C \leq_{\Bor(X)} A_m \). By our choice of the \( A_n \)'s, this implies \( n = m \), whence \( n \in b \). Therefore \( a \subseteq b \), as required.
\end{proof} 

%There are ${\bf\Delta}^0_{<\omega}$-isomorphisms between all uncountable Polish spaces \cite[Proposition 4.3]{MottoRos:2012b}
%and hence \todo{Modify: the construction already gives that this holds for every \( \F \supseteq \D_2 \).}

\begin{remark}
Notice that to get Lemma~\ref{lem:choice} it is enough to assume that \( X \) is a well-orderable set. Therefore, also in Theorem~\ref{th:illfounded hierarchy} we can weaken the assumption \( \AC \) by just requiring that \( X \) (equivalently, any uncountable Polish space) is well-orderable.
\end{remark}

Using essentially the same argument, one can also show that a variant of Theorem~\ref{th:illfounded hierarchy} applies to arbitrary uncountable Polish spaces \( X \) (and not only to the ultrametric ones). 
%Given such an $X$, 
%let $\mathsf{D}_2(X)$ denote the class of all \( \boldsymbol{\Delta}^0_2 \)-functions from \( X \) to itself, i.e.\ of those $f\colon X\rightarrow X$ such that $f^{-1}(A)\in {\bf \Delta}^0_2$ for all $A \subseteq X$ belonging to $ {\bf \Delta}^0_2$.

\begin{theorem}[\( \AC \)]\label{corollary: illfounded Borel hierarchy}
Let \( X \) be an uncountable Polish space. Then there is a map \( \psi \colon \pow(\omega) \to \pow(X) \) such that for every \( a,b \subseteq \omega \)
\begin{enumerate}
\item
if \( a \subseteq b \), then \( \psi(a) \leq_{\mathsf{D}_2(X)} \psi(b) \);
\item
if \( \psi(a) \leq_{\Bor(X)} \psi(b) \), then \( a \subseteq b \).
\end{enumerate}
In particular, \( (\pow(\omega), \subseteq) \) embeds into the \( \F \)-hierarchy on \( X \) for every reducibility \( \mathsf{D}_2(X) \subseteq \F \subseteq \Bor(X) \), hence \( \Deg(\F) \) is very bad.
\end{theorem} 

\begin{proof} 
In the proof of Theorem \ref{th:illfounded hierarchy}, let $(X_n)_{n\in\omega}$ be a partition of $X$ into uncountable ${\bf \Delta}^0_2$ sets. 
\end{proof} 

\begin{remark}
In Theorem~\ref{corollary: illfounded Borel hierarchy} we cannot replace \( \leq_{\D_2(X)} \) with continuous reducibility \( \leq_{\W(X)} \): in fact, in the Cook continuum \( X \) (which is uncountable), we cannot hope to embed \( (\pow(\omega), \subseteq) \) into \( \Deg(\W(X)) \) because there are no infinite chains of subsets of \( X \) (with respect to continuous reducibility).
\end{remark}

We now aim to show that if we further assume \( \mathsf{V=L} \), then the map $\psi$ of Theorems \ref{th:illfounded hierarchy} and \ref{corollary: illfounded Borel hierarchy} can be chosen to range in the collection of ${\bf \Pi}^1_1$ (alternatively: \( \boldsymbol{\Sigma}^1_1 \)) subsets of the given (ultrametric) Polish space: this in particular implies that the $\mathsf{L}(\bar{d})$-hierarchy on ${\bf \Pi}^1_1$ (respectively, \(\boldsymbol{\Sigma}^1_1 \)) subsets of $\pre{\omega}{\omega}$ is very bad in $\mathsf{L}$. 
To prove this, we will modify the recursion used in the proof of Lemma~\ref{lem:choice} so that membership in each of the sets can be computed in the next admissible set. 

\begin{notation}
For \( x,y \in \pre{\omega}{\omega} \), let $\omega_1^{x,y}$ denote the least $(x,y)$-admissible ordinal $\gamma$.%
\footnote{That is, \( \omega_1^{x,y} \) is the least $\gamma>\omega$ such that $L_{\gamma}[x,y]$ is a model of Kripke-Platek set theory.}
%We will use the following theorem. 
To simplify the notation, set also \( \omega_1^x = \omega_1^{x,x} \).
\end{notation}

\begin{theorem}[Spector-Gandy] (see \cite[Theorem 5.5]{Hjorth10}) \label{th:Spector-Gandy}
A set $A\subseteq{}^{\omega}\omega$ is $\Pi^1_1$ in a parameter $p\in {}^{\omega}\omega$ if and only if there is a $\Sigma_1$-formula $\varphi(x)$ such that 
$$x\in A\Leftrightarrow L_{\omega_1^{x,p}}[x,p]\vDash \varphi(x,p)$$
for all $x\in {}^{\omega}\omega$. 
\end{theorem}

\begin{lemma}\label{lemma:codes} 
Let $X$ be a Polish space. Then there is a set $G\subseteq {}^{\omega}\omega\times X^2$ such that: 
\begin{enumerate} 
\item A set $F\subseteq X^2$ is the graph of a Borel function from $X$ to itself if and only if $F=G_x=\{(y,z)\in X^2\mid (x,y,z)\in G\}$ for some $x\in \p(G)$. 
\item The projection $\p(G)$ on the first coordinate is a ${\bf \Pi}^1_1$ set. 
\item $G$ is both ${\bf \Pi}^1_1$ and ${\bf \Sigma}^1_1$ on $\p(G)\times X^2$. 
\end{enumerate} 
\end{lemma} 

\begin{proof}[Sketch of proof] 
Notice that we can concentrate only on ultrametric Polish spaces \( X = (X,d) \), because the result can then be transferred to an arbitrary Polish space \( Y \) by using a Borel isomorphism between \( X \) and \( Y \). Therefore from now on we fix
 an ultrametric
%\todo{@Luca: Does this work for all Polish spaces? \\ \ \\
%@Philipp: In general this is not true: the closure under limits of the continuous functions from a connected space X to a totally disconnected space Y is still the collection of continuous (i.e. constant) functions, not the collection of the Borel ones. However, I don't have a counterexample for \( X = Y \).}
 Polish space \( X = (X,d) \). Recall that \( \Bor(X) \) coincides with the collection of all Baire class $\alpha$ functions (for arbitrary $\alpha<\omega_1$), i.e.\ with the closure under pointwise limits of the collection of all Lipschitz functions (see~\cite[Corollary 2.16]{MottoRos:2009b} and~\cite[Theorems 24.3]{Kechris:1995}). Starting with a function $f$ defined on a fixed countable dense set $D\subseteq X$, we form the (pseudo-)limit $\bar{f}$ of \( f \) by setting $\bar{f}(x)=\lim_{n\rightarrow \infty} f(x_n)$ (for an arbitrary sequence $(x_n)_{n\in\omega}$ in $D$ with $\lim_{n\rightarrow \infty} x_n=x$) if 
\[
\mathrm{osc}_f(x)=\lim_{n\rightarrow \infty} \sup \{ d(f(y), f(z)) \mid y,z \in X \wedge d(x,y), d(x,z)<2^{-n}  \}=0,
\]
and \( \bar{f}(x) = y_0 \) (for \( y_0 \in X \) a fixed value) otherwise. From a countable family of functions \( f \) as above attached to the terminal nodes of a given well-founded tree, we can then build up a Borel function \( g \) by forming (pseudo-)limits (i.e.\ taking the pointwise limit where it exists and some fixed value \( y_0 \in X \) elsewhere) in the obvious way along the tree. The tree is then coded into an element of $x \in {}^{\omega}\omega$, and for all \( x \)'s built in this way we let \( G_x \) be the graph of the corresponding Borel function \( g \). 
Notice that the set of codes is ${\bf \Pi}^1_1$ because of the condition that the trees used in the coding are well-founded. 
\end{proof} 

\begin{theorem}\label{th:illfounded hierarchy in L}
Assume $\mathsf{V=L}$ and let $X = (X,d)$ be an uncountable ultrametric Polish space. Then there is a 
map \( \psi \) from \( \pow(\omega) \) into the \( {\bf\Pi}^1_1 \) subsets of \( X \) such that for all \( a,b \subseteq \omega \)
\begin{enumerate}
\item
if \( a \subseteq b \), then \( \psi(a) \leq_{\L(d)} \psi(b) \);
\item
if \( \psi(a) \leq_{\Bor(X)} \psi(b) \), then \( a \subseteq b \).
\end{enumerate} 
In particular, \( (\pow(\omega), \subseteq) \) embeds into the \( \F \)-hierarchy on the ${\bf\Pi}^1_1$ subsets of $X$ for every reducibility \( \L(d) \subseteq \F \subseteq \Bor(X) \), hence \( \Deg_{\boldsymbol{\Pi}^1_1}(\F ) \) is very bad.
%If moreover \( \mathsf{V=L} \), then the map \( \psi \) can be assumed to range into the \( \boldsymbol{\Sigma}^1_2 \) subsets of \( X \).
% $(P(\omega),\subseteq, \not\subseteq)\rightarrow (P(X),\leq_{\mathsf{L}}, \not\leq_{\mathsf{B}})$.  
\end{theorem} 

\begin{proof} Let $\mathbf{N}_{t}=\{x\in {}^{\omega}\omega\mid t\subseteq x\}$ for $t\in {}^{<\omega}\omega$. 
%If $A\subseteq {}^{\omega}2\times {}^{\omega}2$, let $A_x=\{y\in {}^{\omega}2\mid (x,y)\in A\}$. 
We first assume that $(X,d)=({}^{\omega}\omega,\bar{d})$. 
The map \( \psi \colon \pow(\omega) \to \pow(\pre{\omega}{\omega}) \) is defined by first
constructing a $\Pi^1_1$ set $A\subseteq \pow(\omega)\times {}^{\omega}\omega$,\footnote{We freely identify each $a\in\mathcal{P}(\omega)$ with its characteristic function in ${}^{\omega}2$.  The notions of $\Pi^1_1$ subsets of $\mathcal{P}(\omega)$ and $\mathcal{P}(\omega)\times{}^{\omega}\omega$ are defined accordingly. Since the identification is computable, the Spector-Gandy Theorem \ref{th:Spector-Gandy} holds for $\Pi^1_1$ subsets of $\mathcal{P}(\omega)\times {}^{\omega}\omega$. } and then letting $\psi (a)= A_a=\{y\in {}^{\omega}\omega \mid (a,y)\in A\}$. 
%Suppose that $V=L$. Then there is an $\Pi^1_1$ set $A\subseteq {}^{\omega}2\times {}^{\omega}2$ such that for all infinite $x,y\subseteq\omega$
%\begin{enumerate} 
%\item If $x\subseteq y$, then there is a Lipschitz function $f\colon {}^{\omega}2\rightarrow {}^{\omega}2$ with $A_x=f^{-1}[A_y]$. 
%\item If $x\not\subseteq y$, then there is no Borel measurable function $f\colon {}^{\omega}2\rightarrow {}^{\omega}2$ with $A_x=f^{-1}[A_y]$. 
%\end{enumerate} 
To define the desired \( A \), we will in turn construct by recursion on $\alpha<\omega_1$  a sequence $(A_{n,\alpha}, B_{n,\alpha})_{n<\omega, \alpha<\omega_1}$ such that for all $n<\omega$ and $\alpha \leq \beta <\omega_1$
\begin{enumerate} 
\item $A_{n,\alpha}, B_{n,\alpha}\subseteq \mathbf{N}_{\langle n \rangle}$, 
\item $A_{n,\alpha}\cap B_{n,\alpha}=\emptyset$,  
\item $|A_{n,\alpha}|, |B_{n,\alpha}|<\omega_1$, and
\item $A_{n,\alpha} \subseteq A_{n,\beta}$ and $B_{n,\alpha} \subseteq B_{n,\beta}$.
\end{enumerate} 

Given \( a \subseteq \omega \), we then let for \( n < \omega \) and \( \gamma < \omega_1 \)
\begin{align*}
A_{a,\alpha}= & \bigcup\nolimits_{n\in a} A_{n,\alpha}, &  B_{a,\alpha}=& \bigcup\nolimits_{n\in a} B_{n,\alpha}, \\ 
A_{n,<\gamma}=&\bigcup\nolimits_{\alpha<\gamma} A_{n,\alpha}, &  B_{n,<\gamma}=&\bigcup\nolimits_{\alpha<\gamma} B_{n,\alpha}, \\
A_{a,<\gamma}=& \bigcup\nolimits_{\alpha<\gamma} A_{a,\alpha}, & B_{a,<\gamma}=&\bigcup\nolimits_{\alpha<\gamma} B_{a,\alpha},
\end{align*} 
and finally we set 
\begin{align*}
A_n=&\bigcup\nolimits_{\alpha<\omega_1} A_{n,\alpha}, & B_n= &\bigcup\nolimits_{\alpha<\omega_1} B_{n,\alpha}, \\
A_a=&\bigcup\nolimits_{n\in a} A_n, & B_a=&\bigcup\nolimits_{n\in a} B_n.
\end{align*} 
Notice that by (1), (2) and (4), we have that for all \( a \subseteq \omega \) and \( \gamma < \omega_1 \), \( A_{a,<\gamma} \cap B_{a,<\gamma} = \emptyset \) and \( A_a \cap B_a = \emptyset \). Finally,
to simplify the notation we will also write $s_{\gamma}=(A_{n,\alpha}, B_{n,\alpha})_{n<\omega, \alpha<\gamma}$ for $\gamma<\omega_1$. 

The construction is based on the following claims. 
%Let $\omega_{\iota}^x=\omega_{\iota}^{CK,x}$. 

\begin{claim} \label{claim: many reals with large omega1ck} 
For every $s\in L_{\omega_1}$ and $l<\omega$, there are uncountably many $x \in \mathbf{N}_{\langle l \rangle}$ with $s\in L_{\omega_1^x}$. 
\end{claim} 

\begin{proof} 
Let \( \alpha < \omega_1 \) and $x\in \mathbf{N}_{\langle l \rangle}$ be such that $s\in L_{\alpha}$ and $\omega_1^x>\alpha$. Then $\omega_1^{x,y}>\alpha$ for all $y\in {}^{\omega}\omega$, so $s\in L_{\omega_1^{x\oplus y}}$ for any $y\in {}^{\omega}\omega$ (where $x \oplus y$ is defined by $(x\oplus y)(2i)=x(i)$ and $(x\oplus y)(2i+1)=y(i)$). 
\end{proof} 

Let $A_{n,0}= B_{n,0}=\emptyset$. 
%If $\gamma<\omega_1$ is a limit, let $A_{n,\gamma}=A_{n,< \gamma}$ and $B_{n,\gamma}= B_{n,< \gamma}$. 
Let now $\gamma > 0$, and suppose that the $\gamma^{\mathrm{th}}$ element in $<_L$ is of the form $(c,a,b,l)$, where $c\in {}^{\omega}\omega$ is a code for a Borel measurable function $f\colon {}^{\omega}\omega\rightarrow{}^{\omega}\omega$ as in Lemma~\ref{lemma:codes}, $a,b\subseteq\omega$, and $l\in a\setminus b$ (if this is not the case, we simply let $A_{n,\gamma }=A_{n,<\gamma}$ and $B_{n,\gamma}=B_{n,<\gamma}$ for all $n\in\omega$). 

\begin{claim}\label{claim: extension in successor step} 
There is $x\in \mathbf{N}_{\langle l \rangle}$ such that $s_{\gamma}\in L_{\omega_1^x}[x]$ and $x\notin A_{l,<\gamma}\cup B_{l,<\gamma}$.
%, and one of the following properties holds:\todo{The statement is a little bit strange because it is obvious that one of three properties must be satisfied...}
%\begin{enumerate} 
%\item $i=0$ and $f(a)\in A_{y,<\gamma}$, 
%\item $i=1$ and $f(a)\in B_{y,<\gamma}$, or 
%\item $i=2$ and $f(a)\notin A_{y,<\gamma}\cup B_{y,<\gamma}$. 
%\end{enumerate} 
\end{claim} 

\begin{proof} 
Let 
$W \subseteq {}^{\omega}\omega$ denote the set of $x\in {}^{\omega}\omega$ with $s_{\gamma}\in L_{\omega_1^x}[x]$. Since $W \cap \mathbf{N}_{\langle l \rangle}$ is uncountable by Claim~\ref{claim: many reals with large omega1ck} and $A_{l,<\gamma}\cup B_{l,<\gamma}$ is countable by (3), there is $x\in (W\cap \mathbf{N}_{\langle l \rangle}) \setminus (A_{l,<\gamma}\cup B_{l,<\gamma})$, as required.
\end{proof} 

Let $\bar{x} \in {}^{\omega}\omega$ denote the $<_L$-least element satisfying Claim~\ref{claim: extension in successor step}. 
If $f(\bar{x}) \in A_{b,<\gamma}$, let $B_{l,\gamma}=B_{l,<\gamma}\cup\{ \bar{x} \}$. 
If $f(\bar{x}) \in B_{b, < \gamma}$, let $A_{l,\gamma}= A_{l,<\gamma}\cup \{ \bar{x} \}$. 
Finally, if $f(\bar{x}) \notin A_{b,<\gamma}\cup B_{b,<\gamma}$ and $f(\bar{x})(0) = m$, let $A_{l,\gamma}=A_{l,<\gamma}\cup\{ \bar{x} \}$, and if additionally $m\in b$, then let $B_{m,\gamma}=B_{m,<\gamma}\cup \{f(\bar{x})\}$. 
For all remaining $A_{n,\gamma}$'s and $B_{n,\gamma}$'s, let $A_{n,\gamma}=A_{n,<\gamma}$ and $B_{n,\gamma}=B_{n,<\gamma}$. 
Note that if $f(\bar{x}) \notin A_{b,<\gamma}\cup B_{b,<\gamma}$ and $m\in b$, then %$f(\bar{x})\notin A_{b,<\gamma}\cup B_{b,<\gamma}$ and hence 
$f(\bar{x})\notin A_{m,<\gamma}$. 
This completes the construction of the \( A_{n,\alpha} \)'s and \( B_{n, \alpha} \)'s, and hence also of the sets \( A_a \) and \( B_a \) for \( a \subseteq \omega \). 

\begin{claim} 
If $a \subseteq b \subseteq \omega$, then $A_a\leq_{\mathsf{L}(\bar{d})} A_b$. 
%n there is a Lipschitz function $f\colon {}^{\omega}\omega\rightarrow {}^{\omega}\omega$ with $A_x=f^{-1}[A_y]$. 
\end{claim} 

\begin{proof} 
Let $f(z)=z$ for all $z\in \mathbf{N}_{\langle l \rangle}$ with $l\in a\cup (\omega\setminus b)$, while for  $z\in \mathbf{N}_{\langle l \rangle}$ with $l\in b\setminus a$, fix $z_0\notin A_b$ and let $f(z)=z_0$. Clearly \( f \in \L(\bar{d}) \). The result then follows from the fact that \( A_a \cap \mathbf{N}_{\langle l \rangle} = A_l \subseteq \mathbf{N}_{\langle l \rangle} \) for \( l \in a \) and \( A_a \cap \mathbf{N}_{\langle l \rangle} = \emptyset \) otherwise (and similarly for \( a \) replaced by \( b \)). 
\end{proof} 

\begin{claim}\label{claim: no Borel reduction} 
If $a\not\subseteq b$, then $A_a\not\leq_{\mathsf{Bor}({}^{\omega}\omega)} A_b$. 
%, then there is no Borel measurable function $f\colon {}^{\omega}\omega\rightarrow {}^{\omega}\omega$ with $A_x=f^{-1}[A_y]$. 
\end{claim} 

\begin{proof} 
Suppose that $f\colon {}^{\omega}\omega\rightarrow {}^{\omega}\omega$ is a Borel measurable function with $A_a=f^{-1}(A_b)$. Fix $l\in a\setminus b$, and let \( c \) and \( \gamma \) be such that $c$ codes $f$ as in Lemma~\ref{lemma:codes} and  $(c,a,b,l)$ is the $\gamma^{\text{th}}$ element in $<_L$. 
Suppose that in step $\gamma$ of the construction, $ \bar{x} \in {}^{\omega}\omega$ is the $<_L$-least pair in Claim~\ref{claim: extension in successor step}. 
If $f(\bar{x}) \in A_{b , < \gamma}$, then $\bar{x} \in B_l\subseteq B_a$ and $f(\bar{x})\in A_b$. Then $\bar{x}\notin A_a$ because \( A_a \cap B_a = \emptyset \), and hence $A_a\neq f^{-1}(A_b)$. 
If $f(\bar{x}) \in B_{b, < \gamma} $, then $\bar{x} \in A_l\subseteq A_a$ and $f(\bar{x})\in B_b$. Then $A_a\neq f^{-1}(A_b)$ because \( A_b \cap B_b = \emptyset \). 
Finally, suppose that $f(\bar{x}) \notin A_{b, < \gamma} \cup B_{b, < \gamma}$ and $f(\bar{x})(0) = m$. Then $\bar{x} \in A_l\subseteq A_a$. If $m\notin b$, then $f(\bar{x})\notin A_b$ because \( A_b \cap \mathbf{N}_{\langle m \rangle} = \emptyset \). 
If $m\in b$, then $f(\bar{x})\in B_m \subseteq B_b$, so $f(\bar{x})\notin A_b$ by \( A_b \cap B_b = \emptyset \). So in both cases $A_a\neq f^{-1}(A_b)$. 
\end{proof} 

Note that the condition that $(c,a,b,l)$ is of the required form is $\Pi^1_1$ by Lemma~\ref{lemma:codes}. Since the truth value of $\Pi^1_1$ statements can be calculated in admissible sets by the Spector-Gandy Theorem \ref{th:Spector-Gandy}, this implies that the recursion is absolute between (and definable over) admissible sets. 

Let $A=\{ (a,y) \mid y \in A_a \} \subseteq \pow(\omega) \times \pre{\omega}{\omega}$. 

\begin{claim} 
A is $\Pi^1_1$. 
\end{claim} 

\begin{proof} 
We freely identify each \( a \subseteq \omega \) with its characteristic function, so that \( A \) can be viewed as a 
subset of \( \pre{\omega}{\omega} \times \pre{\omega}{\omega} \). It is then sufficient to show that there is a $\Sigma_1$ formula $\varphi$ such that $(x,y)\in A$ if and only if $L_{\omega_1^{x,y}}[x,y]\vDash \varphi(x,y)$ for all $x,y\in {}^{\omega}\omega$, by the Spector-Gandy Theorem \ref{th:Spector-Gandy}. 
Let $\varphi_l(y)$ state that there is some $\gamma<\omega_1$ and a sequence $s$ of length $\gamma$ constructed according to the recursion 
such that $y$ is added to $A_l$ at step $\gamma$. 
Let $\varphi(x,y) \iff \exists l \in x\ \varphi_l(y)$. 

Suppose that $(x,y)\in A$. Then $y \in A_l\subseteq A_x$ for some $l\in x$. Suppose that $y$ is added to $A_l$ at step $\gamma<\omega_1$ in the construction. Then $s_{\gamma}\in L_{\omega_1^y}[y]$ by the definition in the successor step. Then $L_{\omega_1^{x,y}}[x,y]\vDash \varphi_l(y)$, and hence $L_{\omega_1^{x,y}}[x,y]\vDash \varphi(x,y)$. 
Now suppose that $L_{\omega_1^{x,y}}[x,y]\vDash \varphi(x,y)$. Then $L_{\omega_1^{x,y}}[x,y]\vDash \varphi_l(y)$ for some $l\in x$. Since the recursion is absolute between admissible sets, this implies that $y \in A_x$ and hence $(x,y)\in A$. 
\end{proof} 

Now suppose that $(X,d)$ is an arbitrary uncountable ultrametric Polish space. Fix a sequence $(U_n)_{n\in\omega}$ of disjoint uncountable open balls. 
For each $l\in\omega$, fix a \( \boldsymbol{\Pi}^0_2 \) proper subset $C_l$ of $U_l$ homeomorphic to ${}^{\omega}\omega$, together with a homeomorphism $h_l\colon \mathbf{N}_{\langle l \rangle}\rightarrow C_l$ and a point \( x_l \in U_l \setminus C_l \). 
Let $h\colon {}^{\omega}\omega\rightarrow X \colon x \mapsto h_{x(0)}(x)$. Then $h\colon {}^{\omega}\omega\rightarrow \bigcup_{l\in\omega} C_l$ is a homeomorphism and $\bigcup_{l\in\omega} C_l$ is a Borel subset of $X$. Finally, let $\psi(a)=h(A_a)$ for $a \subseteq\omega$. 

\begin{claim}\label{claim: Lipschitz reduction for ultrametric space} 
If $a\subseteq b\subseteq\omega$, then $\psi(a)\leq_{\mathsf{L}(d)}\psi(b)$. 
%there is a Lipschitz function $f\colon X \rightarrow X$ with $\psi(x)=f^{-1}[\psi(y)]$. 
\end{claim} 

\begin{proof} 
Let $f(z)=z$ for all $z\in C_l$ with $l\in a\cup (\omega\setminus b)$ and all $z\in X\setminus \bigcup_{l\in\omega} U_l$, and let $f(z)=x_l$ for all $z\in U_l$ with $l\in b\setminus a$. Then $f\colon X\rightarrow X$ is Lipschitz and $\psi(a)=f^{-1}(\psi(b))$. 
\end{proof} 

\begin{claim} \label{claim:nonBorelreduction}
If $a\not\subseteq b$, then $\psi(a)\not\leq_{\mathsf{Bor}(X)}\psi(b)$. 
%there is no Borel measurable function $f\colon X\rightarrow X$ with $\psi(x)=f^{-1}[\psi(y)]$. 
\end{claim} 

\begin{proof} 
Suppose towards a contradiction that \( f \colon X \to X \) is a Borel map reducing \( \psi(a) \) to \(\psi(b) \). Since \( b \neq \omega \) by \( a \not\subseteq b \), there is \( y_0 \in (\bigcup_{l \in \omega} C_l ) \setminus \psi(b) \). Let \( \bar{f} \colon X \to X \) by letting \( \bar{f}(x) = f(x) \) if \( f(x) \in \bigcup_{l \in \omega} C_l \) and \( \bar{f}(x) = y_0 \) otherwise: then \( \bar{f} \colon X \to \bigcup_{l \in \omega} C_l \) is Borel and still reduces \( \psi(a) \) to \( \psi(b) \). Since the map $h\colon  {}^{\omega}\omega\rightarrow \bigcup_{l\in\omega} C_l$ is a homeomorphism, the function \( h^{-1} \circ \bar{f} \circ h \colon \pre{\omega}{\omega} \to \pre{\omega}{\omega} \) is a well-defined Borel function reducing \(A_a \) to \(A_b \), contradicting Claim~\ref{claim: no Borel reduction}. 
\end{proof} 
This completes the proof of Theorem \ref{th:illfounded hierarchy in L}. 
\end{proof}

Similar to Corollary \ref{corollary: illfounded Borel hierarchy}, we now obtain: 

\begin{theorem}\label{corollary: illfounded Borel hierarchy in L}
Assume $\mathsf{V=L}$ and let \( X  \) be an uncountable Polish space. Then there is a map \( \psi \) from \( \pow(\omega) \) into the \( {\bf\Pi}^1_1\) subsets of \( X \) such that for every \( a,b \subseteq \omega \)
\begin{enumerate}
\item
if \( a \subseteq b \), then \( \psi(a) \leq_{\mathsf{D}_2(X)} \psi(b) \);
\item
if \( \psi(a) \leq_{\Bor(X)} \psi(b) \), then \( a \subseteq b \).
\end{enumerate} 
In particular, \( (\pow(\omega), \subseteq) \) embeds into the \( \F \)-hierarchy on the ${\bf\Pi}^1_1$ subsets of $X$ for every reducibility \( \mathsf{D}_2(X) \subseteq \F \subseteq \Bor(X) \), hence \( \Deg_{\boldsymbol{\Pi}^1_1}(\F) \) is very bad.
\end{theorem} 

\begin{proof}
Define \(\psi\) as at the end of the proof of Theorem~\ref{th:illfounded hierarchy in L} (where the case of an arbitrary ultrametric Polish space is considered). If \( a \subseteq b \subseteq \omega \), we define \( f \colon X \to X \) as in the proof of Claim~\ref{claim: Lipschitz reduction for ultrametric space}: then \( f \) is clearly $\mathsf{D}_2(X)$ and $\psi(a)=f^{-1}(\psi(b))$. The converse direction can be easily proved as in Claim~\ref{claim:nonBorelreduction}, hence we are done.
\end{proof} 

The existence of maps \( \widetilde{\psi} \colon \pow(\omega) \to {\bf\Sigma}^1_1(X) \) with the properties stated in Theorems~\ref{th:illfounded hierarchy in L} and~\ref{corollary: illfounded Borel hierarchy in L} follows immediately by taking complements, i.e.\ by setting \( \widetilde{\psi}(a) = X \setminus \psi(a) \) for every \( a \subseteq \omega \) (where \( \psi \colon \pow(\omega) \to \boldsymbol{\Pi}^1_1(X) \) is as in Theorem~\ref{th:illfounded hierarchy in L} or Theorem~\ref{corollary: illfounded Borel hierarchy in L}). 

\begin{remark}
By Borel determinacy, the requirement that \( \psi \) ranges into \( \boldsymbol{\Pi}^1_1 \) (alternatively: \( \boldsymbol{\Sigma}^1_1 \)) subsets of \( X \) in Theorems~\ref{th:illfounded hierarchy in L} and~\ref{corollary: illfounded Borel hierarchy in L} cannot be further improved, and therefore such results are optimal.
\end{remark}

It is well-known that ${\bf \Pi}^1_1$-determinacy implies that e.g.\ the $\mathsf{L}(\bar{d})$-hierarchy on  ${\bf \Pi}^1_1$ subsets of \( \pre{\omega}{\omega} \) is very good. In fact, Harrington~\cite{Harrington:1978} (essentially) showed that the following are equivalent:
\begin{itemize}
\item
every \( \boldsymbol{\Pi}^1_1 \) subset of \( \pre{\omega}{\omega} \) is determined;
\item
for all \( x \in \pre{\omega}{\omega} \), \( x^\# \) exists;
\item
\( \SLO^{\L(\bar{d})} \) holds for \( \boldsymbol{\Pi}^1_1 \) subsets of \( \pre{\omega}{\omega} \).
\end{itemize}
Since sharps do not exist  if \( \mathsf{V = L} \), Theorem~\ref{th:illfounded hierarchy in L} can then be regarded as a strengthening of (one direction)
 of the above mentioned Harrington's result: under the further assumption \( \mathsf{V=L} \), not only \( \SLO^{\L(\bar{d})} \) for
 \( \boldsymbol{\Pi}^1_1 \) subsets of \( \pre{\omega}{\omega} \) does not hold, but in fact we can embed a reasonably complicated partial order in
  \( \Deg_{\boldsymbol{\Pi}^1_1}(\L(\bar{d})) \).  Notice also that since
\( \Deg_{\boldsymbol{\Delta}^1_1}(\L(\bar{d})) \) needs to be very good by Borel determinacy, Theorem~\ref{th:illfounded hierarchy in L} actually shows that if \( \mathsf{V=L} \), then 
\( (\pow(\omega), \subseteq) \) embeds into the \( \L(\bar{d}) \)-hierarchy on \emph{proper} \( \boldsymbol{\Pi}^1_1 \)  subsets of
 \( \pre{\omega}{\omega} \), and Theorem~\ref{corollary: illfounded Borel hierarchy in L} shows that the same partial order embeds also e.g.\ in the \( \Bor(X) \)-hierarchy on \emph{proper} \( \boldsymbol{\Pi}^1_1 \) (alternatively:
 \emph{proper} \( \boldsymbol{\Sigma}^1_1 \)) subsets of any uncountable Polish space \( X\). This conclusion considerably strengthen the well-known
 fact that if \( \boldsymbol{\Pi}^1_1 \)-determinacy fails then there are proper \( \boldsymbol{\Pi}^1_1 \) subsets which are not (Borel-)complete for that
 class.

The next questions essentially asks if it is possible to further strengthen Theorems~\ref{th:illfounded hierarchy} and~\ref{th:illfounded hierarchy in L} by either trying to embed  a more complicated quasi-order into the relevant hierarchies, or by weakening the assumption required for those results to \( \boldsymbol{\Pi}^1_1 \)-determinacy.

\begin{question} 
Assume $\mathsf{AC}$.  
\begin{enumerate}
\item Is there a map \( \psi \colon \pow(\omega) \to \pow (\pre{\omega}{\omega}) \) such that $a\subseteq^* b \iff \psi(a)\leq_{\mathsf{Bor}(\pre{\omega}{\omega})}\psi(b)$ for all \( a,b \subseteq \omega \)? 
\item Does the non-existence of $0^{\#}$ %for all \( a \in \pre{\omega}{\omega} \) 
already imply that the $\mathsf{Bor}(\pre{\omega}{\omega})$-hierarchy on ${\bf \Pi}^1_1$ subsets of \( \pre{\omega}{\omega} \) is ill-founded? 
\end{enumerate} 
\end{question}

%\begin{question} Suppose $X$ is an ultrametric Polish space and $V=M_n$. Then there is a $\Delta^1_{n+2}$ wellorder of $P(\omega)$ and $\Sigma^1_{n+1}$ sets have the perfect set property. Is there a homomorphism $(P(\omega),\subseteq,\not\subseteq)\rightarrow (\Sigma^1_{n+2}(X),\leq_{\mathsf{L}},\not\leq_{\Sigma^1_{n+1}})$? 
%\end{question} 

%\bibliographystyle{alpha}
%\nocite{*}
%\bibliography{BibliographyWadge-likeReducibilitiesOnUltrametricPolishSpaces}

\end{document}